\documentclass[11pt,oneside]{article}

\usepackage{amsmath}
\usepackage{amssymb}
\usepackage{amsfonts}
\usepackage{mathrsfs}
\usepackage{graphicx}
\usepackage{epsfig}
\usepackage{accents}
\usepackage{geometry}
\usepackage[pagewise]{lineno}
\usepackage{amscd,amsmath,amssymb,amsthm,amsfonts,epsfig,graphics}
\usepackage{fancyhdr}
\usepackage{tikz}
\usepackage{pgf,pgfplots}
\usepackage{caption}
\usepackage{titlesec}
\usepackage{setspace}

\usepackage{indentfirst}
\usepackage[skins]{tcolorbox}
\tcbuselibrary{skins, breakable}

\usepackage{abstract}

\usepackage[numbers, sort]{natbib}
\newcommand \footnoteONLYtext[1]
{
\let \mybackup \thefootnote
\let \thefootnote \relax
\footnotetext{#1}
\let \thefootnote \mybackup
\let \mybackup \imareallyundefinedcommand
}

\usepackage{enumitem}
\setenumerate[1]{itemsep=-4pt, partopsep=0pt, parsep=\parskip, topsep=5pt}
\setitemize[1]{itemsep=-4pt, partopsep=0pt, parsep=\parskip, topsep=5pt}
\setdescription{itemsep=-4pt, partopsep=0pt, parsep=\parskip, topsep=5pt}

\allowdisplaybreaks

\newtheorem{definition}{Definition}[section]
\newtheorem*{claim}{Claim}
\theoremstyle{plain}
\newtheorem{thm}{Theorem}
\newtheorem{lem}{Lemma}[section]

\newtheorem{fact}{Fact}[section]
\newtheorem{prop}{Proposition}
\newtheorem{remark}{Remark}

\newtheorem*{nonametheorem}{Main Theorem}

\newtheorem*{thmR}{Reduced Main Theorem}

\numberwithin{equation}{section}
\numberwithin{figure}{section}
\numberwithin{table}{section}

\newcommand\keywords[1]{\it{Keywords}: #1}
\newcommand\msc[1]{\it{2010 Mathematics Subject Classification}:#1}

\makeatletter
\renewcommand{\section}{\@startsection{section}{1}{0mm}
{-\baselineskip}{0.5\baselineskip}{\Large\bf\leftline}}
\makeatother

\makeatletter
\renewcommand{\subsection}{\@startsection{subsection}{1}{0mm}
{-\baselineskip}{0.5\baselineskip}{\large\bf\leftline}}
\makeatother

\title{\Large{\bf{{Stochastic stability for weakly hyperbolic contracting Lorenz maps}}}}

\author{{Haoyang Ji }
\\
}

\date{}

\begin{document}

\maketitle

\vspace{-2cm}

\begin{abstract}
\noindent{ \bf{Abstract.}} In this article we study the expanding properties of random perturbations of contracting Lorenz maps satisfying the summability condition of exponent 1. Under general conditions on the maps and perturbation types, we prove stochastic stability in the strong sense: convergence of the densities of the stationary measures to the density of the physical measure of the unperturbed map in the $L^1$-norm. This improves the main result in \cite{Me}.

\end{abstract}

\footnoteONLYtext{Date: \date{\today}}
\footnoteONLYtext{\msc{\rm{ 37E05} \rm{37H30}}}
\footnoteONLYtext{\keywords{\rm{one dimensional dynamics, Lorenz map, physical measure, stochastic stability}}}
\footnoteONLYtext{H. Ji was supported by NSFC Grant No.12301103.}
\section{Introduction}
\label{intro}

Lorenz flows are related to the systems numerically studied by Lorenz in \cite{Lo} originated by truncating  Navier-Stokes equations for modeling atmospheric conditions. This system exhibits the famous strange Lorenz attractor and has played an important role in the development of the subject of dynamical systems. The existence of a strange attractor for classic Lorenz flows was listed by Steven Smale as one of several challenging problems for the twenty-first century, and was proved by Tucker in \cite{Tu}. Guckenheimer and Williams \cite{GW}, and also  Afra$\rm{\breve{i}}$movi$\rm{\check{c}}$-Bykov-Shilnikov \cite{ABS}, introduced the geometric Lorenz flows in which it was supposed that the eigenvalues $\lambda_2< \lambda_1 < 0 < \lambda_3$ at the singularity of the flow satisfying the expanding condition $\lambda_1 + \lambda_3 >0$. In \cite{ACT} Arneodo, Coullet and Tresser began to study a model obtained in the same way just replacing the expanding condition by the contracting condition $\lambda_1 + \lambda_3 <0$. The general assumptions used to construct the geometric models also permit the reduction of the 3-dimension problem, first to a 2-dimensional Poincar\'e section and then to a one-dimensional map, the so-called Lorenz maps.

From a topological viewpoint, a Lorenz map $f : I \setminus \{ c\} \to I$ is nothing else than an interval map with two monotone branches and a discontinuity $c$ in between. On both one-sided neighborhoods of the discontinuity the Lorenz map equals $|x|^{\alpha}$ near the origin up to coordinate changes. The parameter $\alpha>0$ is the {\it critical exponent} which by construction equals the ratio of the absolute value between the stable and unstable eigenvalues. If $\alpha < 1$, then the derivative of $f$ at $c$ is infinite. Such maps are typically overall expanding and chaotic, and by this reason these maps are called {\it expanding Lorenz maps}. Since $\alpha<1$ holds in the situation of the classical Lorenz systems, expanding Lorenz maps has been studied widely and their dynamics is well understood. If $\alpha>1$, then $f$ is called {\it contracting Lorenz maps}. This case is significantly harder due to the interplay between contraction near the discontinuity and expansion outside.

The dynamics of smooth interval maps has been studied exhaustively in the last forty years, especially for unimodal maps. Critical points and critical values play fundamental roles in the study of interval dynamics. From this point of view, Lorenz maps are of hybrid type: these maps have a single critical point as unimodal maps, but two critical values as bimodal maps. The presence of both contraction and discontinuity means that many techniques from the theory of expanding maps and one-dimensional maps are not applicable. However the starting points should still be the refined theory of smooth one-dimensional dynamics, especially of  unimodal maps. The symbolic and topological dynamics of such Lorenz maps have been widely studied, see \cite{B, KP}. The measurable dynamics was studied previously in \cite{AS, CD, KP, Ro, Me1} among others. The first step towards a theory of Lorenz renormalization was taken by Martens and de Melo \cite{MM} who developed a combinatorial counterpart of unimodal renormalization. Further study in this direction can be found in \cite{GG,  MW} among others.

In this paper we study random perturbations of contracting Lorenz maps under weakly hyperbolic assumption, with a main focus on stochastic stability. We shall study composition of maps of the form $f_{t_{n-1}} \circ \cdots \circ f_{t_1} \circ f_{t_0}$ where $f_{t_0}, f_{t_1}, \cdots$ are independently chosen random maps from a one-parameter continuous family perturbed from a weakly hyperbolic contracting Lorenz maps $f$. We prove stochastic stability: a typical random orbit  $f_{t_{n-1}} \circ \cdots \circ f_{t_1} \circ f_{t_0}(x)$ has roughly the same asymptotic distribution in the phase space as a typical orbit of the unperturbed map $f$ in a strong sense. For precise description, see subsection 2.2.

Stochastic stability of dynamical systems was introduced by Kolmogrov and Sinai. It is natural in consideration that any system arising from real world is unavoidably affected by external noises. An extensive historical account on stochastic stability of dynamical systems can be found in \cite{BDV} or \cite{V}. Uniformly expanding maps and uniformly hyperbolic systems are known to be stochastically stable \cite{Ki1}. For non-uniformly expanding interval maps which satisfy a condition of Benedicks-Carleson type, stochastic stability was previously studied in \cite{BV, BY, Ts}. These systems are assumed to exhibit expansion away from a critical region with slow recurrence rate to it and hence admit absolutely continuous invariant measure. In \cite{S} Shen proved strong stochastic stability for interval maps under a much weaker non-uniformly assumption and more general perturbation types. Even for unimodal maps with a wild attractor, stochastic stability was proved in the weak sense in \cite{LW}. For stochastic stability in other direction, see \cite{BV} for H\'enon-like maps, \cite{AA, AV} for multidimensional local diffeomorphisms, and \cite{SvS} for intermittent circle maps with a neutral fixed point. In recent years, there is also an increasing interest in the study of statistical properties of random systems, including quenched (path-wise) decay of correlation. See \cite{BBM} or \cite{D} for non-uniformly expanding unimodal maps. Inducing schemes are powerful tools in these research.

For contracting Lorenz maps, Metzger \cite{Me} used methods and strategy in \cite{BV} to prove strong stochastic stability for Rovella-like maps. Note that the Rovella-like condition is also a kind of Benedicks-Carleson type condition. For infinitely renormalizable contracting Lorenz maps with {\it a priori bounds}, stochastic stability was proved by Wang and the author in \cite{JW}. In \cite{LR}, quenched exponential decay of correlations for random Rovella-like maps was studied. The main goal of the present work is to improve the result in \cite{Me} to more general conditions. The non-uniformly expanding condition is significantly weaker, and the perturbation types allowed here are also more general. In particular, no recurrence condition is imposed. The Main Theorem will be proved using an inducing scheme borrowed from \cite{S}. The random inducing scheme constructed here is weaker than the random Young Tower appeared in \cite{AV} and \cite{D}, but is enough to prove stochastic stability.

This paper is organized as follows. In Section 2 we present formally the main definitions and the Main Theorem. In Subsection 3.1, we state the Reduced Main Theorem which contains a form of inducing scheme, and prove the Main Theorem. In Subsection 3.2, we study the expansion results for deterministic contracting Lorenz maps under large derivatives condition. The backward contraction property appeared in \cite{BRSS} plays a crucial role. Subsection 3.3 contains lemmas about binding arguments initiated in \cite{BC} and \cite{Ts}, and a stochastic version of Ma${\rm \tilde{n}}$\'e's Theorem (Proposition 6). Some of the results in Subsection 3.2 and 3.3 have been proved in \cite{LR} under a stronger non-uniformly expanding condition. Therefore we only provide the proof of the lemmas in case that our proofs are different from therein. In Section 4 we shall obtain the lower bounds on the growth of derivatives along random orbits which stays outside a particular neighborhood of the critical point (Theorem 2). This is based on a combination of analysis on expansion  results and binding argument. As a consequence of Theorem 2, we shall prove the first landing maps of random orbits into a suitably chosen critical neighborhood $\tilde B(\epsilon)$ has small total distortion (Proposition 9) in Subsection 4.3. Section 5 and Section 6 are the most technical parts and are devoted to the proof of the Reduced Main Theorem. In Section 5 we study the recurrence of random orbits into $\tilde B(\epsilon)$ and estimate diffeomorphic return times. The final inducing step is carried out in Section 6 (Proposition 16). To do this, we shall use the so-called {\it $\theta$-good return time}  introduced in \cite{S} instead of hyperbolic time used in \cite{AV} and \cite{LR}. The key point is that the size of the tail sets decays at a polynomial rate.

\section{Statement of results}

\subsection{Contracting Lorenz maps}

Denote $I=[0, 1]$. A piecewise $C^3$ interval map $f: I \to I$ with a discontinuity at $c \in (0, 1)$ is called a {\it Lorenz map} if $f(0) =0, f(1) =1$, $Df(x) >0$ for all $x \in I \setminus \{ c\}$. The point $c$ is called the {\it singular point} (or {\it critical point}). A Lorenz map has two critical values defined by $c_1^- = \lim_{x \to c^-}f(x)$ and $c_1^+ = \lim_{x \to c^+}f(x)$, thus implicitly thinking of $c^+$ and $c^-$ as distinct critical points.

Let $CV : = \{ c_1^+, c_1^-\}$ denote the set of critical values of $f$. Denote $I^- = [0, c)$, $I^+ = (c, 1]$ and $I_c = I \setminus \{ c\}$. When we consider the iterates of a Lorenz map $f: I \to I$, we are essentially considering the iterates of $f : I_c \to I$, just recognizing that the pre-images of the singular point $c$ are countable.

A Lorenz map is called {\it contracting} provided $Df(c^-) = Df(c^+) =0$. A contracting Lorenz map $f$, with singularity $c$, is called {\it non-flat} if there exist $u \in [0, 1], v \in [0, 1]$, $\ell > 1$ and $C^3$ diffeomorphisms $\phi : [0, c] \to [0, u^{{1}/{\ell}}]$ and $\psi : [c, 1] \to [0, v^{{1}/{\ell}}]$ such that $\phi(c) = 0 =\psi(c)$, $\phi(0) = u^{{1}/{\ell}}, \psi(1) =v^{{1}/{\ell}}$ and 
\begin{equation} f(x) = \begin{cases}
u - (\phi(x))^{\ell} & \mbox{ if } x < c\\
1 - v + (\psi(x))^{\ell} & \mbox{ if } x > c.
\end{cases}
\end{equation}
The exponent $\ell$ are referred to as the {\it critical order} of the one-sided critical points $c^-$ and $c^+ $, respectively. Note that $u$ and $1-v$ are the two critical values of $f$.

A Lorenz map is called non-trivial if $c_1^+ < c < c_1^-$. Otherwise, all points converge to some fixed point under iteration and for this reason, $f$ is called trivial. Unless otherwise noted, all Lorenz maps are assumed to be nontrivial. In general, $c_k^{\pm}$ will denote points in the orbit of the critical values:
\[
c_k^{\pm} = \lim_{x \to c\pm}f^k(c), k \geq 1.
\]

The Schwarzian derivative of a $C^3$ diffeomorphism $h : J \to h(J)$ is denoted by
\[ S h (x) = \frac{D^3 h (x)}{D h (x)} - \frac{3}{2} \left( \frac{D^2 h (x)}{D h (x)}\right)^2 ( Dh (x) \neq 0).
\]

Given a contracting Lorenz map $f : I_c  \to I$. We say that $f$ satisfies:
\begin{itemize}
\item[(1)] the {\it Large derivatives condition} (abbreviated (LD)), if for each $v \in CV$, we have
\[
\lim_{n \to \infty} Df^n(v) = \infty;
\]
\item[(2)] the {\it Summability condition of exponent 1} (abbreviated (${\rm SC_1}$)), if for each $v \in CV$, we have
\[
\sum_{n=0}^{\infty} \frac{1}{Df^n(v)} < \infty;
\]
\item[(3)] the {\it Collet-Eckmann condition} (abbreviated (CE)), for each $v \in CV$, we have
\[
\liminf_{n \to \infty} \frac{1}{n} \log D f^n (v) >0.
\]
\end{itemize}

Let $\mathcal A$ denote the collection of $C^3$ contracting Lorenz maps $f : I_c \to I$ with non-flat critical point and with the following properties:
\begin{itemize}
\item[(A1)] $f$ has no attracting or neutral periodic orbits;
\item[(A2)] $f$ has negative Schwarzian derivative;
\item[(A3)] $f$ is topologically mixing on $[c_1^+, c_1^-]$. 
\end{itemize}
Let $\mathcal S_1$ denote the collection of maps $f \in \mathcal A$ which satisfies (${\rm SC_1}$) and let $\mathcal {LD}$ denote the collection of maps $f \in \mathcal A$ which satisfies (LD). Clearly, $\mathcal S_1 \subset \mathcal{LD}$ and a map $f \in \mathcal{LD}$ has no critical relation: for any $v \in CV$ and any integer $n \geq 1$, $f^n(v) \neq c$.

The following theorem was proved by Bruin {\it et al} for multimodal maps \cite{BRSS} and by Cui and Ding for contracting Lorenz maps \cite{CD}.

\begin{thm}\cite{CD}
Let $f \in \mathcal{LD}$, then $f$ has an invariant probability $\mu$ which is absolutely continuous with respect to the Lebesgue measure (abbreviated acip) and the density of $\mu$ belongs to $L^p$ for all $p < \ell/(\ell-1)$. Moreover, $f$ admits no wandering intervals. 
\end{thm}

By condition (A3), the acip $\mu$ for $f \in \mathcal {LD}$  is ergodic, and moreover, unique. Such a measure is clearly a {\it physical measure} in the sense that its {\it basin}
\begin{equation}
B(\mu)= \{ x \in I :  \frac{1}{n} \sum_{k=1}^n \delta_{f^k(x)} \to \mu \mbox{ as } n \to \infty \mbox{ in the weak$^{\star}$ topology} \}
\end{equation}
has positive Lebesgue measure.

\subsection{Random perturbations}

To model random perturbations of a discrete-time system $f : I \to I$ we may consider sequences obtained by iteration $x_{n+1} = g_n \circ g_{n-1} \circ \cdots \circ g_0(x_0)$ of maps $g_n$ chosen at random $\epsilon$-close to $f$.

For each $k =0, 1, \cdots$, we use $\mathscr F_k$ to denote the space of all $C^k$ contracting Lorenz maps from $I$ into itself which have only hyperbolic repelling periodic points endowed with the $C^k$ metric. For $g \in \mathscr F_1$, let $c_g$ denote the critical point of $g$.

Setting $f_0 = f \in \mathcal A$. A one-parameter family $\{ f_t \}_{t \in [-1, 1]} \subset \mathscr F_1$ is called {\it admissible} if the following four conditions are satisfied.

\begin{enumerate}

\item[(C1)] The critical point $c_t$ stays fixed for all $t \in [-1, 1]$, that is, $c_t = c$ for some $c \in (0, 1)$. 

\item[(C2)] $|\partial_t \Phi(x, t)| \leq 1$ for any $x \in I_c, t \in [-1, 1]$, where $\Phi(x, t) = f_t(x) $.

\item[(C3)] There exists a constant $C > 0$ such that for any $t \in [-1, 1]$ and $x, y \in I_c$, we have
\begin{equation}
2 d(x, y) < d(x , c) \Longrightarrow \bigg| \log \frac{Df_t(x)}{Df_t(y)} \bigg| \leq C \frac{d(x, y)}{d(x, c)}.
\end{equation}
\item[(C4)] There exist real numbers $\ell >1$ and $\delta>0, O_1>0, O_2 >0$ such that for any $t \in [-1, 1]$ and whenever $x \in (c -\delta, c+\delta) \setminus \{ c\}$, we have
\[
O_1 d(x, c)^{\ell-1} \leq Df_t(x) \leq O_2  d(x, c)^{\ell-1}.
\]
\end{enumerate}

It is convenient to use condition (C1) since for any $f, g \in \mathscr F_1$ with $|| f-g||_{C^1} < \epsilon$, we must have $c_f = c_g$ provided $\epsilon$ small enough, because the critical point is a jump discontinuity. Condition (C3) and (C4) are inspired by the non-flatness of the singular point.

Denote $\Omega = [-1, 1]^{\mathbb N}$ and $\Omega_{\epsilon} = [-\epsilon, \epsilon]^{\mathbb N}$ for $\epsilon \in (0, 1]$. For any $\omega \in \Omega$, where $\omega = (\omega_0, \omega_1, \cdots, \omega_n, \cdots)$, and $n \geq 1$, write 
\begin{equation}
f_{\omega}^n = f_{\omega_{n-1}} \circ f_{\omega_{n-2}} \circ \cdots \circ f_{\omega_1} \circ f_{\omega_0}, \ f_{\omega}^0(x) = x.
\end{equation}
The corresponding $\epsilon$-{\it random orbits} can be formulated as 
\[
x_n = f_{\omega}^n(x), n \geq 0, \omega \in  \Omega_{\epsilon}.
\]
As usual, let $F : I \times \Omega \to I \times \Omega$ denote the skew-product map:
\[
(x, \omega) \to (f_{\omega}(x), \sigma \omega).
\]

For $\epsilon \in (0, 1]$,  let $\nu_{\epsilon}$ be a Borel probability measure supported in $[-\epsilon, \epsilon]$. We denote by $P_{\epsilon}$ the measure ${\rm Leb}|_I \times \nu_{\epsilon}^{\mathbb N} $, where ${\rm Leb}$ is the Lebesgue measure. This measure naturally induces a probability measure on the space of $\epsilon$-random orbits which is our reference measure.  A Borel probability measure $\mu_{\epsilon}$ is called {\it physical} for {\it $\epsilon$-perturbations} if the set of $\epsilon$-random orbits $\{ x_n \}_{n=0}^{\infty}$ with the following property has positive measure:
\[
\frac{1}{n} \sum_{i=0}^{n-1} \delta_{x_i} \to \mu_{\epsilon} \text{ as $n \to \infty$ in the weak$^{\star}$ topology}.
\]

There is an associated Markov chain, denoted by $\chi^{\epsilon}$, with state space $I$ and transition probabilities $\{p_{\epsilon}(x, \cdot) \}_{x \in I}$ defined by
\begin{equation}
p_{\epsilon} (x, A) = \nu_{\epsilon} (\{ t \in [-\epsilon, \epsilon] : f_t(x) \in A \}). 
\end{equation}
Then each $p_{\epsilon}(x, \cdot)$ is supported in the $\epsilon$-neighborhood of $f(x)$. To obtain meaningful results, we shall assume certain regularity of $\nu_{\epsilon}$. Denote a family $\mathbb P_{\epsilon} = \{p_{\epsilon}(x, \cdot) \}_{x \in I}$ of probability measure on $I$. We write $\nu_{\epsilon} \in \mathbb M_{\epsilon} (L)$ if for each $x \in I$, and each Borel set $A \subset I$, we have 
\begin{equation}
p_{\epsilon} (x, A) \leq L \left(\frac{|A|}{2 \epsilon} \right)^{\frac{1}{L}},
\end{equation}
where $|A|$ denote the Lebesgue measure of $A$ and $L>1$ is a constant.

Indeed, for each $\epsilon >0$ small, the physical measure $\mu_{\epsilon}$ is also the stationary measure for homogenous Markov chains $\chi^{\epsilon}$ with transition probabilities $p_{\epsilon}(x, \cdot)$. Recall that a probability measure $\mu_{\epsilon}$ on $I$ is called a {\it stationary measure} for $\chi^{\epsilon}$, or for $\mathbb P_{\epsilon}$, or for $\nu_{\epsilon}$, if for each Borel set $A \subset I$, we have
\begin{equation}
\mu_{\epsilon} (A) = \int_I p_{\epsilon} (x ,A)d \mu_{\epsilon}(x) = \int_{[-\epsilon, \epsilon]} \mu_{\epsilon} (f_t^{-1}(A)) d \nu_{\epsilon}(t).
\end{equation}
Stationary measures always exist, provided that the transition probabilities $p_{\epsilon}(x, \cdot)$ depend continuously on the point $x$. It is also well-known that $\mu_{\epsilon}$ is a stationary measure for $\chi^{\epsilon}$ if and only if $\mu_{\epsilon} \times \nu_{\epsilon}^{\mathbb N}$ is invariant under $F$, see for example \cite{Ki1, AA, AV}.

If $f$ has unique physical measure $\mu_f$, then we say that $f$ is {\it stochastically stale} with respect to $(\nu_{\epsilon})_{\epsilon >0}$ if for each $\epsilon >0$ small enough, there exists a unique stationary measure $\mu_{\epsilon}$ for $\nu_{\epsilon}$ and $\mu_{\epsilon} \to \mu_f$ as $\epsilon \to 0$ in the weak$^{\star}$ topology. We say that $f$ is {\it strongly stochastically stale} if $\mu_{\epsilon} \to \mu_f$ in the strong topology, i.e. if ${\rm d}_{tv}(\mu_{\epsilon}, \mu_f) \to 0$ as $\epsilon \to 0$. Here ${\rm d}_{tv}(\mu_{\epsilon}, \mu_f) = \sup_A | \mu_{\epsilon}(A) - \mu_f(A) |$ where $A$ runs over all Borel sets. If $\mu_{\epsilon}, \mu_f$ are absolutely continuous with densities $\zeta_{\epsilon}$ and $\zeta_{f}$, then strong convergence is equivalent to $|| \zeta_{\epsilon} - \zeta_f ||_1 \to 0$ as $\epsilon \to 0$ where $|| \cdot ||_1$ stands for the $L^1$ norm.

The main theorem of this article is the following.

\begin{nonametheorem}

Let $f \in \mathcal S_1$ and let $\{ f_t \}_{t \in [-1, 1]} $ be an admissible one-parameter family from $\mathscr F_1$ with $f_0 = f$. For each $\epsilon > 0$ small, let $\nu_{\epsilon}$ be a Borel probability measure on $[-\epsilon, \epsilon]$ such that $\nu_{\epsilon} \in \mathbb M_{\epsilon}(L)$ for some constant $L>1$. Then there exists $\epsilon_0>0$ such that the following holds for each $\epsilon \in (0, \epsilon_0]$:

(1) The random system $f_{\omega}$ has a unique physical measure $\mu_{\epsilon}$ for $\nu_{\epsilon}$. The physical measure $\mu_{\epsilon}$ is absolutely continuous w.r.t. the Lebesgue measure and for almost all $\epsilon$-random orbits $\{ x_n \}_{n=0}^{\infty}$,
\[
\frac{1}{n} \sum_{i=0}^{n-1} \delta_{x_i} \to \mu_{\epsilon} \text{ as $n \to \infty$ in the weak$^{\star}$ topology}.
\]

(2) The map $f$ is strongly stochastically stable.
\end{nonametheorem}

\section{Preliminaries and Reduced Main Theorem}

Denote $a \asymp b$ if there exists a constant $C \geq 1$ such that $C^{-1} b \leq a \leq C b$;  denote $a \lesssim b$ if there exists a constant $C \geq 1$ such that $ a \leq C b$; and denote $a \ll b$ if there exists a large constant $C \geq 1$ such that $a \leq C b$.

Denote by $| \cdot |$ or ${\rm Leb}( \cdot )$ the Lebesgue measure. For a subset $X$ of $I \times \Omega$, let $X^{\omega}$ denote the fiber of $X$ over $\omega$, that is, $X^{\omega} = \{ x \in I :  (x, \omega) \in X\}$.

Given a $C^1$ diffeomorphism $\varphi: J \to T$ between bounded intervals, define
\[
{\rm Dist}(\varphi|J) = \sup_{x, y \in J} \log \frac{|D\varphi(x)|}{|D\varphi(y)|}
\]
and
\[
\mathcal N(\varphi|J) = \sup_{J'} {\rm Dist}(\varphi | J') \frac{|J|}{|J'|},
\]
where the supremum is taken over all subintervals $J'$ of $J$. Note that when $\varphi$ is $C^2$, we have
\[
\mathcal N(\varphi|J) = \sup_{x \in J} \frac{|D^2\varphi(x)|}{|D\varphi(x)|} |J|.
\]

Suppose $f \in \mathcal A$. For each $\delta >0$, let
\[
\tilde B(\delta) = f^{-1} (c_1^+, c_1^+ + \delta) \cup f^{-1} (c_1^- - \delta, c_1^-),
\]
and 
\[
D(\delta) = \frac{\delta}{|\tilde B(\delta)|} \asymp \frac{\delta}{\delta^{\frac{1}{\ell}}} = \delta^{1 - \frac{1}{\ell}}.
\]
Let
\[
\hat B(\delta) = \tilde B(\delta) \cup \{ c\}
\]
which is an interval. To simplify the notation, we shall not distinguish $\tilde B(\delta)$ and $\hat B(\delta)$ in the rest of the article. So when we are talking about a diffeomorphism $g : T \to \tilde B(\delta)$, we are actually referring to $g : T \to \hat B(\delta)$.

Throughout we fix a small constant $\delta_* = \delta_*(f)>0$ and let
\[
d_*(x, c) = \begin{cases}
d(f(x), CV) & \text{ if } x \in \tilde B(\delta_*),\\
\delta_* & \text{ otherwise}.
\end{cases}
\] 
Replacing $\delta_*$ by a smaller constant, we may assume the following:
\begin{equation}
x \in \tilde B(\delta_*), \delta = d_*(x, c) \text{ and } t \in [-\delta, \delta] \Longrightarrow Df_t(x) \geq D(\delta).
\end{equation}
Recall that we set $I_c = I \setminus \{ c\}$.

If $J$ is an interval and $\lambda >0$, we use $\lambda J$ to denote the concentric open interval which has length $\lambda |J|$. We say that $J$ is $\lambda$-well-inside another interval $I$ or $I$ contains the $\lambda$-scaled neighborhood of $J$, if $I \supset (1 + 2 \lambda) J$. We shall use the following result throughout our analysis. For a proof, see \cite{BRSS} or  \cite{MS}.

\begin{prop}
For any $f \in \mathcal A$. Let $s \geq 1$ be an integer and let $T = (a, b)$ be an interval. Assume that $f^s|T$ is a diffeomorphism onto its image. Then

\begin{itemize}
\item[(1)] (the Koebe principle) If $J$ is a subinterval of $T$ such that $f^s(J)$ is $\tau$-well inside $f^s(T)$, then for any $x, y \in J$,
\[
\left( \frac{\tau}{1+\tau} \right)^2 \leq \frac{Df^s(x)}{Df^s(y)} \leq \left( \frac{1+\tau}{\tau} \right)^2.
\]
\item[(2)] (the macroscopic Koebe principle) If $J$ is a subinterval of $T$ such that $f^s(J)$ is $\tau$-well inside $f^s(T)$, then $J$ is $\tau'$-well inside $T$, where $\tau'=\tau^2/(1+2\tau)$.
\item[(3)] (the one-sided Koebe principle) Let $x \in T$ be such that 
\[
|f^s(a) - f^s(x)| \geq \tau |f^s(x) - f^s(b)|,
\]
then
\[
Df^s(x) \geq \left( \frac{\tau}{1+\tau} \right)^2 Df^s(b).
\]
\end{itemize}
\end{prop}

The following is a well-known result for smooth interval dynamics due to Ma${\rm \tilde{n}}$\'e \cite{Ma}. For a similar result for Rovella-like maps, see Alves and Soufi \cite{AS}.

\begin{prop}
Let $f \in \mathcal A$. For each neighborhood $U$ of $c$, there exists $C>0$ and $\lambda > 1$ depending only on $f$ such that for each $x \in I_c$ and $n \geq 1$, if $x, f(x), \cdots, f^{n-1} \notin U$, then $Df^n(x) \geq C \lambda ^n$. Moreover, for Lebesgue almost every $x \in I_c$ there exists an integer $n \geq 1$ such that $f^n(x) \in U$.
\end{prop}

\subsection{Reduced Main Theorem}

We adopt the following concept of {\it nice sets} as in the deterministic case.

\begin{definition}
A nice set for $\epsilon$-random perturbations is a measurable subset $V$ of $I_c \times \Omega_{\epsilon}$ with the following properties:

(1) For each $\omega \in \Omega_{\epsilon}, V^{\omega}$ (as a subset of $I_c$) is an open neighborhood of $c$.

(2) For each $\omega \in \Omega_{\epsilon}, x \in \partial V^{\omega}$  and for each $ n \geq 1$, we have
\[
f_{\omega}^n(x) \notin V^{\sigma^n \omega}.
\]
\end{definition}

\begin{definition}

Assume that $V$ is a nice set as above. A positive integer $m$ is called a Markov inducing time of $(x, \omega)\in V$, if there exists an interval $J \ni x$ such that

(1) $f_{\omega}^m$ maps $J$ diffeomorhpically onto $V^{\sigma^m \omega}$ with $\mathcal N(f_{\omega}^m|J) \leq 1$;

(2) if $x \in V^{\omega}$, then 
\[
\inf_{y \in J} Df_{\omega}^m (y) \geq e^2 \frac{|V^{\sigma^m \omega}|}{|V^{\omega}|}.
\]
For $(x, \omega) \in V$, let $m_V(x, \omega)$ denote the minimal Markov inducing time of $(x, \omega)$. If such a time does not exist, then set $m_V(x, \omega) = \infty$.
\end{definition}

\begin{thmR}
Let $\{f_t \}_{t \in [-1, 1]} $ be an admissible family with $f_0 =f \in \mathcal S_1$. For each $\epsilon > 0$ small, $\nu_{\epsilon}$ is a probability measure on $[-\epsilon, \epsilon]$ which belongs to the class $\mathbb M_{\epsilon}(L)$, where $L > 1$ is a fixed constant. Fix $p \geq 1$. Then for each $\delta_0 > 0$ small enough, there exist constants $C_0 > 0$ and $\epsilon_0 > 0$ with the following property: For each $\epsilon \in (0, \epsilon_0]$, there exists a nice set $V$ for $\epsilon$-random perturbations such that
\[
\tilde B(\delta_0) \times \Omega_{\epsilon} \subset V \subset \tilde B(2\delta_0) \times \Omega_{\epsilon};
\]
and such that 
\[
P_{\epsilon} (\{ (x, \omega) \in V: m_V(x, \omega) >m \}) \leq C_0 m^{-p}. 
\]

\end{thmR}

We shall briefly state the proof of the Main Theorem since the argument follows from \cite{S}[Section 3] easily.

Let $\mathcal P$ denote the set of Borel probability measure on $I$ and let $\mathcal T_{\epsilon} : \mathcal P \to \mathcal P$ be defined as 
\[
\mathcal T_{\epsilon} m(A) = \int_{[-\epsilon, \epsilon]} m(f_t^{-1}(A)) d \nu_{\epsilon} (t)  = \int_I p_{\epsilon} (x, A) d m(x)
\]
for $m \in \mathcal P$ and each Borel set $A \subset I$. It is equivalently to say that the measure $\mathcal T_{\epsilon} m$ is the pullback under transition kernel $p_{\epsilon}(x, \cdot)$. Note that a stationary measure $\mu_{\epsilon}$ for $\chi^{\epsilon}$ is just a fixed point of $\mathcal T_{\epsilon}$. It is also well-known that for each $m \in \mathcal P$, any weak$^{\star}$ accumulation point of the sequence $(1/n) \sum_{i=0}^{n-1} \mathcal T_{\epsilon}^i m$ is a stationary measure.

Assume $\nu_{\epsilon} \in \mathbb M_{\epsilon}(L)$. Then for each $m \in \mathcal P$ and each Borel set $A \subset I$, we have
\[
\mathcal T_{\epsilon} m(A) = \int_I p_{\epsilon} (x, A) d m(x) \leq L \left( \frac{|A|}{2\epsilon} \right)^{\frac{1}{L}}.
\]
It follows that any stationary measure $\mu_{\epsilon}$ for $\nu_{\epsilon}$ is absolutely continuous w.r.t. the Lebesgue measure.

By the Reduced Main Theorem, there exist $\delta_0 > 0$, $\epsilon_0 > 0$ and $C_0 > 0$ such that for each $\epsilon \in (0, \epsilon_0]$ there exists a nice set $V = V_{\epsilon}$ for $\epsilon$-random perturbations with $\tilde B(\delta_0) \subset V^{\omega} \subset \tilde B(2 \delta_0)$ for all $\omega \in \Omega_{\epsilon}$ and such that 
\begin{equation}
P_{\epsilon} (\{ (x, \omega) \in V: m_V(x, \omega) >m \}) \leq C_0 m^{-2}. 
\end{equation}
In the following, we fix such a choice of $V_{\epsilon}$ for each $\epsilon \in (0, \epsilon_0]$. Let 
\[
U_{\epsilon} = \{ (x, \omega) \in V_{\epsilon} : m_{V_{\epsilon}}(x, \omega) < \infty \},
\]
and let $G_{\epsilon} : U_{\epsilon} \to V_{\epsilon}$ denote the map $(x, \omega) \to F^{m_{V_{\epsilon}}(x, \omega)}(x, \omega)$. Since $P_{\epsilon} (V_{\epsilon} \setminus U_{\epsilon}) = 0$, $G^n_{\epsilon} (x, \omega)$ is defined for each $n \geq 0$ and almost every $(x, \omega) \in V_{\epsilon}$.

\begin{lem}
If $(x, \omega) \in {\rm dom}(G_{\epsilon}^n)$ and $k = \sum_{i=0}^{n-1} m_{V_{\epsilon}}(G_{\epsilon}^i(x, \omega)) $, then $f_{\omega}^k$ maps an interval $J_k^{\omega} (x)$ diffeomorphically onto $V_{\epsilon}^{\sigma^k \omega}$ and $\mathcal N (f_{\omega}^k| J_k^{\omega} (x) ) \leq 2e^2$, $|J_k^{\omega} (x)| \leq (e^2)^{-n}$.
\end{lem}

\begin{proof}

For each $0 \leq i \leq n-1$, let $m_i = m_{V_{\epsilon}}(G_{\epsilon}^i(x, \omega))$ and $k_i = \sum_{0 \leq j \leq i} m_j$. Set $k_{-1} = 0$. Then for $0 \leq i \leq n$, $G_{\epsilon}^i(x, \omega) = (f_{\omega}^{k_{i-1}}(x), \sigma^{k_{i-1}} \omega)$ with $f_{\omega}^{k_{i-1}}(x) \in V_{\epsilon}^{\sigma^{k_{i-1}} \omega}$. By Definition 3.2, let $T_i \ni f_{\omega}^{k_{i-1}}(x)$ be the interval such that 
\[
f_{\sigma^{k_{i-1}}\omega}^{m_i} : T_i \to V_{\epsilon}^{\sigma^{k_i} \omega} \mbox{ with \ } \mathcal N(f_{\sigma^{k_{i-1}}\omega}^{m_i} | T_i ) \leq 1, 0 \leq i \leq n-1.
\]
To simplify notations, let 
\[
F_i = f_{\sigma^{k_{i-1}}\omega}^{m_i} | T_i, 0 \leq i \leq n-1.
\]
By Definition 3.2 statement (2), 
\[
\inf_{y \in T_i} D F_i \geq e^2 \frac{|V_{\epsilon}^{\sigma^{k_i} \omega}|}{|V_{\epsilon}^{\sigma^{k_{i-1}}\omega}|} \Rightarrow \frac{|T_i|}{| V_{\epsilon}^{\sigma^{k_{i-1}}\omega}|} \leq e^2.
\]
Let 
\[
J_k^{\omega}(x) : = J_0 = (F_{n-1} \circ \cdots \circ F_0)^{-1} (V_{\epsilon}^{\sigma^k \omega}), J_1 = F_0 (J_0), \cdots, J_{n-1} = F_{n-2} (J_{n-2}) = T_{n-1}.
\]
By the same reason, we have that for $0 \leq i \leq n-1$,
\[
J_i \subset T_i \subset V_{\epsilon}^{\sigma^{k_{i-1}} \omega} \mbox{ with \ } \frac{|J_i|}{| V_{\epsilon}^{\sigma^{k_{i-1}} \omega}|} \leq (e^2)^{- (n-i)}.
\]
In particular, $|J_k^{\omega}(x)| \leq (e^2)^{-n} |V_{\epsilon}^{ \omega}| \leq (e^2)^{-n}$.

We first prove by induction on $i$ that
\[
\frac{|J_i|}{|T_i|} \leq e^{- (n - 1 - i)}.
\]
Indeed, for $i = n-1$, the claim is trivial since $J_{n-1} = T_{n-1}$. Assume the claim holds for some $1 \leq i \leq n-1$. Recall that ${\rm Dist}(F_i |T_i) \leq \mathcal N(F_i |T_i) \leq 1 $ for all $i \geq 0$. Hence by bounded distortion, for $i - 1$ we have
\[
\frac{|J_{i-1}|}{|T_{i-1}|} \leq e \frac{|J_i|}{|V_{\epsilon}^{\sigma^{k_{i-1}} \omega}|} = e \frac{|J_{i}|}{|T_i|}  \frac{|T_i|}{|V_{\epsilon}^{\sigma^{k_{i-1}} \omega}|} \leq e \cdot e^{- (n-1 -i)} \cdot e^{-2} = e^{- (n-1 -(i -1))}.
\]

\begin{claim}
For each $0 \leq i \leq n-1$, we have
\[
 {\rm Dist} (F_{n-1} \circ \cdots \circ F_i |J_i) \leq 2.
\]
\end{claim}

Since $\mathcal N(F_i |T_i) \leq 1$, then on each subinterval $J' \subset T_i$, ${\rm Dist}(F_i|J') \leq |J'|/ |T_i| $. By the chain rule, 
\begin{align*}
 {\rm Dist} (F_{n-1} \circ \cdots \circ F_i |J_i)  & \leq  {\rm Dist}(F_{n-1} |J_{n-1}) + \cdots + {\rm Dist}(F_i |J_i)  \\
 & \leq 1 + \frac{|J_{n-2}|}{|T_{n-2}|} + \cdots \frac{|J_i|}{|T_i|}  \\
 & \leq 1 + e^{-1} + \cdots + e^{-(n-1-i)} \leq 2.
\end{align*}

Finally, by the above claim we can show that 
\begin{align*}
\mathcal N(f_{\omega}^k | J_k^{\omega}(x)) & = \mathcal N(F_{n-1} \circ \cdots \circ F_0  | J_0) = \sup_{J' \subset J_0} {\rm Dist} (F_{n-1} \circ \cdots \circ F_0  | J') \frac{|J_0|}{|J'|} \\
& \leq \sup_{J' \subset J_0} \left\{ {\rm Dist} (F_{n-1} | F_{n-2} \circ \cdots \circ F_0 (J')) + \cdots + {\rm Dist} (F_0 | J') \right\} \frac{|J_0|}{|J'|} \\ 
& \leq \sup_{J' \subset J_0} \left\{ \frac{F_{n-2} \circ \cdots \circ F_0 (J')}{|T_{n-1}|}\frac{|J_0|}{|J'|} + \cdots +  \frac{|J'|}{|T_0|} \frac{|J_0|}{|J'|} \right\} \\
& \leq e^2 \left\{ \frac{|J_{n-1}|}{|T_{n-1}|} + \cdots + \frac{|J_0|}{|T_0|} \right\} \leq 2e^2.
\end{align*}

\end{proof}

Let $L^1 = L^1(I)$ denote the Banach space of all $L^1$ functions $\varphi: I \to \mathbb R$ w.r.t. the Lebesgue measure and let $||\varphi||_1$ denote the $L^1$ norm of $\varphi$. As usual, we will use the Perron-Frobenius operator. Given $ J \subset I$, $\omega \in \Omega$ and $n \geq 0$, define
\[
\mathcal L_{J, n}^{\omega} (x) = \sum_{\substack{f_{\omega}^n (y) = x\\ y \in J}} \frac{1}{D f_{\omega}^n (y)} \mbox{ and \ } \hat {\mathcal L}_{J, n}^{\omega}(x) = \frac{1}{|J|}\mathcal L_{J, n}^{\omega} (x) .
\]
We should remark that these are functions in $L^1$. Moreover, $\mathcal L_{J, n}^{\omega} (x) $ is the density function of the absolutely continuous measure $(f_{\omega}^n)_*({\rm Leb}|J) $ and has support $f_{\omega}^n (J)$. We also note that $\hat {\mathcal L}_{J, n}^{\omega}(x) $ is the density of the push forward of the relative measure on $J$, so the integral over $I$ of this density equals $1$.

\begin{lem}

For each $\rho > 0$ there exists a compact subset $\mathcal K(\rho)$ of $L^1$ such that for any interval $J \subset I$, any $\omega \in \Omega$ and any integer $n \geq 0$, if 
\begin{enumerate}
\item[(1)] $|f_{\omega}^n (J)| > \rho$, 

\item[(2)] $f_{\omega}^n$ maps $J$ diffromorhpically onto its image, and

\item[(3)] $\mathcal N(f_{\omega}^n |J) \leq 2e^2$.
\end{enumerate}
Then $\hat {\mathcal L}_{J, n}^{\omega}(x)  \in \mathcal K(\rho)$.

\end{lem}

\begin{proof}

For $C>1$, let $\mathscr D_C$ denote the subset of $L^1$ consisting of maps $\psi : I \to \mathbb R$ for which there exists an interval $I_{\psi} \subset I$ such that 

\begin{enumerate}
\item[(1)] $I_{\psi} \geq C^{-1}$;
\item[(2)] $\psi(x) = 0$ for all $I \setminus I_{\psi}$;
\item[(3)] $\psi(x) > 0$ for $x \in I_{\psi}$;
\item[(4)] $|\psi(x) - \psi(y)| \leq C\psi(x) |x-y|$ for all $x, y \in I_{\psi}$;
\item[(5)] $\int_0^1 \psi(x) dx =1$.
\end{enumerate}
Clearly, $\mathscr D_C$ is a compact subset of $L^1$. Moreover, for each $\rho>0$, there exists $C>1$ such that for any $\omega, J$ and $n$ as in the lemma, we have $\hat {\mathcal L}_{J, n}^{\omega}(x) \in \mathscr D_C$. Indeed, only (4) requires explanation. Take any $z_1, z_2 \in f_{\omega}^n (J)$ and $f_{\omega}^n (y_1) = z_1, f_{\omega}^n (y_2) = z_2$. Then
\begin{align*}
\big| \hat {\mathcal L}_{J, n}^{\omega}(z_1) - \hat {\mathcal L}_{J, n}^{\omega}(z_2) \big| & = \frac{1}{|J|} \cdot \bigg| \frac{1}{Df_{\omega}^n(y_1)} - \frac{1}{Df_{\omega}^n(y_2)} \bigg|  \\
& =  \frac{1}{|J|} \frac{1}{Df_{\omega}^n(y_1)}  \cdot \frac{|Df_{\omega}^n(y_2) - Df_{\omega}^n(y_1)|}{Df_{\omega}^n(y_2)} \\
& \leq \frac{1}{|J|} \frac{1}{Df_{\omega}^n(y_1)} \cdot 2e^2 \frac{D^2f_{\omega}^n(\xi_1)}{Df_{\omega}^n(\xi_1)} \cdot |y_2 - y_1| \\
& \leq  \frac{1}{|J|} \frac{1}{Df_{\omega}^n(y_1)} \cdot  \frac{(2e^2)^2}{|J|} \cdot \frac{|z_2 - z_1|}{Df_{\omega}^n (\xi_2)} \\
& \leq \frac{4e^4}{\rho} \hat {\mathcal L}_{J, n}^{\omega}(z_1)\cdot |z_2 - z_1|,
\end{align*} 
where $\xi_1$ and $\xi_2$ are between $y_1$ and $y_2$. So, taking $K(\rho) = \mathscr D_C$ which completes the proof.

\end{proof}

\begin{proof}[Proof of the Main Theorem] 

Statement (1) follows exactly from \cite{S}[subsection 3.1], so we only prove statement (2) here.

Take $Z = \tilde B(\delta_0)$. Let 
\[
\varphi_n(x) = \int_{\Omega_{\epsilon}} \mathcal L_{Z, n}^{\omega} (x) d \nu_{\epsilon}^{\mathbb N}(\omega).
\]
Since $\varphi_i(x) dx = \mathcal T_{\epsilon}^i ({\rm Leb} |Z)$, then $(1/n) \sum_{i=0}^{n-1} \varphi_i(x) dx$ converges to the unique stationary measure $\mu_{\epsilon}$. 

\begin{claim}
It suffices to prove that there exists a compact subset $\mathcal K$ of $L^1$  independent of $\epsilon$ and $n$ such that $\varphi_n \in \mathcal K$ for all $n$ and all $\epsilon >0$ small.  
\end{claim}

\noindent
{\bf Proof of Claim.} The assumption of the claim implies that there is a compact subset $\mathcal K_0$ (the convex hull of $\mathcal K$) of $L^1$, such that for each $n = 1, 2, \cdots $ and each $\epsilon > 0$ small enough, we have $(1/n) \sum_{i=0}^{n-1} \varphi_i (x) \in \mathcal K_0$. Since $\mu_{\epsilon}$ is the weak$^{\star}$ limit of $(1/n) \sum_{i=0}^{n-1} \varphi_i (x) dx $ as $n \to \infty$, it follows that $\mu_{\epsilon} = \varphi_{\epsilon} (x) dx$ for some $\varphi_{\epsilon}(x) \in \mathcal K_0$. Since any limit of $\varphi_{\epsilon}$ as $\epsilon \to 0$ is the density of an acip of $f$, and since $f$ is topological mixing, $f$ has at most one acip. It follows that $\varphi_{\epsilon}$ converges in $L^1$ as $\epsilon \to 0$, and converges to the density $\varphi$ of the acip of $f$.

By considering subsequences, it even suffices to show that for each $\eta >0$, there exists a compact subset $\mathcal K_{\eta} $ of $L^1$ such that for each $n$, $\varphi_n$ can be written in the following form:
\begin{equation}
\varphi_n : = \varphi_n^0 + \varphi_n^1 + \varphi_n^2,
\end{equation}
where $||\varphi_n^i||_1 \leq 2 \eta, i =0, 1$ and $\varphi_n^2 \in \mathcal K_{\eta}$.

Let $V = V_{\epsilon}$ and $G = G_{\epsilon}$. For $(x, \omega) \in V$, let $\mathbb M(x, \omega) $ denote the collection of positive integers of the form $\sum_{j=0}^{n-1} m_V(G^j (x, \omega))$, where $n$ runs over all positive integers for which $(x, \omega) \in {\rm dom}(G^n)$. For $m \geq 1$ and $k \geq 1$, let 
\[
U_{0, m} = \{ (x, \omega) \in V : m_V(x, \omega) = m \}
\]
and 
\[
U_{k, m} = \{ (x, \omega) \in V : k \in \mathbb M(x, \omega) \mbox{ and } F^k(x, \omega) \in U_{0, m} \}.
\]
Moreover, let $H_{k, m} = U_{k, m} \cap (Z \times \Omega_{\epsilon})$ for each $k \geq 0$ and $ m \geq 1$. Let $\mathcal H_{k, m}^{\omega}$ (resp. $\mathcal U_{k, m}^{\omega}$) denote the collection of the components of $H_{k, m}^{\omega}$ (resp. $U_{k, m}^{\omega}$). Note that if $x \in J \in \mathcal H_{k, m}^{\omega}$, then $k +m \in \mathbb M(x, \omega)$ and $ J \subset J_{k+m}^{\omega} (x)$, hence
\begin{equation}
\mathcal N(f_{\omega}^n |J) \leq 2e^2.
\end{equation}

Fix $n \geq 0$ and let $\Sigma_n = \{ (k, m) \in \mathbb N^2 : 0\leq k \leq n, m+k > n \}$. Then the sets $H_{k, m}, (k, m) \in \Sigma_n$ are pairwise disjoint, and for almost every $\omega \in \Omega_{\epsilon}$, $\mathcal H^{\omega} : = \bigcup_{(k, m) \in \Sigma_n} \mathcal H_{k, m}^{\omega}$ forms a measurable partition of $Z$ up to a set of Lebesgue measure 0. Then 
\begin{equation}
\varphi_n = \int_{\Omega_{\epsilon}} \sum_{J \in \mathcal H^{\omega}} \mathcal L_{J, n}^{\omega} d \nu_{\epsilon}^{\mathbb N} (\omega). 
\end{equation}

Now fix $ \eta >0$. For each $\omega \in \Omega_{\epsilon}$, we shall introduce a decomposition
\begin{equation}
\mathcal H^{\omega} = \mathcal H^{\omega, 0} \cup \mathcal H^{\omega, 1} \cup \mathcal H^{\omega, 2},
\end{equation}
and write
\[
\varphi_n^i =  \int_{\Omega_{\epsilon}} \sum_{J \in \mathcal H^{\omega, i}} \mathcal L_{J, n}^{\omega} d \nu_{\epsilon}^{\mathbb N} (\omega).
\]

The set $\mathcal H^{\omega, 0}$ is the collection of elements $\mathcal J$ of $\mathcal H^{\omega}$ for which $\partial J \cap \partial Z \neq \emptyset$ and $|J | < \eta$. For each $\omega \in \Omega_{\epsilon}$, $\mathcal H^{\omega, 0}$ has at most two elements. Thus, $|| \varphi_n^0 ||_1 \leq 2 \eta$.

To define $\mathcal H^{\omega, 1}$, we first observe that for each $(k, m) \in \Sigma_n$ and $\omega \in \Omega_{\epsilon}$, we have 
\[
| H_{k, m}^{\omega}| \leq | U_{k, m}^{\omega}| \leq C_1 | U_{0, m}^{\sigma^k \omega}|,
\]
where $C_1>0$ is a constant. Indeed, for any $(x, \omega) \in U_{k, m}$, we have
\[
f_{\omega}^k (U_{k, m}^{\omega} \cap J_k^{\omega} (x)) \subset U_{0, m}^{\sigma^k \omega},
\] 
and then the statement follows since $\mathcal N(f_{\omega}^k | J_k^{\omega}(x)) \leq 2e^2$. Let $M$ be a positive integer such that $C_0 C_1 M^{-1} < \eta$ and let $\mathcal H^{\omega, 1}$ be the collection of all components of $\bigcup_{(k, m) \in \Sigma_n, m >M} \mathcal H_{k, m}^{\omega}$ which are not contained in $\mathcal H^{\omega, 0}$.

Let us prove $|| \varphi_n^1 ||_1 < 2 \eta$. Let $G_{k, m} = \bigcup_{(k, m) \in \Sigma_n, m >M} H_{k, m}$, and let $G_M = \bigcup_{k=0}^n G_{k, M}$. Then for each $0 \leq k \leq n$,
\begin{align*}
P_{\epsilon} (G_{k, m}) & = \int_{\Omega_{\epsilon}} |G_{k, m}^{\omega}| d \nu_{\epsilon}^{\mathbb N} (\omega) \leq C_1 \int_{\Omega_{\epsilon}} \sum_{m > \max\{M, n-k \}} |U_{0, m}^{\sigma^k \omega}| d \nu_{\epsilon}^{\mathbb N} (\omega) \\
& = C_1 \int_{\Omega_{\epsilon}} \sum_{m > \max\{M, n-k \}} |U_{0, m}^{\omega}| d \nu_{\epsilon}^{\mathbb N} (\omega) \\
& = C_1 P_{\epsilon} (\{ (x, \omega) \in V : m_V(x, \omega) > \max\{M, n-k \} \}).
\end{align*}
By (3.2), we have
\[
|| \varphi_n^1 ||_1 \leq P_{\epsilon} (G_M) \leq C_0 C_1 \sum_{k=0}^{n} \max \{M, n-k \}^{-2} < 2 C_0 C_1 M^{-1} < 2 \eta.
\]

Finally, define $\mathcal H^{\omega, 2} = \mathcal H^{\omega} \setminus (\mathcal H^{\omega, 0} \cup \mathcal H^{\omega, 1})$. We shall show that for each $J \in \mathcal H^{\omega, 2}$, $|f_{\omega}^n (J)|$ is bounded from below by a constant $\rho = \rho(\eta) > 0$. Indeed, letting $(k, m) \in \Sigma_n$ be such that $J \in \mathcal H^{\omega}_{k, m}$, it suffices to show that $|f_{\omega}^{k+m}(J)|$ is bounded away from zero, since 
\[
|f_{\omega}^{k+m}(J)| \leq (\sup Df + \epsilon)^{k+m-n} |f_{\omega}^n (J)| \leq (\sup Df + \epsilon)^{M} |f_{\omega}^n (J)|.
\]
If $\partial J \cap \partial Z = \emptyset$, then $f_{\omega}^{k+m}(J)= V^{\sigma^{k+m} \omega}$, hence its length is bounded away from zero. If $\partial J \cap \partial Z \neq \emptyset$, then $|J| \geq \eta$, so by definition of $m_V$, $|f_{\omega}^{k+m}(J)|$ is bounded away from zero as well.

Combining with (3.4), by Lemma 3.2 there exists a compact subset $\mathcal K(\rho)$ of $L^1$ such that for each $J \in \mathcal H^{\omega, 2}$, $\hat {\mathcal L}_{J, n}^{\omega}(x)  \in \mathcal K(\rho)$. Therefore, $\varphi_n^2$ is contained in some compact subset $K_{\eta}$ of $L^1$.

\end{proof}

\subsection{Deterministic dynamcis}

In this subsection we study the deterministic dynamics of maps $f \in \mathcal {LD}$.  The following concept of `backward contraction'  was introduced in \cite{R} and  was very important in \cite{BRSS}.

\begin{definition} 
For a constant $r > 1$, that will be usually large, we say that $f$ satisfies the backward contraction property with constant $r$ (abbreviated $BC(r)$) if the following holds: there exists $\delta_0 >0$ such that for each $\delta < \delta_0$, each $s \geq 1$ and each component $W$ of $f^{-s}\tilde B(r \delta)$, 
\[
dist(W, CV) < \delta \Rightarrow |W| < \delta.
\]
We say that $f$ satisfies $BC(\infty)$ if it satisfies $BC(r)$ for all $r > 1$.
\end{definition}

Note that provided that $\delta$ is small enough, then
\[
\tilde B(r \delta) \approx r^{1/\ell} \tilde B(\delta).
\]

A sequence of open intervals $\{ G_j \}_{j=0}^s$ is called a {\bf chain} if for each $0 \leq j < s$, $G_j$ is a component of $f^{-j}(G_{j+1})$. The {\bf order} of the chain is defined to be the number of $j$'s with $0 \leq j < s$ and such that $\partial G_j$ contains the critical point $c$.

The follow lemma was adapted from \cite{BRSS}[Lemma 1] and \cite{CD}[Lemma 4]. The proof here also differs slightly from \cite{LR}[Sublemma 3.11] since our assumption on $f$ is weaker.

\begin{lem}
Suppose $f \in \mathcal {LD}$. Then for any $r>1$, there exists $\delta_0 >0$ such that the following holds for each $\delta \in (0, \delta_0)$. For each $c \in \{c^-, c^+ \}$, if $f^s(c) \in \tilde B(r\delta)$  for some $s \geq 1$  and if $J$ is the component of $f^{-s} \tilde B(r\delta)$ containing $c$ in its boundary, then 
\[
J \subset \tilde B(\delta).
\]
\end{lem}

\begin{proof}

Since $f \in \mathcal {LD}$, for a constant $K > 18 \cdot 8^{\ell} O_2 r$ to be large enough, there exists a neighborhood $\mathcal V$ of $c$ such that if $f^n(c) \in \mathcal V$, $c \in \{c^-, c^+ \}$, then $Df^n(f(c)) > K$.

Choose $\delta_0>0$ small enough such that $\tilde B(8^{\ell}r \delta) \subset \mathcal V$ for $\delta < \delta_0$. Consider the chains $\{G_j \}_{j=0}^s$ and $\{H_j \}_{j=0}^s$ with $G_s = \tilde B(8^{\ell}r \delta) \supset H_s = \tilde B(r \delta)$ and $G_0 \supset H_0 =J$. Let $s_1 < s$ be the maximal integer such that $G_{s_1}$ contains $c$ in its boundary, $c \in \{c^-, c^+ \}$. Let $H_{s_1+1}'$ be the convex hull of $H_{s_1 + 1} \cup \{ f(c)\}$, and observe that $H_{s_1+1}' \subset G_{s_1+1}$.

\begin{claim}
\[
H_{s_1} \subset \tilde B(\delta).
\]
\end{claim}

In fact, since 
\[
f^{s-s_1-1} : G_{s_1 +1} \to G_s
\]
is a diffeomorphism, and $H_s$ is $3$-well-inside $G_s$. By one-sided Koebe principle, applied to each components of $G_{s_1+1} \setminus \{ f(c) \}$ intersecting $H_{s_1+1}'$, that for each $x \in H_{s_1+1}'$, we have
\[
D f^{s-s_1-1} (x) \geq \left(\frac{3}{7}\right)^2 D f^{s-s_1-1} (f(c)).
\]
Since $f^{s-s_1}(c) \in \mathcal V$, $Df^{s-s_1}(f(c)) > K$. Then
\[
D f^{s-s_1-1} (f(c)) = \frac{D f^{s-s_1}(f(c))}{Df(f^{s-s_1}(c))} > \frac{K}{Df(f^{s-s_1}(c))},
\]
By non-flatness (C4), 
\[
Df(f^{s-s_1}(c)) \leq O_2 |G'|^{\ell -1},
\]
where $G'$ is the connected component of $G_s \setminus \{ c \}$ containing $f^{s-s_1}(c)$. By the Mean Value Theorem, 
\[
\frac{|f^{s-s_1-1}H_{s_1+1}'|}{|H_{s_1+1}'|} = D f^{s-s_1-1} (\zeta) \geq \left(\frac{3}{7}\right)^2 D f^{s-s_1-1} (f(c)).
\]
Therefore,
\[
|H_{s_1+1}'| \leq \left(\frac{7}{3}\right)^2 \frac{|f^{s-s_1-1}H_{s_1+1}'|}{D f^{s-s_1-1} (f(c))} \leq \frac{9 |G_s|}{D f^{s-s_1-1} (f(c))}.
\]
 Combining these together, we have
\[
|H_{s_1+1}'| \leq \frac{9 O_2 |G_s|^{\ell}}{K} \leq \frac{18 \cdot 8^{\ell} O_2 r \delta}{K} < \delta.
\]
The claim follows. If $s_1= 0$ then the proof is completed. For the general case, the lemma follows by an induction on $s$.

\end{proof}

\begin{lem}
Suppose $f \in \mathcal {LD}$. Then for any $r>1$, $f$ satisfies $BC(r)$. 
\end{lem}

\begin{proof}

Let $\delta >0$ be a small constant, let $x \in \tilde B(\delta)$. Let $s \geq 1$ be such that $f^s(x) \in \tilde B(r \delta)$ and let $J_k$ be the component of $f^{-(s-k)} \tilde B(r \delta)$ which contains $f^k(x)$. We want to show that $|J_1| < \delta$. 

Let us prove this by induction on $s$. If $s =1$, the statement is trivially true, since $J_1$ is empty set. Fix $s_0$ and assume that the statement holds if $s < s_0$. To prove the statement for $s = s_0$, consider the chains $\{ G_j \}_{j=0}^s$ with $G_s = \tilde B(8^{\ell}r \delta)$ and $G_0 \ni x$. We distinguish two cases:

{\bf Case 1.} There exists $0 \leq s_1 < s$ such that $G_{s_1}$ contains the critical point $c$ in its boundary. By the previous lemma, $G_{s_1} \subset \tilde B(\delta)$. If $s_1 = 0$, then $J_0 \subset G_0 \subset \tilde B(\delta)$. Otherwise, the statement follows by the induction hypothesis.

{\bf Case 2.} For any $0 \leq k < s$, $G_k$ contains no critical point in its boundary. Then $f^{s-1} : G_1 \to G_s$ is a diffeomorphism. By Koebe principle, $J_1$ is 1-well-inside $G_1$. Since $c \notin G_0, f(c) \notin G_1$ and $f(x) \in B(f(c), \delta)$, it follows that $|J_1| < \delta$. 

\end{proof}

We shall use the following version of backward contraction property. We emphasize the difference in the statement formulation between Definition 3.3 and Proposition 3 (as the difference between \cite{BRSS}[Theorem 1] and \cite{S}[Proposition 4.3]) is that: in the former, $\hat r(\delta)$ is first chosen as a fixed constant and then the contraction property is shown to hold for all $\delta > 0$ small; whereas in the latter, $\hat r(\delta) > 1$ depends on $\delta$ which is chosen to be fixed.

\begin{prop}
Suppose $f \in \mathcal {LD}$. For each $\delta >0$ small, there exists a constant $\hat r (\delta) > 1$ such that $\lim_{\delta \to 0} \hat r(\delta) = \infty$ and for each integer $s \geq 1$, if $W$ is a component of $f^{-s}\tilde B(\hat r(\delta) \delta)$ and $d(W, CV) \leq \delta$, then $|W| < \delta$.
\end{prop}

\begin{proof}

We only need to find a explicit form of the constant $\hat r(\delta)$. Without loss of generality, we may assume that the singular point $c$ is recurrent in the sense that $c \in \omega(c_1^+) = \omega(c_1^-)$. Since the non-recurrent case is much easier.

For each $\delta>0$ small, consider the neighborhood $\tilde B(8^{\ell} \sqrt \delta)$. Let $N_{\delta}$ be the maximal integer such that for each $c \in \{c^-, c^+ \}$ and each $1 \leq i < N_{\delta}$,
\[
f^{i} (c) \notin \tilde B(8^{\ell} \sqrt \delta).
\]
Clearly, $N_{\delta}$ is non-decreasing and $N_{\delta} \to \infty$ as $\delta \to 0$. Then let 
\[
K_{\delta} = \inf \{ Df^n(f(c)) | n \geq N_{\delta}, c \in \{ c^-, c^+\} \}.
\]
So $K_{\delta}$ is also non-decreasing and $K_{\delta} \to \infty$ as $\delta \to 0$. Therefore whenever $f^s(c) \in \tilde B(8^{\ell} \sqrt \delta)$, we have $n \geq N_{\delta}$ and hence $Df^n(f(c)) \geq K_{\delta}$.

Now we can take 
\[
\hat r(\delta) = \min \left\{\frac{K_{\delta}}{18 \cdot 8^l O_2}, \frac{1}{\sqrt{\delta}} \right\}, \hat r(\delta) \to \infty \mbox{ as \ } \delta \to 0.
\] 
Taking into account the proof of Lemma 3.3 and Lemma 3.4, we can conclude the proof here. In particular, we have $\hat r(\delta) \delta \to 0$ as $\delta \to 0$.

\end{proof}

\begin{remark}

The form of the growth function $\hat r(\delta)$ for multimodal interval maps under different growth condition on derivatives was studied in \cite{BRSS}[Theorem 3]. For example, if $f$ satisfies the Collet-Eckmann condition, then $\hat r(\delta) = C \delta^{- \alpha}$ where $C > 0, \alpha \in (0, 1]$ are constants.

\end{remark}

Once we have proved Proposition 3, we can obtain the following results which are reformulations of \cite{LR}[Proposition 3.7, Proposition 3.8]. The proofs hold without modification since the Rovella-like condition (R2) and (R3) are not essentially used in this part.

Let $\mathcal L(\delta)$ denote the collection of all orbits $\{f^j(x) \}_{j=0}^n$ with $f^j(x) \notin \tilde B(\delta)$ for each $j = 0, 1, \cdots, n-1$ and $f^n(x) \in \tilde B(2\delta)$, for some $n \geq 1$.

\begin{prop}

Given $f \in \mathcal{LD}, L > 1, \theta \in (0, 1)$ and $\zeta > 0$, for any critical value $v \in \{c_1^-, c_1^+ \}$ and any $\delta > 0$ small enough, there exists a positive integer $M_v(\delta)$ such that the following hold:
\begin{equation}
A(v, f, M_v(\delta)) : = \sum_{i=0}^{M_v(\delta) -1} \frac{Df^i (x)}{d(f^i(x), c) } \leq \frac{\theta}{\delta}, 
\end{equation}
\begin{equation}
f^j (v) \notin \tilde B(L \delta) \text{ for each } j = 0, 1, \cdots, M_v(\delta) -1, 
\end{equation}
and 
\begin{equation}
 Df^{M_v(\delta) +1} (v)  \geq \left( \frac{\delta'}{\delta} \right)^{1 - \zeta},  
\end{equation}
where 
\[
\delta' = \max\{ d_*(f^{M_v(\delta)}(v), c), \delta \}.
\]
Moreover, 
\begin{equation}
M_v(\delta) \to \infty \text{ as } \delta \to 0. 
\end{equation}

\end{prop}

\begin{prop}
Given $f \in \mathcal {LD}$, there exists a constant $\kappa_0 > 0$ such that for each $\delta > 0$ small, the following holds. For $\{f^j(x) \}_{j=0}^n \in \mathcal L(\delta)$, putting $\delta'' = \max\{ d(x, CV), \delta\}$, we have
\[
Df^n (x) \geq \frac{\kappa_0}{D(\delta)} \left(\frac{\delta}{\delta''} \right)^{1 - \frac{1}{\ell}}.
\]
\end{prop}

\subsection{Random expansion results}

In order to control the distortion of iterates of random perturbations of $f$, we shall use the well-known `telescope' technique appearing in \cite{BC, Ts}. 

Consider an admissible one-parameter family $\{f_t \}_{t \in [-1, 1]}$ with $f_0 =f \in \mathcal A$. For $x \in I_c, \omega \in \Omega$ and an integer $n \geq 1$, let 
\[
A(x, \omega, n) = \sum_{i=0}^{n-1} \frac{Df_{\omega}^i(x)}{d(f_{\omega}^i(x), c)}.
\]
So, if $f_{\omega}^j (x) = c$ for some $j \in \{ 0, 1, \cdots, n-1 \}$, then we set $A(x, \omega, n) = \infty$, and in this case $f_{\omega}^j(x)$ is meaningless for $i \in \{ j+1, \cdots, n-1 \}$.

The following lemma is Lemma 3.2 in \cite{LR} whose proof is an easy adaption of Lemma 2.3 in \cite{S}. This lemma provides us Markov structure.

\begin{lem}
Let $\{f_t \}_{t \in [-1, 1]} $ be an admissible family with $f_0 =f \in \mathcal A$. Then there exists a constant $\theta_0 >0$ such that for any $(x, \omega) \in I_c \times \Omega$ and any integer $n \geq 1$ with $A(x, \omega, n) < \infty$, setting 
\[
J = \bigg[ x - \frac{\theta_0}{A(x, \omega, n)}, x + \frac{\theta_0}{A(x, \omega, n)} \bigg] \cap I_c.
\]
Then we have that $f_{\omega}^n |J$ is a diffeomorphism and $\mathcal N (f_{\omega}^n | J) \leq 1/2 < 1$. In particular, $c \notin J$. Moreover, for each $y \in J$, we have 
\begin{equation}
e^{-1} A(x, \omega, n) <A(y, \omega, n) < e A(x, \omega, n)
\end{equation}
and
\begin{equation}
e^{-2} \frac{Df_{\omega}^n(x)}{ A(x, \omega, n)} \leq \frac{Df_{\omega}^n(y)}{A(y, \omega, n)} \leq e^2 \frac{ Df_{\omega}^n(x)}{A(x, \omega, n)}.
\end{equation}
\end{lem}

When a point $y$ is sufficiently close to a critical value $v$, then we expect the orbit of $y$ to shadow the early iterates of $v$ at least for some period of time. This insight leads to the so-called binding argument.

\begin{definition}
Let $\{f_t \}_{t \in [-1, 1]} $ be an admissible family with $f_0 =f \in \mathcal A$. Given $v \in I_c, \epsilon >0$ and $C > 0$, a positive integer $N$ is called a $C$-binding period for $(v, \epsilon)$ if for each $y \in I_c$ with $d(y, v)\leq \epsilon$, each $\omega \in \Omega_{\epsilon}$ and each $0 \leq j <N$, the following hold:
\begin{align}
2|f_{\omega}^j (y) - f^j(v)|& \leq d(f^j(v), c);\\
\frac{1}{e} Df^{j+1}(v) & \leq D f_{\omega}^{j+1}(y)  \leq e Df^{j+1}(v);\\
C\epsilon Df^{j+1}(v) & \geq |f_{\omega}^{j+1}(y) - f^{j+1}(v) |.
\end{align}

\end{definition}

We remark here that (3.14) implies that $v$ and $y$ are always on the same side of $c$.

The following is Lemma 3.4 in \cite{LR} which is also an adaption of Lemma 2.5 in \cite{S}.

\begin{lem}
Let $\{f_t \}_{t \in [-1, 1]} $ be an admissible family with $f_0 =f \in \mathcal A$. Then there exists a constant $\theta_1>0$ such that the following holds provided that $\epsilon >0$ is small enough. For any point $v \in I_c$, and let $N$ be a positive integer such that 
\[
W : = \sum_{j=0}^N \frac{1}{Df^j(v)} < \infty \text{ and } A(v, 0, N) W \leq \frac{\theta_1}{\epsilon}.
\]
Then $N$ is an $eW$-binding period for $(v, \epsilon)$.
\end{lem}

We will prove expansion result for random systems analogously to \cite{S}[Proposition 2.7]. The proof differs slightly with Lemma 3.5 and Lemma 3.6 in \cite{LR}, so we will briefly give the proof here.

\begin{prop}
Let $\{f_t \}_{t \in [-1, 1]} $ be an admissible family with $f_0 =f \in \mathcal A$. For any neighborhood $U$ of $c$, there exist $K>1$ and $\eta >0$ such that the following hold provided $\epsilon >0$ is small enough.

\begin{itemize}
\item[(1)] For $x \in I_c$, $\omega \in \Omega_{\epsilon}$ and $n \geq 1$, if $f_{\omega}^j(x) \notin U$ for all $0 \leq j <n$, then $Df_{\omega}^n (x) \geq K^{-1} e^{\eta n}$.
\item[(2)] For each $\omega \in \Omega_{\epsilon}$ and $n \geq 1$,
\[
\big|\{x  \in I_c : f_{\omega}^j (x) \notin U \text{ for } 0 \leq j <n  \} \big| \leq K e^{- \eta n}.
\]
\end{itemize}
\end{prop}

\begin{proof}[Proof of statement (1)]

First note that due to the presence of perturbations, even if $f_{\omega}^j(x) \notin U$ for all $0 \leq j < n$, it doesn't mean that $f^j(x) \notin U$ for all $0 \leq j < n$. So we shall consider a neighborhood of $c$ strictly inside $U$.

Let $U_0$ be a neighborhood of $c$ such that $U_0 \Subset U$. Let $C>0$ and $\lambda>1$ be given by 
Proposition 2 for $U_0$ and let $N$ be a large integer such that $C \lambda^N \geq 4$. By continuity, provided $\epsilon < d(\partial U, \partial U_0)$ small enough, for any $x \in I_c$, $\omega \in \Omega_{\epsilon}$ and any $0 \leq l \leq N$, we have
\[
|f^l(x) - f_{\omega}^l(x)| < \epsilon \text{ and } |Df^l(x) - Df_{\omega}^l(x) | < \frac{C}{2}.
\]

Consider $x \in I_c$, $\omega \in \Omega_{\epsilon}$ such that $f_{\omega}^j (x) \notin U$ for all $0 \leq j < n$. Assume $\epsilon$ is small, and write $n = k N + r$ with $k \in \mathbb N$ and $0 \leq r < N$. Then
\[
D f_{\omega}^r (x) \geq Df^r (x) - \frac{C}{2} \geq \frac{C}{2} \lambda^r.
\]
Similarly, for each $0 \leq i <k$, we have
\[
D f_{\sigma^{iN +r}\omega}^N (f_{\omega}^{iN +r} (x) ) \geq D f^N (f_{\omega}^{iN +r} (x) ) - \frac{C}{2} \geq \frac{C}{2} \lambda^N \geq 2.
\]
Thus
\begin{align*}
Df_{\omega}^n (x) & = Df_{\omega}^r (x) \cdot Df_{\sigma^r\omega}^N (f_{\omega}^r (x)) \cdots D f_{\sigma^{(k-1)N +r}\omega}^N (f_{\omega}^{(k-1)N +r} (x) ) \\
& \geq 2^k \cdot \frac{C \lambda^r}{2} \geq \frac{C}{4} \cdot 2^{k+1} \geq \frac{C}{4} \cdot 2^{\frac{n}{N}} \geq \frac{1}{K} e^{\eta n},
\end{align*}
where $K = 4/C$ and $\eta = \log 2/N$.

\end{proof}

\begin{proof}[Proof of statement (2)]

For each $\omega \in \Omega_{\epsilon}$ and each $n \geq 1$, let
\[
\Lambda_n^{\omega} (U) = \{x\in I_c : f_{\omega}^j(x) \notin U \text{ for } 0 \leq j <n \} \mbox{ and \ } 
\Lambda_{\infty}^{\omega} (U) = \bigcap_{n=1}^{\infty} \Lambda_n^{\omega} (U).
\]
Let $U_0 \Subset U$ be an open neighborhood of $c$ and define $\Lambda_n^{\omega} (U_0)$ and $\Lambda_{\infty}^{\omega} (U_0)$ similarly.

Assume $\epsilon >0$ small. By statement (1), for each $x \in \Lambda_n^{\omega} (U)$ we have
\[
A(x, \omega, n) \asymp D f_{\omega}^n (x).
\]
In fact,  
\[
\frac{A(x, \omega, n)}{D f_{\omega}^n (x)} = \sum_{i=0}^{n-1} \frac{Df_{\omega}^i(x)}{D f_{\omega}^n (x)} \cdot \frac{1}{d(f_{\omega}^i(x), c)} = \sum_{i=0}^{n-1} \frac{1}{D f_{\sigma^i \omega}^{n-i} (f_{\omega}^i (x))} \cdot \frac{1}{d(f_{\omega}^i(x), c)}.
\]
Since $f_{\omega}^j(x) \notin U \text{ for } 0 \leq j <n$, $d(f_{\omega}^i(x), c) \geq \frac{|U|}{2}$. Therefore,
\[
\frac{A(x, \omega, n)}{D f_{\omega}^n (x)} \leq \frac{2}{|U|} \sum_{i=0}^{n-1} \frac{1}{D f_{\sigma^i \omega}^{n-i} (f_{\omega}^i(x))} \leq \frac{2}{|U|} \sum_{i=0}^{n-1} \frac{K}{e^{\eta(n-i)}} \leq C_1.
\]
On the other hand,
\[
\frac{A(x, \omega, n)}{D f_{\omega}^n (x)} \geq \sum_{i=0}^{n-1} \frac{1}{D f_{\sigma^i \omega}^{n-i} (f_{\omega}^i(x))} \geq \frac{1}{D f_{\sigma^{n-1} \omega} ( f_{\omega}^{n-1} (x))} \geq \frac{1}{\sup_{\omega \in \Omega_{\epsilon}} Df_{\omega}(x)} = C_2 > 0.
\]

By Lemma 3.5, there exists a constant $\theta_0 >0$ independent of $x$ and $n$ such that $f_{\omega}^n | J$ is a diffeomorphism and $\mathcal N(f_{\omega}^n | J) \leq 1$, where 
\[
J =  \bigg[ x - \frac{\theta_0}{A(x, \omega, n)}, x + \frac{\theta_0}{A(x, \omega, n)} \bigg].
\]
Then
\[
|f_{\omega}^n (J)| \asymp Df_{\omega}^n (x) \cdot |J| = 2 \theta_0 \cdot \frac{Df_{\omega}^n (x)}{A(x, \omega, n)} \asymp \theta_0.
\]
Hence there exists $\tau >0$ independent of $x$ and $n$ such that $f_{\omega}^n$ maps a neighborhood $J_n(x)$ of $x$ onto an interval of length $\tau$ with $\mathcal N(f_{\omega}^n | J_n(x)) \leq 1$. We may assume that $\tau < d(\partial U, \partial U_0)$ small, which guarantees $J_n(x) \subset \Lambda_n^{\omega} (U_0)$ for each $x \in \Lambda_n^{\omega}(U)$.

Let $\rho$ be a small constant to be determined later. We first claim that there exists positive integer $N = N(\rho)$ such that 
\[
| \Lambda_N^{\omega}(U_0)| < \rho
\]
whenever $\epsilon>0$ is small enough. Indeed, for the unperturbed map $f$, since $\Lambda_{\infty}^{0}(U_0)$ is a compact set with Lebesgue measure $0$, there exists an integer $N \geq 1$ such that $|\Lambda_{N}^{0}(U_0)|< \rho$. Assuming $\epsilon >0$ small enough, then for each $\omega \in \Omega_{\epsilon}$, $ \Lambda_N^{\omega}(U_0)$ is contained in a small neighborhood of  $\Lambda_N^{0}(U_0)$. The claim follows.

Now for each $k \geq 1$, define $\eta_k : = \sup_{\omega \in \Omega_{\epsilon}} |\Lambda_{kN}^{\omega}(U)|$. We claim that $\eta_{k+1} < \eta_k/2$ holds for all $k =1, 2, \cdots$ provided $\epsilon>0$ small enough. Let $\tilde \Lambda$ be the union of the intervals $J_N(x), x \in \Lambda_{N}^{\omega}(U)$. Then 
\[
|\tilde \Lambda| \leq |\Lambda_{N}^{\omega}(U_0)| < \rho.
\]
For each $x \in \Lambda_{(k+1)N}^{\omega} (U)$, let $J$ be any connected component of $J_N(x) \cap \Lambda_{(k+1)N}^{\omega} (U)$. By bounded distortion, we have
\[
\frac{|J|}{|J_N(x)|} \leq e \frac{|f_{\omega}^N(J)|}{|f_{\omega}^N (J_N(x)) | } \leq e \tau^{-1} |f_{\omega}^N(J)|.
\]
Since $f_{\omega}^N: \Lambda_{(k+1)N}^{\omega} (U) \to  \Lambda_{kN}^{\sigma^N \omega } (U) $, summing over all $J$, we have
\[
| J_N(x) \cap \Lambda_{(k+1)N}^{\omega} (U)| \leq e\tau^{-1} \eta_k |J_N(x)|.
\]
By Besicovitch's covering lemma, there exists a sub-family of $\{ J_N(x) : x \in \Lambda_{N}^{\omega}(U) \}$ with uniformly bounded intersection multiplicity which forms a covering of $\Lambda_{N}^{\omega}(U) $. Thus,
\[
| \Lambda_{(k+1)N}^{\omega} (U) | \leq C \tau^{-1} \eta_k |\tilde \Lambda| \leq C \eta_k \frac{\rho}{\tau},
\]
where $C$ is a universal constant. So we can choose $\rho$ small enough such that
\[
 \frac{ C \rho}{\tau} < \frac{1}{2}.
\]
This proves $\eta_{k+1} < \eta_k/2$.

For each $n \geq 1$, as in the proof of statement (1), write $n = kN +r$ with $k \in \mathcal N$ and $0 \leq r < N$. Since $Df_{\omega}^r (x) \geq C \lambda^r/2$, we have
\[
| \Lambda_{n}^{\omega} (U) | \leq \frac{ | \Lambda_{kN}^{\sigma^k \omega} (U)|}{Df_{\omega}^r (x)} \leq \frac{2}{C \lambda^r} \left(\frac{1}{2} \right)^k.
\]
This finishes the proof.

\end{proof}

\section{Growth of derivatives along pseudo-orbits}

The main goal of this section is to prove the following theorem. 

\begin{thm}
Let $\{f_t \}_{t \in [-1, 1]} $ be an admissible family with $f_0 =f \in \mathcal S_1$. For each $\epsilon > 0$ small enough, there exist $\Lambda(\epsilon) > 1$ and $\alpha(\epsilon) > 0$ such that 
\[
\lim_{\epsilon \to 0} \Lambda(\epsilon) = \infty \text{ and } \lim_{\epsilon \to 0} \alpha(\epsilon) = 0,
\]
and such that the following hold.
\begin{enumerate}
\item[(1)] Let $x \in I_{ c }$ and $\omega \in \Omega_{\epsilon}$ be such that $d(x, CV) \leq 4 \epsilon, f_{\omega}^j(x) \notin \tilde B(\epsilon)$ for $1 \leq j \leq s-1$ and $f_{\omega}^s (x) \in \tilde B(2 \epsilon)$. Then
\[
Df_{\omega}^s (x) \geq \frac{\Lambda(\epsilon)}{D(\epsilon)} e^{\epsilon^{\alpha(\epsilon)} s}.
\]
\item[(2)]  Let $x \in I_c$ and $\omega \in \Omega_{\epsilon}$ be such that $f_{\omega}^j(x) \notin \tilde B(\epsilon)$ for $0 \leq j \leq s-1$. Then
\[
Df_{\omega}^s (x) \geq A \epsilon^{1 - \frac{1}{\ell}} e^{\epsilon^{\alpha(\epsilon)} s},
\]
where $A>0$ is a constant independent of $\epsilon$.
\end{enumerate}
\end{thm}

In case that $f$ is in the Rovella family, it was proved in \cite{LR}[Proposition 3.1] that $\epsilon^{\alpha(\epsilon)}$ can be replaced by a positive constant independent of $\epsilon$. However, we believe that the recurrence condition (R3) therein can be dropped.

To prove Theorem 2, we shall decompose the random orbit into pieces, each of which is shadowed by either the true orbit of the critical value or a true orbit corresponding to a first landing into a critical neighborhood. The proof will be given in subsection 4.2. In subsection 4.3, we collect a few properties for return maps to the critical neighborhood $\tilde B(\epsilon)$ for random perturbations. 

\subsection{Return to critical neighborhoods}

In this subsection we shall prove the following proposition.

\begin{prop}
Let $\{f_t \}_{t \in [-1, 1]} $ be an admissible family with $f_0 =f \in \mathcal S_1$. For each $\epsilon >0$ small, there exists a constant $\hat \Lambda(\epsilon) > 0$ such that $\lim_{\epsilon \to 0} \hat \Lambda(\epsilon) = \infty$ and such that for each $\omega \in \Omega_{\epsilon}, x \in I_c$ with $d(x, CV) \leq 4\epsilon$ and an integer $s \geq 1$, if $f_{\omega}^j (x) \notin \tilde B(\epsilon)$ for $1 \leq j < s$ and $f_{\omega}^s(x) \in \tilde B(2\epsilon)$, then
\[
D f_{\omega}^s(x) \geq \frac{\hat \Lambda(\epsilon)}{D(\epsilon)}.
\]
\end{prop}

\begin{remark}

The proof of Proposition 7 follows from \cite{LR}[Proposition 3.14]. However, there are several important places that needed to be modified. First, the small constant $\eta_*$ appeared in \cite{LR}[Definition 3.15] is confusing and useless for contracting Lorenz maps, since there is only one critical point. Second, the constant $\zeta_2$ appeared in the proof of Proposition 3.14 and Proposition 3.16 therein is also redundant. For convenience of the readers, we shall state the proof here. 

\end{remark}

Let $C_0 = \max_{[0, 1]} Df \geq 1$. Let 
\[
W_0 = \max_{v \in CV} \sum_{n=0}^{\infty} \frac{1}{Df^n (v)}.
\]
Let $\theta > 0$ be a small constant such that 
\begin{equation}
4 \theta W_0 \leq \theta_1
\end{equation}
where $\theta_1$ is given in Lemma 3.6. Moreover, fix constants $L > 2^{\ell + 1}$ and $\zeta \in (0, 1/\ell)$. For $v \in CV$ and $\delta > 0$ small, we fix a positive integer $M_v(\delta)$, called {\bf the preferred binding period} for $(v, \delta)$, such that the conclusion of Proposition 4 holds for these constants $\theta, L$ and $\zeta$. Since $M_v(\delta) \to \infty$ as $\delta \to 0$ for each $v \in CV$, we have
\begin{equation}
\Lambda_0(\delta) : = \inf_{v \in CV} Df^{M_v(\delta) + 1}(v) \to \infty \text{ as } \delta \to 0. 
\end{equation}

\begin{prop}
Let $\{f_t \}_{t \in [-1, 1]} $ be an admissible family with $f_0 =f \in \mathcal S_1$, there exists a positive constant $\zeta_1$ with the following property. For $\delta > 0$ sufficiently small, and $v \in CV$, let $M = M_v(\delta) \geq 1$ be the preferred binding period defined as above. Then for any $\omega \in \Omega_{\delta}$ and $y \in I_c$ with $d(y, v)\leq 4 \delta$, we have 
\begin{equation}
y_j : = f_{\omega}^j(y) \notin \tilde B(2 \delta) \text{ for all } 0 \leq j < M, 
\end{equation}
\begin{equation}
Df_{\omega}^M (y) \geq \frac{\Lambda_0(\delta)^{\zeta_1}}{D(\delta)}. 
\end{equation}
Moreover, if $y_M \notin \tilde B(\delta)$, then
\begin{equation}
Df_{\omega}^{M+1}(y) \geq \Lambda_0(\delta)^{\zeta_1} \left( \frac{d_*(y_M, c)}{\delta} \right)^{1 - \frac{1}{\ell}}. 
\end{equation}
\end{prop}

\begin{proof}

Fix $v \in CV$ and $\delta > 0$ small. By (3.7) and (4.1),
\[
A(v, f, M)W_0 \leq \frac{\theta}{\delta} \cdot W_0 \leq \frac{4 \theta W_0}{4 \delta} \leq \frac{\theta_1}{4\delta}.
\]
Then by Lemma 3.6, $M$ is an $eW_0$ binding period for $(v, 4 \delta)$. By (3.8) and (3.13), (4.3) holds provided $\delta > 0$ is small enough.

By non-flatness, there exists a constant $C_1 > 1$ independent of $\delta$ such that 
\begin{equation}
Df(f^M(v)) \leq C_1 D(\delta'), 
\end{equation}
where $\delta' = \max\{d_*(f^M(v), c), \delta\}$. To be precise, we may assume $\delta \leq \delta_*$. If $f^M(v) \in \tilde B(\delta_*)$, then $Df(f^M(v)) \leq O_2 d(f^M(v), c)^{\ell -1} \leq C_1 D(\delta')$; for otherwise, $Df(f^M(v)) \leq (C_0/D(\delta_*)) D(\delta_*) $.

Let $\zeta_1 = ({\ell}^{-1} - \zeta)/(2 - 2 \zeta)$. By (3.9) and the definition of $\Lambda_0(\delta)$ we obtain
\begin{equation}
Df^{M+1}(v) \geq \Lambda_0(\delta)^{2 \zeta_1} \cdot \Lambda_0(\delta)^{1- 2 \zeta_1} \geq \Lambda_0(\delta)^{2 \zeta_1} \left( \frac{\delta'}{\delta} \right)^{1 - \frac{1}{\ell}}. 
\end{equation}

Let us prove $(4.4)$. By (3.15), (4.6) and (4.7), we have
\[
Df_{\omega}^M (y) \geq \frac{Df^{M+1}(v)}{e Df(f^M(v))} \geq \frac{\Lambda_0(\delta)^{2 \zeta_1}}{eC_1 D(\delta')} \left( \frac{\delta'}{\delta} \right)^{1 - \frac{1}{\ell}} \geq \frac{C_2 \Lambda_0(\delta)^{2 \zeta_1}}{D(\delta')},
\]
where $C_2 > 0$ is a constant and we used the fact $\delta' \geq \delta$. When $\delta > 0$ is small enough, $C_2 \Lambda_0(\delta)^{\zeta_1} > 1$, so (4.4) holds.

In the rest we prove (4.5). Assume that $\delta'' : = d_*(y_M, c) \geq \delta$. By (3.1) we have
\begin{equation}
Df_{\sigma^M \omega} (y_M) \geq C_3 D(\delta''),
\end{equation}
where $C_3 > 0$ is a constant.  We distinguish two cases.

{\it Case i.} $\delta'' \geq \Lambda_0(\delta)^{2 \zeta_1} \delta' \geq \Lambda_0(\delta)^{2 \zeta_1} \delta$. When $\delta$ is small enough, $\Lambda_0(\delta)$ is large and $\delta' \leq \delta''/\Lambda_0(\delta)^{2 \zeta_1}$. So $\delta' \ll \delta''$. Thus, there exists $C_4 > 0$ such that $\eta_M : = |y_M - f^M(v)| \geq C_4 |\tilde B(\delta'')|$. By (4.8),
\[
\eta_M Df_{\sigma^M \omega} (y_M) \geq C_3 C_4  |\tilde B(\delta'')| D(\delta'') = C_3 C_4 \delta''.
\]
By (3.14) and (3.15), 
\[
Df_{\omega}^{M+1}(y)  \geq \frac{Df^M(v)}{e} Df_{\sigma^M \omega} (y_M) \geq \frac{\eta_M Df_{\sigma^M \omega} (y_M)}{4e^2 W_0 \delta}\geq C_5 \frac{\delta''}{\delta} \geq C_5 \Lambda_0(\delta)^{2 \zeta_1} \left( \frac{\delta''}{\delta}\right)^{1 - \frac{1}{\ell}},
\]
where $C_5 > 0$ is a constant. The inequality (4.5) holds when $\delta$ is small enough.

{\it Case ii.} $\delta'' \leq \Lambda_0(\delta)^{2 \zeta_1} \delta' $. In this case, combine (3.14), (4.6), (4.7) and (4.8), we have
\[
Df_{\omega}^{M+1}(y)  \geq \frac{Df^{M+1}(v)}{e} \frac{Df_{\sigma^M \omega}(y_M)}{Df(f^M(v))}  \geq C_6 \Lambda_0(\delta)^{2 \zeta_1} \left( \frac{\delta'}{\delta} \right)^{1 - \frac{1}{\ell}} \left( \frac{\delta''}{\delta'} \right)^{1 - \frac{1}{\ell}}  \geq C_6 \Lambda_0(\delta)^{2 \zeta_1} \left( \frac{\delta''}{\delta} \right)^{1 - \frac{1}{\ell}},
\]
where $C_6 > 0$ is a constant. The inequality (4.5) holds when $\delta$ is small enough. 

\end{proof}

Let $\mathcal O^{\epsilon} (\delta)$ denote of the collection of $\epsilon$-random orbits $\{x_j \}_{j=0}^n$ for which $x_j \notin \tilde B(\delta)$ for each $0 \leq j < n$, and let $\mathcal L^{\epsilon}(\delta)$ denote the collection of $\epsilon$-random orbits $\{x_j \}_{j=0}^n \in \mathcal O^{\epsilon} (\delta)$ for which $x_n \in \tilde B(\delta)$.

\begin{lem}
Let $\{f_t \}_{t \in [-1, 1]} $ be an admissible family with $f_0 =f \in \mathcal S_1$. For each $\delta > 0$, there exists $\epsilon = \epsilon(\delta) > 0$ and $\hat \eta = \hat \eta(\delta) > 0$ such that for any $\epsilon$-random orbit $\{ f_{\omega}^j(x) \}_{j=0}^n \in \mathcal L^{\epsilon}(\delta)$, we have
\[
Df_{\omega}^n(x) \geq \frac{\kappa}{D(\delta)} \left( \frac{\delta}{\delta''} \right)^{1 - \frac{1}{\ell}}e^{\hat \eta n},
\]
where $\delta'' = \max\{ d(x, CV), \delta\}$ and $\kappa > 0$ is a constant independent of $\delta$.
\end{lem}

\begin{proof}

Fix $\delta >0$. By Proposition 6 statement (1), there exist $C>0$ and $\eta >0$ such that if $\epsilon > 0$ is small enough, then for each $\epsilon$-random orbit $\{ f_{\omega}^j(x) \}_{j=0}^n \in \mathcal L^{\epsilon}(\delta)$ we have
\[
Df_{\omega}^n(x) \geq C e^{\eta n}.
\]
If $C e^{\eta n/2} \geq 1 / D(\delta)$, then the desired estimate holds with $\kappa = 1$ and $\hat \eta = \eta/2$. So assume the contrary. Then $n$ is bounded from above by a constant $N(\delta)$ since $C$ and $\eta$ depend only on $\delta$. When $\epsilon$ is small enough, we have $\{ f^j(x) \}_{j=0}^n \in \mathcal L(0.9 \delta)$ and 
\[
D f_{\omega}^n(x) \geq \frac{Df^n(x)}{2}.
\]
By Proposition 5, there is a constant $\kappa_0 > 0$ such that 
\[
D f_{\omega}^n(x) \geq \frac{\kappa_0}{2D(\delta)} \left( \frac{\delta}{\delta''} \right)^{1 - \frac{1}{\ell}}.
\]
Taking $\eta' >0$ such that $exp(\eta' N(\delta)) < 2$, we obtain the desired estimate with $\hat \eta = \eta '$ and $\kappa = \kappa_0 /4$.

\end{proof}

Let $\mathcal I^{\epsilon}(\delta, \hat \delta)$ denote the collection of $\epsilon$-random orbits $\{x_j\}_{j=0}^n$ for which there exists $v \in CV$ such that $d(x_0, v) \leq 4 \delta$ and such that one of the following holds:

\begin{enumerate}
\item[(1)] either $x_{M_v(\delta)} \in \tilde B(\hat \delta)$ and $n = M_v(\delta)$;
 
\item[(2)] or $x_{M_v(\delta)} \notin \tilde B(\hat \delta), n >M_v(\delta)$ and $\{ x_j \}_{j=M_v(\delta)+1}^n \in \mathcal L^{\epsilon}(\hat \delta)$. 
\end{enumerate}
In the language of \cite{BC}, $n$ is the first {\it free return} of the random orbits $\{ x_j \}_{j=0}^n$ into $\tilde B(\hat \delta)$.

\begin{lem}
There exists a constant $\zeta_2>0$ such that the following holds. For each $\delta_0 > 0$ small enough, there exists $\epsilon_0>0$ such that for each $\{ f_{\omega}^j(x) \}_{j=0}^n \in \mathcal I^{\epsilon}(\delta, \delta_0)$ with $\delta \in (0, \delta_0]$ and $0 \leq \epsilon \leq \min \{ \epsilon_0, \delta\}$, we have
\begin{equation}
Df_{\omega}^n(x) \geq \frac{\Lambda_0(\delta)^{\zeta_2}}{D(\delta)}. 
\end{equation}
Moreover, if $x_n \notin \tilde B(\delta)$, then
\begin{equation}
Df_{\omega}^{n+1}(x) \geq \Lambda_0(\delta)^{\zeta_2} \left(\frac{d_*(x_n, c)}{\delta} \right)^{1 - \frac{1}{\ell}}. 
\end{equation}
\end{lem}

\begin{proof}

Let $\delta_0 > 0$ be a small constant such that for all $\delta \in (0, \delta_0]$ the conclusion of Proposition 8 holds. Let $\epsilon_0 = \epsilon(\delta_0)$ be the constant determined by Lemma 4.1. Let $\zeta_2 = \zeta_1 /2$. Assume that $\delta \in (0, \delta_0]$ and $0 \leq \epsilon \leq \min \{ \epsilon_0, \delta\}$, and consider $\{x_j \}_{j=0}^n = \{ f_{\omega}^j(x) \}_{j=0}^n \in \mathcal I^{\epsilon}(\delta, \delta_0)$. Let $v \in CV$ be such that $d(x_0, v)\leq 4 \delta$ and let $M = M_v(\delta)$.

If $x_M \in \tilde B(\delta_0)$ and $n =M$,  then by Proposition 8 the desired estimate holds. Assume that $x_M \notin \tilde B(\delta_0)$. Let $\delta' = d_*(x_M, c) \geq \delta_0$, let $\delta''= d(x_{M+1}, CV)$. Then there exists a constant $C_1 >0$ such that $\delta'' \leq C_1 \delta'$. Indeed, if $x_M \notin \tilde B(\delta_*)$, then $\delta'' \leq 1, \delta' \geq \delta_*$, so $C_1 = 1/ \delta_*$ is enough; otherwise, $\delta'' \leq \delta' + \epsilon \leq 2 \delta'$, so $C_1 = 2$ is enough. By Lemma 4.1 and (4.5) in Proposition 8, we have
\[
Df_{\omega}^n(x) = Df_{\omega}^{M+1}(x) \cdot \prod_{j=M+1}^{n-1} Df_{\sigma^j \omega} (x_j)  \geq  \frac{\kappa \Lambda_0(\delta)^{\zeta_1}}{D(\delta_0)}  \left( \frac{\delta'}{\delta} \cdot \frac{\delta_0}{\delta''} \right)^{1 - \frac{1}{\ell}} \geq \frac{C_2 \Lambda_0(\delta)^{\zeta_1}}{D(\delta_0)}\left( \frac{\delta_0}{\delta} \right)^{1 - \frac{1}{\ell}}, 
\]
where $C_2$ is a constant. Since $\delta_0 \geq \delta$, there exists a constant $C_3 > 0$ such that 
\[
Df_{\omega}^n(x)  \geq \frac{C_3 \Lambda_0(\delta)^{\zeta_1}}{D(\delta_0)}.
\]
Provided $\delta_0$ is small enough, $\Lambda_0(\delta)$ is large so that (4.9) holds. To prove (4.10), assume that $\rho : = d_*(x_n, c) \geq \delta$. Then $Df_{\sigma^n \omega}(x_n) \geq D(\rho)$. Since $\rho < \delta_0$, then there exists a constant $C_4>0$ such that 
\[
D f_{\omega}^{n+1} (x) = Df_{\sigma^n \omega}(x_n) \cdot D f_{\omega}^{n} (x) \geq C_4 \Lambda_0(\delta)^{\zeta_1} \left( \frac{\rho}{\delta} \right)^{1 - \frac{1}{\ell}}.
\]
Then (4.10) holds provided $\delta_0$ is small enough.

\end{proof}

We now give the proof of Proposition 7.

\begin{proof}[Proof of Proposition 7]

Let $\delta_0 >0$ be a small constant such that $\Lambda_0(\delta) > 1$ for all $\delta \in (0, \delta_0]$. Reducing $\delta_0$ if necessary, we may assume that there exists $\epsilon_0 > 0$ such that the conclusion of Lemma 4.2 holds. Consider $0 \leq \epsilon \leq \min\{\epsilon_0, \delta_0/2 \}$. Let $\{x_j \}_{j=0}^s = \{ f_{\omega}^j(x) \}_{j=0}^s $ be an $\epsilon$-random orbit with $d(x_0, v) \leq 4 \epsilon$ for some $v \in CV$, $x_j \notin \tilde B(\epsilon)$ for each $1 \leq j < s$ and $x_s \in \tilde B(2 \epsilon)$. We shall prove that 
\[
Df_{\omega}^s(x) \geq \frac{\Lambda_0(\epsilon)^{\zeta_2}}{D(\epsilon)}.
\]

Let $s_1$ be the minimal integer such that $s_1 \geq M_v(\epsilon)$ and $x_{s_1} \in \tilde B(\delta_0) $. Such $s_1$ exists because $s \geq M_v(\epsilon)$ by the definition of preferred binding period and $x_s \in \tilde B(\delta_0) $. Then $\{x_j \}_{j=0}^{s_1} \in \mathcal I^{\epsilon}(\epsilon, \delta_0)$.

If $s_1 = s$, then it follows by (4.9). Assume that $s_1 < s$. Then $\delta_1 = d_*(x_{s_1}, c) \geq \epsilon$. By (4.10), we have
\begin{equation}
Df_{\omega}^{s_1+1}(x) \geq \Lambda_0(\delta)^{\zeta_2} \left(\frac{\delta_1}{\epsilon} \right)^{1 - \frac{1}{\ell}}.  
\end{equation}
Now consider the orbit $\{ x_j\}_{j=s_1 + 1} ^s$, let $v_1 \in CV$ be the closet critical value to $x_{s_1+1}$. Let $s_2$ be the minimal integer such that $s_2 - (s_1 + 1) \geq M_{v_1}(\delta_1)$ and $x_{s_2} \in \tilde B (\delta_0)$. If $s_2 = s$, then stop. Otherwise, we define $v_2, \delta_2$ and $s_3$ similarly. The procedure ends when $s_k = s$. Then for each $i = 1, 2, \cdots, k-1$, $\{x_j \}_{j=s_i +1}^{s_{i+1}} \in \mathcal I^{\epsilon} (\delta_i, \delta_0)$. By (4.10), we have
\[
\prod_{j=s_i+1}^{s_{i+1}} Df_{\sigma^j \omega}(x_j) \geq \Lambda_0(\delta_i)^{\zeta_2} \left(\frac{\delta_{i+1}}{\delta_i} \right)^{1 - \frac{1}{\ell}}
\]
for all $i = 1, 2, \cdots, k-2$ and by (4.9), 
\[
\prod_{j=s_{k-1}+1}^{s-1} Df_{\sigma^j \omega}(x_j) \geq \frac{\Lambda_0(\delta_{k-1})^{\zeta_2}}{D(\epsilon)}.
\]
Combining these together, we have
\begin{align*}
Df_{\omega}^s (x)& = Df_{\omega}^{s_1+1} (x) \left( \prod_{i=1}^{k-2} \prod_{j = s_i+1}^{s_{i+1}} Df_{\sigma^j \omega}(x_j) \right) \prod_{j=s_{k-1}+1}^{s-1} Df_{\sigma^j \omega}(x_j)\\
& \geq \frac{\Lambda_0(\epsilon)^{\zeta_2}}{D(\epsilon)} \prod_{j=1}^{k-1} \Lambda_0(\delta_j)^{\zeta_2} \left( \frac{\delta_{k-1}}{\epsilon} \right)^{1 - \frac{1}{\ell}} \geq \frac{\Lambda_0(\epsilon)^{\zeta_2}}{D(\epsilon)}.
\end{align*}
This finishes the proof.

\end{proof}

\subsection{Exponential rate of expansion}

Let $\mathcal R^{\epsilon}(\delta)$ denote the collection of $\epsilon$-random orbits $\{ x_j \}_{j=0}^s = \{ f_{\omega}^j(x) \}_{j=0}^s$ for which $d(x_0, CV) \leq 4 \delta, x_j \notin \tilde B(\delta)$ for $1 \leq j < s$ and $x_s \in \tilde B(2 \delta)$. Let $\eta_0(\delta)$ be the maximal number in $[0, 1]$ such that for any $\{ x_j \}_{j=0}^s  \in \mathcal R^{\epsilon}(\delta)$ with $0 \leq \epsilon \leq \min\{ \epsilon_1, \delta \}$, we have
\begin{equation}
D f_{\omega}^s(x) \geq \frac{e}{D(\delta)} e^{\eta_0(\delta) s}.
\end{equation}

Let $\epsilon_0$ be a small constant such that Proposition 7 holds for all $\epsilon \in (0, \epsilon_0]$ with $\hat \Lambda(\epsilon) > 2e$. Let $\epsilon_1 = \epsilon(\epsilon_0)$ and $\hat \eta = \hat \eta(\epsilon_0)$ be constants determined by Lemma 4.1 for $\delta = \epsilon_0$. Replacing $\epsilon_1$ by a smaller constant if necessary, we assume that $\epsilon_1 < \epsilon_0$. 

For any orbit $\{ f_{\omega}^j(x) \}_{j=0}^s \in \mathcal R^{\epsilon}(\epsilon_0)$ with $0 \leq \epsilon \leq \epsilon_1$, $Df_{\omega}^s(x)$ is exponentially large in $s$. Combining with Proposition 7, we have
\begin{equation}
\eta_0(\epsilon_0) > 0. 
\end{equation}
Let $\kappa: (0, \epsilon_0] \to (0, 1)$ be a continuous function such that
\begin{itemize}
\item[(1)] $\kappa(\epsilon) \to 0$ as $\epsilon \to 0$,

\item[(2)] $\Lambda(\epsilon) : = \hat \Lambda(\epsilon)^{\kappa(\epsilon)} e^{1 - \kappa(\epsilon)} \geq 2e^2$ for all $\epsilon \in (0, \epsilon_0]$ and 

\item[(3)] $\hat \Lambda(\epsilon)^{\kappa(\epsilon)} \to \infty$ as $\epsilon \to 0$,
\end{itemize}
and let
\begin{equation}
\tilde \eta_0(\delta) : = (1 - \kappa(\delta)) \eta_0(\delta).
\end{equation}
Combining the estimate given by Proposition 7 and (4.12), we obtain that for each $\delta \in (0, \epsilon_0]$ and each $\{ f_{\omega}^j(x) \}_{j=0}^s \in \mathcal R^{\epsilon}(\delta)$,
\begin{equation}
Df_{\omega}^s(x) \geq \left( \frac{\hat \Lambda(\delta)}{D(\delta)}\right)^{\kappa(\delta)} \left(  \frac{e}{D(\delta)} e^{\eta_0(\delta) s}\right)^{1 - \kappa(\delta)} \geq \frac{\Lambda(\delta)}{D(\delta)} e^{\tilde \eta_0(\delta) s}.
\end{equation}

\begin{lem}
For each $\delta \in (0, \epsilon_0]$ and $\delta' \in [\delta/2, \delta)$, we have $\eta_0(\delta') \geq \tilde \eta_0(\delta)$.
\end{lem}

\begin{proof}

Given any random orbit $\{ f_{\omega}^j(x) \}_{j=0}^s \in \mathcal R^{\epsilon}(\delta')$ with $0 \leq \epsilon \leq \min\{ \epsilon_1, \delta' \}$, it suffices to prove
\begin{equation}
Df_{\omega}^s(x) \geq \frac{e}{D(\delta')}e^{\tilde \eta_0(\delta) s}.
\end{equation}

Let $1 \leq s_1 < s_2 < \cdots < s_k = s$ be all the positive integers such that $x_{s_i} \in \tilde B(2 \delta)$. Then for each $0 \leq i < k, \{ x_j\}_{j = s_i +1}^{s_{i+1}} \in \mathcal R^{\epsilon}(\delta)$, where we set $s_0 = -1$. By (4.15), we have that for each $0 \leq i < k$,
\[
D_i : = \prod_{j =s_i+1}^{s_{i+1} -1} Df_{\sigma^j \omega}(x_j) \geq \frac{2e^2}{D(\delta)} e^{\tilde \eta_0(\delta) (s_{i+1} -s_i -1)}.
\]
Thus,
\begin{align*}
Df_{\omega}^s(x) &= \prod_{i=0}^{k-1} D_i \prod_{i=1}^{k-1} Df_{\sigma^{s_i} \omega}(x_{s_i}) \geq \frac{(2e^2)^k}{D(\delta)} e^{\tilde \eta_0(\delta) (s-k)} \prod_{i=1}^{k-1} \frac{D f_{\sigma^{s_i} \omega}(x_{s_i})}{D(\delta)}\\
& \geq \frac{(2e^2)^k}{2D(\delta') } e^{\tilde \eta_0(\delta) (s-k)} \prod_{i=1}^{k-1} \frac{1}{2} \geq \frac{e^{2k}}{D(\delta') } e^{\tilde \eta_0(\delta) (s-k)} \geq \frac{e}{D(\delta')}e^{\tilde \eta_0(\delta) s}.
\end{align*}
Where we use the fact that $D(\delta) \leq 2 D(\delta')$, $Df_{\sigma^{s_i} \omega}(x_{s_i}) \geq D(\delta') \geq D(\delta)/2$ and $\tilde \eta_0(\delta) \leq 1$. This completes the proof.

\end{proof}

Now we will prove Theorem 2.

\begin{proof}[Proof of Theorem 2]

(1) Since Lemma 4.3 implies that $\log_{\epsilon} \eta_0(\epsilon) \to 0$ as $\epsilon \to 0$. Hence $\alpha(\epsilon) : = \log_{\epsilon} \tilde \eta_0(\epsilon) \to 0$ as $\epsilon \to 0$. By (4.15), statement (1) holds.

(2) Let $\epsilon_0, \epsilon_1, \hat \eta$ be as above in the beginning of this subsection. Let $\eta$ be the constant given by Proposition 6 for the neighborhood $U = \tilde B(\epsilon_0)$. For each $\delta \in (0, \epsilon_0]$, let $\eta_0(\delta)$ and $\tilde \eta_0(\delta)$ be as above and let 
\[
\eta(\delta) = \min\{ \hat \eta, \eta, \frac{1}{4}, \inf\{ \tilde \eta_0(\delta') : \delta' \in [ \delta, \epsilon_0] \} \}.
\]
Then $\alpha(\delta) : = \log_{\delta} \eta(\delta) \to 0$ as $\delta \to 0$.

Now let $\epsilon \in (0, \epsilon_1]$ be small and consider an $\epsilon$-random orbit $\{ x_j \}_{j=0}^s = \{ f_{\omega}^j(x) \}_{j=0}^s$ with $x_j \notin \tilde B(\epsilon)$ for all $0 \leq j < s$. Let $\rho_j = d_*(x_j, c)$, $\rho_j \geq \epsilon$, $0 \leq j < s$. By Proposition 6, statement (1), if $\rho_j \geq \epsilon_0$ for all $0 \leq j < s$, then the desired estimate holds.

So we assume the contrary. Without loss of generality, we may assume that $\rho_0 < \epsilon_0$ and $\rho_{s-1} < \epsilon_0$. If there exists $s' < s-1$ such that $\rho_{s'} < \rho_{s-1}$, then let $s'$ be the maximal integer with this property. Therefore, the orbit $\{ x_j \}_{j=s'+1}^{s-1} \in \mathcal R^{\epsilon}(\rho_{s-1})$. By (4.15),
\begin{align*}
\prod_{j = s'+1}^{s-1} Df_{\sigma^j \omega}(x_j) & = Df_{\sigma^{s-1} \omega}(x_{s-1}) \prod_{j = s'+1}^{s-2} Df_{\sigma^j \omega}(x_j)\\
& \geq Df_{\sigma^{s-1} \omega}(x_{s-1}) \frac{\Lambda(\rho_{s-1})}{D(\rho_{s-1})} e^{\eta(\epsilon) (s - s' -2)}\\
& = \frac{Df_{\sigma^{s-1} \omega}(x_{s-1})}{D(\rho_{s-1})} \Lambda(\rho_{s-1}) e^{\eta(\epsilon) (s - s' -2)}\\ 
& \geq \Lambda(\rho_{s-1}) e^{\eta(\epsilon) (s - s' -2)}\\
& \geq 2 e^2 \cdot e^{\eta(\epsilon) (s - s' -2)} > 2 e^{\eta(\epsilon) (s - s' )},
\end{align*}
where we use the fact that $Df_{\sigma^{s-1} \omega}(x_{s-1}) \geq D(\rho_{s-1})$, $\Lambda(\rho_{s-1}) \geq 2e^2$ and $\eta(\delta) \leq 1$.

Now it suffices to prove the desired estimate under the further assumption that $\rho_{s-1} \leq \rho_j$ for each $0 \leq j < s$. Let $s_0 < s_1 < \cdots < s_k = s-1$ be a sequence of integers such that $s_0 = 0$ and such that for each $0 \leq i < k$, $s_{k+1}$ is the minimal integer such that $\rho_{s_{i+1}} \leq \rho_{s_i}$. Then for $0 \leq i < k$, $\{ x_j \}_{j = s_i+1}^{s_{i+1}} \in \mathcal R^{\epsilon} (\rho_{s_i})$. So by (4.15) again, we have
\[
\prod_{j = s_i +1}^{s_{i+1} -1} Df_{\sigma^j \omega}(x_j) \geq \frac{\Lambda(\rho_{s_i})}{D(\rho_{s_i})} e^{\tilde \eta_0(\rho_{s_i}) (s_{i+1} -s_i -1)} \geq \frac{2e^{\eta(\delta) (s_{i+1} -s_i)} }{D(\rho_{s_i})}.
\]
Therefore,
\begin{align*}
Df_{\omega}^s(x)& \geq Df_{\omega}(x_0) \prod_{i = 1}^k \frac{2 Df_{\sigma^{s_i} \omega}(x_{s_i} )}{D(\rho_{s_{i-1}})} \cdot e^{\eta(\epsilon) (s -1)} \\
& \geq D(\rho_0) \prod_{i = 1}^k \frac{2 D(\rho_{s_i})}{D(\rho_{s_{i-1}})} \cdot e^{\eta(\epsilon) (s -1)}  \geq A e^{\eta(\epsilon) s} \rho_{s-1}^{1 - \frac{1}{\ell}},
\end{align*}
where $A>0$ is a constant. Since $\rho_{s-1} \geq \epsilon$, then the desired estimate holds.

\end{proof}

\subsection{More properties of return maps to $\tilde B(\epsilon)$}

We first prove the `small total distortion' result for iterates of random systems.

\begin{prop}
Let $\{f_t \}_{t \in [-1, 1]} $ be an admissible family with $f_0 =f \in \mathcal S_1$. For each $\epsilon > 0$ small enough there exists $\theta(\epsilon) \geq 0$ such that $\lim_{\epsilon \to 0} \theta(\epsilon) =0$ and such that the following holds. For $x \in I_c$ and $\omega \in \Omega_{\epsilon}$, if $n \geq 1$ is an integer such that $f_{\omega}^j(x) \notin \tilde B(\epsilon), 0 \leq j \leq n-1$, and $f_{\omega}^n(x) \in \tilde B(\epsilon)$, then
\begin{equation}
A(x, \omega, n) |\tilde B(\epsilon)| \leq \theta(\epsilon) Df_{\omega}^n(x).
\end{equation}
\end{prop}

\begin{proof}

Consider an orbit $\{f_{\omega}^j(x) \}_{j = 0}^{n} \in \mathcal L^{\epsilon}(\epsilon)$ for some $n \geq 1$. By Theorem 2 statement (2), the random system $f_{\omega}$ is uniformly expanding outside $\tilde B(\epsilon)$. Then
\[
\frac{A(x, \omega, n) }{Df_{\omega}^n(x)} \leq \frac{2}{\tilde B(\epsilon)} \sum_{i=0}^{n-1} \frac{1}{Df_{\sigma^i \omega}^{n-i}(f_{\omega}^i(x))} \leq \frac{2 \epsilon^{\frac{1}{\ell}-1}}{A|\tilde B(\epsilon)|} \sum_{i=0}^{n-1} \frac{1}{ e^{\epsilon^{\alpha(\epsilon)}(n-i)}}  \leq \frac{C'}{|\tilde B(\epsilon)|},
\]
the constant $C'$ depends only on $\epsilon$ and $\ell$. Therefore, given any $\epsilon > 0$, there exists a minimal non-negative integer $\theta(\epsilon)$ such that $(4.17)$ holds for each orbit in $\mathcal L^{\epsilon}(\epsilon)$. Using the same argument, for each $\epsilon_0 > 0$, $\theta(\epsilon)$ is bounded from above for $\epsilon \geq \epsilon_0$.

We aim to show that $\theta(\epsilon) \to 0$ as $\epsilon \to 0$. It suffices to prove that for $\epsilon > 0$ small enough, we have
\[
\theta(\epsilon /2) \leq \kappa (\theta(\epsilon) + \tau (\epsilon) )\mbox{ where \ } 
\kappa^2 = \sup_{\delta \in (0, \delta_*]} \frac{|\tilde B(\delta/2)|}{|\tilde B(\delta)|} < 1,
\]
and $\tau(\epsilon) \to 0$ as $\epsilon \to 0$.

Now consider an orbit $\{ f_{\omega}^j(x)\}_{j=0}^n \in \mathcal L^{\epsilon/2}(\epsilon/2)$. Let $0 \leq s_1 < s_2 < \cdots < s_m = n$ be all the integers such that $f_{\omega}^{s_i}(x) \in \tilde B(\epsilon)$. Let $\rho_i = d_*(f_{\omega}^{s_i}(x), c)$. We may assume that $\epsilon \leq \delta_*$, then $f_{\omega}^{s_i}(x) \in \tilde B(\epsilon) \subset \tilde B(\delta_*)$ and hence  $\rho_i = d (f_{\omega}^{s_i+1}(x), CV)$ for each $i = 1, 2, \cdots, m-1$. Moreover, $\rho_i \in [\epsilon/2, \epsilon]$. By (3.1), we have
\[
Df_{\sigma^{s_i} \omega}(f_{\omega}^{s_i}(x)) \geq D(d_*(f_{\omega}^{s_i}(x), c)) = D(\rho_i).
\]
Taking $y = f_{\omega}^{s_i+1}(x)$ and $k = s_{i+1} - s_i -1$, by the choice of $s_i$, we have that for $0 \leq j < k$,
\[
f_{\sigma^{s_i +1}\omega}^j (y) \notin \tilde B(\epsilon) \text{ while } f_{\sigma^{s_i +1}\omega}^k(y) = f_{\omega}^{s_{i+1}}(x) \in \tilde B(\epsilon). 
\]
Then the orbit $\{f_{\sigma^{s_i +1}\omega}^j (y) \}_{j=0}^k \in \mathcal L^{\epsilon}(\epsilon)$ with $d(y, CV) \leq 4\epsilon$. By Theorem 2 statement (1), we have  
\[
Df_{\sigma^{s_i +1}\omega}^k (y) = Df_{\sigma^{s_i +1}\omega}^{s_{i+1} - s_i -1}(f_{\omega}^{s_i+1}(x)) \geq \frac{\Lambda(\epsilon)}{D(\epsilon)} e^{\epsilon^{\alpha(\epsilon)}k} \geq \frac{\Lambda(\epsilon)}{D(\epsilon)}.
\]
By the chain rule,
\begin{align*}
Df_{\sigma^{s_i} \omega}^{s_{i+1} - s_i }(f_{\omega}^{s_i}(x)) & = Df_{\sigma^{s_i +1}\omega}^{s_{i+1} - s_i -1}(f_{\omega}^{s_i+1}(x)) \cdot Df_{\sigma^{s_i}\omega}(f^{s_i}(x)) \\ & 
\geq \frac{\Lambda(\epsilon) D(\rho_i)}{D(\epsilon)} 
= \frac{\Lambda(\epsilon) \rho_i}{\delta} \frac{|\tilde B(\epsilon)|}{|\tilde B(\rho_i)|} \geq \frac{\Lambda(\epsilon) }{2} \frac{|\tilde B(\epsilon)|}{|\tilde B(\rho_i)|},
\end{align*}
which implies 
\[
\frac{Df_{\omega}^{s_i}(x)}{|\tilde B(\epsilon)|} \leq \frac{Df_{\omega}^{s_i}(x)}{|\tilde B(\rho_i)|} \leq \frac{2}{\Lambda(\epsilon)} \frac{Df_{\omega}^{s_{i+1}}(x)}{|\tilde B(\epsilon)|}.
\]
Since $\rho_i \in [\epsilon/2, \epsilon]$, $d(f_{\omega}^{s_i}(x), c) \asymp |\tilde B(\epsilon)|$, there exists a universal constant $C>0$ (depending only on $\ell$) such that by iterating the above result, 
\begin{equation}
\frac{Df_{\omega}^{s_i}(x)}{d(f_{\omega}^{s_i}(x), c)} \leq C \frac{Df_{\omega}^{s_i}(x)}{|\tilde B(\epsilon)|} \leq C \tau_1(\epsilon)^{m-i} \frac{Df_{\omega}^{n}(x)}{|\tilde B(\epsilon)|}, 
\end{equation}
where $\tau_1(\epsilon) = 2/\Lambda(\epsilon)$.

Since $\{f_{\sigma^{s_i+1}\omega}^j(y) \}_{j=0}^k = \{ f_{\omega}^j(x) \}_{j=s_i + 1}^{s_{i+1}} \in \mathcal L^{\epsilon}(\epsilon)$, we have
\[
\sum_{j = s_i + 1}^{s_{i+1} - 1} \frac{Df_{\sigma^{s_i+1} \omega}^j(x)}{ d (f_{\omega}^j(x), c)} \leq \theta(\epsilon) \frac{Df_{\omega}^{s_{i+1}}(x)}{|\tilde B(\epsilon)|}.
\]
Similarly, if $s_1 \neq 0$, then
\[
\sum_{j = 0}^{s_1 - 1} \frac{Df_{\omega}^j(x)}{ d_{\omega} (f^j(x), c)} \leq \theta(\epsilon) \frac{Df^{s_1}(x)}{|\tilde B(\epsilon)|}.
\]
It follows that 
\[
A(x, \omega, n) \leq (1 + \theta(\epsilon)) \sum_{i=1}^{m-1} \frac{Df_{\omega}^{s_i}(x)}{d(f_{\omega}^{s_i}(x), c)} + \theta(\epsilon) \frac{Df_{\omega}^n(x)}{|\tilde B(\epsilon)|}.
\]
By (4.18) we have
\[
A(x, \omega, n) \leq  [ \tau(\epsilon) ( 1 + \theta(\epsilon)) + \theta(\epsilon) ] \frac{Df_{\omega}^n(x)}{|\tilde B(\epsilon)|} =  [ \tau(\epsilon) + (1 + \tau(\epsilon)) \theta(\epsilon) ] \frac{Df_{\omega}^n(x)}{|\tilde B(\epsilon)|},
\]
where
\[
\tau(\epsilon) = \frac{C \tau_1(\epsilon)}{1 - \tau_1(\epsilon)}.
\]
Note that for $\epsilon > 0$ small enough, we have $1 + \tau(\epsilon) < \kappa^{-1}$. Since $|\tilde B (\epsilon/2)| \leq \kappa^2 |\tilde B(\epsilon)|$, it follows that
\[
A(x, \omega, n) \leq \kappa (\theta(\epsilon) + \tau(\epsilon)) \frac{Df_{\omega}^n(x)}{|\tilde B(\epsilon/2)|}.
\]
This finishes the proof.

\end{proof}

\begin{prop}
Let $\{f_t \}_{t \in [-1, 1]} $ be an admissible family with $f_0 =f \in \mathcal S_1$. Given any $0 < \xi < \xi' \leq 2$, we have the following holds for each $\epsilon>0$ small enough. For any $\omega \in \Omega_{\epsilon}$ and any integer $s \geq 1$, if $W \subset I^{\pm}$ is an interval intersecting $\tilde B(\xi \epsilon)$ and $f_{\omega}^s(W) \subset \tilde B(2 \epsilon)$, then $W \subset \tilde B(\xi' \epsilon)$.
\end{prop}

\begin{proof}

This is sublemma 4.8 in \cite{LR}, the proof remains valid when using Theorem 2 instead of Proposition 3.1 therein.

\end{proof}

The following proposition provides us nice sets. The proof is similar to the deterministic case in \cite{BRSS} and \cite{CD}.

\begin{prop}
Let $\{f_t \}_{t \in [-1, 1]} $ be an admissible family with $f_0 =f \in \mathcal S_1$. If $0 < \epsilon \leq \delta$ are small enough, then there exists a nice set $V$ for $\epsilon$-random perturbations such that for each $\omega \in \Omega_{\epsilon}$, we have
\[
\tilde B(\delta) \subset V^{\omega} \subset \tilde B(2\delta).
\]
\end{prop}

\begin{proof}

Assume that $0 < \epsilon \leq \delta$ are small. By Proposition 10, for any $\omega \in \Omega_{\epsilon} \subset \Omega_{\delta}$, if $J$ is an interval intersecting $\tilde B(\delta)$ and $f_{\omega}^n(J) \subset \tilde B(2 \delta)$ for some integer $n \geq 1$, then $J \subset \tilde B(2 \delta)$.

For each $\omega \in \Omega_{\epsilon}$ and $n \geq 0$, let $V^{\omega}(n)$ denote the component of $\bigcup_{i=0}^n f_{\omega}^{-i}(\tilde B(\delta))$ containing $c$. Let $V^{\omega} = \bigcup_{n=0}^{\infty} V^{\omega}(n)$. It is easy to check that $V = \bigcup_{\omega \in \Omega_{\epsilon}} V^{\omega} \times \{ \omega\}$ is a nice set for $\epsilon$-random perturbations. It remains to show that for each $n \geq 0$ and $\omega \in \Omega_{\epsilon}$, we have
\[
\tilde B(\delta) \subset V^{\omega}(n) \subset \tilde B(2 \delta). 
\]

We prove this by induction on $n$. The case $n=0$ is trivial. Assume that the statement holds for some integer $n \geq 0$. Fix $\omega \in \Omega_{\epsilon}$. To show that $V^{\omega}(n+1) \subset \tilde B(2 \delta)$, it suffices to show that each component $J$ of  $V^{\omega}(n+1) \setminus \tilde B( \delta)$ is contained in $\tilde B(2 \delta)$. To this end, let $m \in \{0, 1, \cdots, n \}$ be minimal such that $f_{\omega}^{m+1}(J) \cap \tilde B(\delta) \neq \emptyset$. Then we have 
\[
f_{\omega}^{m+1}(J) \subset V^{\sigma^{m+1} \omega}(n-m).
\]
By induction hypothesis, this implies that $f_{\omega}^{m+1}(J) \subset \tilde B(2 \delta)$, hence $J \subset \tilde B(2 \delta)$. This completes the induction step and the proof is finished.

\end{proof}

\section{Slow recurrence of random orbits}

From now on, unless otherwise stated, let $\{f_t \}_{t \in [-1, 1]} $ be an admissible family with $f_0 =f \in \mathcal S_1$. For each $\epsilon > 0$ small, $\nu_{\epsilon}$ is a probability measure on $[-\epsilon, \epsilon]$ which belongs to the class $\mathbb M_{\epsilon}(L)$, where $L > 1$ is a fixed constant. Recall that $\Omega = [-1, 1]^{\mathbb N}, \Omega_{\epsilon} = [-\epsilon, \epsilon]^{\mathbb N}$ and $P_{\epsilon} = {\rm Leb}|_{[0, 1]} \times \nu_{\epsilon}^{\mathbb N}$. Let $F : I \times \Omega \to I \times \Omega$ denote the skew-product map:
\[
(x, \omega) \to (f_{\omega}(x), \sigma \omega).
\]

Let $\theta_0 > 0$ be a small constant determined by Lemma 3.5. For each $x \in I_c, \omega \in \Omega$ and $n \geq 1$, let  
\begin{equation}
\hat J^{\omega}_{x, n} = \left[ x - \frac{\theta_0}{A(x, \omega, n)},  x + \frac{\theta_0}{A(x, \omega, n)} \right] \text{ and } J^{\omega}_{x, n} = \hat J^{\omega}_{x, n} \cap I_c. 
\end{equation}
Then $f_{\omega}^n$ maps $J^{\omega}_{x, n}$ diffeomorphically onto its image with $\mathcal N(f_{\omega}^n | J^{\omega}_{x, n}) \leq 1$. Note that for $0 < \epsilon \leq \delta$ small enough, if $x \in \tilde B(\delta)$ and $\omega \in \Omega_{\epsilon}$, then $\hat J^{\omega}_{x, n} = J^{\omega}_{x, n} \subset I^{\pm}$. This is because each component of $\hat J^{\omega}_{x, n} \setminus \{ x \}$ has length $\theta_0/A(x, \omega, n) \leq \theta_0 d(x, c) \leq d(x, c)$.

\begin{definition}
We say that an integer $s \geq 1$ is a $\theta$-good return time of $(x, \omega)$ into $\tilde B(\delta) \times \Omega$ (resp. $\tilde B(\delta) \times \Omega_{\epsilon}$) if $f_{\omega}^s (x) \in \tilde B(\delta)$ and such that 
\begin{equation}
\theta Df_{\omega}^s (x) \geq A(x, \omega, s) |\tilde B(\delta)|. 
\end{equation}
We say that a positive integer $s$ is a {\it $\tau$-scale expansion time of $(x, \omega)$} if
\[
\theta_0 Df_{\omega}^s (x) \geq e \tau A(x, \omega, s).
\]
\end{definition}

\begin{lem}

Assume that $s$ is a $\theta$-good return time of $(x, \omega)$ into $\tilde B(\delta) \times \Omega$ with $ \delta \in (0,  \delta_{*}]$ and such that $J_{x, s}^{\omega} = \hat J_{x, s}^{\omega}$.  
\begin{itemize}
\item[(1)] If $\theta \leq \theta_0 /e$, then $f_{\omega}^s (J^{\omega}_{x, s})$ contains $\tilde B(\delta)$.

\item[(2)] If $\theta \leq \theta_0 /(e\tilde \kappa)$ where $\tilde \kappa : = \sup_{\delta \in (0, \delta_*]} |\tilde B(2 \delta)|/ |\tilde B(\delta)|$, then $f_{\omega}^s (J^{\omega}_{x, s})$ contains $\tilde B(2\delta)$.
\end{itemize}
In particular, $\tilde \kappa \asymp 2^{1/\ell} $ is independent of $\delta$ and $\omega$. 
\end{lem}

\begin{proof}
Let $\mathcal J$ be any component of $J^{\omega}_{x, s} \setminus \{ x\}$. By Lemma 3.5, we have
\[
|f_{\omega}^s (\mathcal J) | \geq \frac{Df_{\omega}^s (x)}{e} \cdot \frac{\theta_0}{A(x, \omega, n)} \geq \frac{\theta_0}{\theta e} |\tilde B(\delta)|. 
\]
Since $f_{\omega}^s(x) \in \tilde B(\delta)$, the lemma holds.

\end{proof}

We shall use the following notations:
\begin{align}
h^{\theta}_{\delta}(x, \omega)& = \inf \{s \geq 1: s \text{ is a $\theta$-good return time of $(x, \omega)$ into $\tilde B(\delta) \times \Omega$} \}, \\
T^{\tau}(x, \omega) & = \inf \{s \geq 1: s \text{ is a $\tau$-scale expansion time of $(x, \omega)$} \}, \\
\hat h^{\theta}_{\delta, \tau} (x, \omega) & = \min  \{ \inf_{\delta' \geq \delta} h^{\theta}_{\delta'}(x, \omega), T_{\tau}(x, \omega) \},
\end{align}
and
\begin{equation}
l_{\delta} (x, \omega) = \inf \{s \geq 0 : f_{\omega}^s(x) \in \tilde B(\delta) \}.
\end{equation}
The following is an easy consequence of Proposition 9.

\begin{lem}
Given $\theta >0$ there exists $\delta_0 > 0$ such that for $x \in I_c \setminus \tilde B(\delta_0)$ and $\omega \in \Omega_{\epsilon}$ with $\epsilon \in (0, \delta_0]$, we have $h^{\theta}_{\delta_0}(x, \omega)= l_{\delta_0} (x, \omega)$.
\end{lem}

\begin{proof}

By definition, $h^{\theta}_{\delta_0}(x, \omega) \geq  l_{\delta_0} (x, \omega)$. By Proposition 9, $l_{\delta_0} (x, \omega)$ is either infinity or a $\theta$-good return time of $(x, \omega)$ into $\tilde B(\delta_0) \times \Omega$ provided $\delta_0$ is small enough. Thus, $h^{\theta}_{\delta_0}(x, \omega) \leq  l_{\delta_0} (x, \omega)$. This finishes the proof.

\end{proof}

We shall prove the following two propositions in subsection 5.3.

\begin{prop}
Given $\theta > 0, p \geq 1$ and $\gamma > 0$, there exists $\tau > 0$ such that
\[
\frac{1}{|\tilde B(\epsilon)|} \iint_{\tilde B(\epsilon) \times \Omega_{\epsilon}} (\hat h^{\theta}_{\epsilon, \tau} (x, \omega) )^p dP_{\epsilon} < \epsilon^{- \gamma},
\]
provided $\epsilon > 0$ is small enough.
\end{prop}

\begin{prop}
Given $\theta >0, \alpha > 0$ and $0<b < 1$, there exists $\tau = \tau(\theta, \alpha, b) > 0$ and $m = m(b) \geq 1$ is an integer such that the following holds provided $\epsilon>0$ is small enough. For each $x \in \tilde B(\epsilon) \setminus \tilde B(b\epsilon) $ and $\omega  \in \Omega_{\epsilon}$, we have $\hat h^{\theta}_{\epsilon, \tau} (x, \omega)  \leq m \epsilon^{-\alpha}$.
\end{prop}

\subsection{Bad return estimate}

\begin{definition}
For $x \in I_c$ and $\omega = (\omega_0, \omega_1, \cdots) \in \Omega_{\epsilon}$, define the depth function:
\[
q_{\epsilon}(x, \omega) =  \inf \{q \in \mathbb N : Df_{\omega}(x) d(x, c) \geq e^{-q} \epsilon \}.
\]
For non-negative integers $0 \leq n_1 \leq n_2$, define
\[
Q_{n_1}^{n_2} (x, \omega, \epsilon) = \sum_{j=n_1}^{n_2} q_{\epsilon}(F^j(x, \omega)),
\]
and
\[
\Gamma_{n_1}^{n_2} (x, \omega, \epsilon) = \# \{ n_1 \leq j \leq n_2 : f_{\omega}^j(x) \in \tilde B(\epsilon) \}.
\]
For $\kappa > 0$ and integer $m \geq 0$, let $Bad_m(\kappa, \epsilon)$ be the collection of $(x, \omega) \in I_c \times \Omega_{\epsilon}$ such that the following holds:
\begin{itemize}
\item[(1)] $Q_0^s (x, \omega, \epsilon) > \min\{ m, \kappa \Gamma_0^s(x, \omega, \epsilon) \}$ for each integer $s \geq 0$;
\item[(2)] $\lim_{s \to \infty} Q_0^s (x, \omega, \epsilon) \geq m$.
\end{itemize}
Finally, let $Bad^c_m(\kappa, \epsilon) = \{ (x, \omega) \in Bad_m(\kappa, \epsilon) : x \in \tilde B(\epsilon) \}$.
\end{definition}

Note that by non-flatness (C4), we have that if $x \in \tilde B(\epsilon)$, 
\[
Df_{\omega} (x) \geq O_1 \left( \frac{e^{-q_{\epsilon}(x, \omega)}\epsilon}{O_2} \right)^{1 - \frac{1}{\ell}} \mbox{ and } d(x, c) \geq \left( \frac{e^{-q_{\epsilon}(x, \omega)}\epsilon}{O_2} \right)^{\frac{1}{\ell}}.
\]

The aim of this subsection is to prove the following proposition.

\begin{prop}
There exist $\kappa >1, K > 0$ and $\rho>0$ such that if $\epsilon > 0$ small enough, then for each integer $m \geq 0$, we have
\[
P_{\epsilon} (Bad^c_m(\kappa, \epsilon)) \leq K e^{- \rho m} |\tilde B(\epsilon)|.
\]
\end{prop}

To prove this proposition, we need a few preparations.

\begin{lem}
For each $\epsilon > 0$ small enough the following holds: For each $\omega \in \Omega_{\epsilon}$ and $x \in I_c$ with $d(x, CV) \leq 4 \epsilon$, if $n: = l_{\epsilon} (x, \omega) < \infty$ and $J$ is the component of $f_{\omega}^{-n} (\tilde B(\epsilon))$ which contains $x$, then $f_{\omega}^n |J$ is a diffeomorphism onto its image, $\mathcal N(f_{\omega}^n |J) \leq 1$ and $|J| < \epsilon$.
\end{lem}

\begin{proof}

Let $\theta = \theta_0 /e$. By Lemma 5.2, $h_{\epsilon}^{\theta} (x, \omega) = l_{\epsilon} (x, \omega)  = n$ provided $\epsilon >0$ small enough. Let $T = J_{x, n}^{\omega}$ be defined as in (5.1). Then $f_{\omega}^n |T$ is a diffeomorphism onto its image with $\mathcal N(f_{\omega}^n |T) \leq 1$. Since $n$ is a $\theta$-good return time, by Lemma 5.1 we have $f_{\omega}^s (T)  \supset  \tilde B(\epsilon)$. Thus $J \subset T$, so $\mathcal N(f_{\omega}^n | J) \leq 1$. Since $d(x, CV) \leq 4 \epsilon$, by Theorem 2 statement (1), $Df_{\omega}^n(x) > \frac{e}{D(\epsilon)}$ provided $\epsilon$ small enough. Then
\[
|J| \leq \frac{e}{Df_{\omega}^n(x)} |\tilde B(\epsilon)| < D(\epsilon) |\tilde B(\epsilon)| = \epsilon.
\]

\end{proof}

Let $\mathbb F_{\epsilon}$ denote the first entry map into the region $\tilde B(\epsilon) \times \Omega_{\epsilon}$ under $F$, that is 
\[
\mathbb F_{\epsilon} (x, \omega) = F^{R_{\epsilon}(x, \omega)}(x, \omega)
\]
where 
\[
R_{\epsilon}(x, \omega) = \begin{cases}
l_{\epsilon}(x, \omega), & \mbox{ if } x \notin \tilde B(\epsilon),\\
l_{\epsilon}(F(x, \omega)), & \mbox{ if } x \in \tilde B(\epsilon).
\end{cases}
\]
Note that $\mathbb F_{\epsilon}$ is defined on a subset of $I_c \times \Omega_{\epsilon}$.

To complete the proof, we shall need the following result which was proved in \cite{LR}.

For $n \geq 1$ and a vector $\vec q = (q_1, q_2, \cdots, q_n) \in \mathbb N^n$,  let $|\vec q| = \sum_{i=1}^n q_i$. For each $x \in I_c$, denote
\[
U_{\epsilon}^n (x, \vec q) = \{ \omega \in \Omega_{\epsilon} : q_{\epsilon} (\mathbb F^i_{\epsilon}(x, \omega)) \geq q_i, i =1, 2, \cdots, n \}.
\]

\begin{lem}[Lemma 4.9 \& 4.10, \cite{LR}]

There exist $K_0 > 0$ and $\rho_0 > 0$ such that for each $\epsilon > 0$ small enough the following holds. For each $x \in \tilde B(\epsilon)$ and any $\vec q  \in \mathbb N^n, n \geq 1$, we have 
\[
\nu_{\epsilon}^{\mathbb N} (U_{\epsilon}^n (x, \vec q)) \leq K_0^n e^{- \rho_0 |\vec q|}.
\]
\end{lem}

\begin{proof}[Proof of Proposition 14]

Let $K_0, \rho_0$ be given by Lemma 5.4, let $\rho= \min\{ \rho_0/5, 1/(2\ell) \}$. By the Stirling's formula, there exists $\kappa > 1$ such that if $m, n \in \mathbb N^+$ with $m > \kappa n/2$, then $C_{m+n-1}^{m-1} \leq e^{\rho m}$. Replacing $\kappa$ by a larger constant, we may assume that $K_0 \leq e^{\kappa \rho}$.

Fix $m \geq 0$. Let 
\[
\Delta_0 = \{(x, \omega) : x \in \tilde B(\epsilon), q_{\epsilon}(x, \omega) \geq \frac{m}{2} - \kappa \}.
\]
For non-negative integers $m', n$, let
\[
\Delta_n^{m'} = \{ (x, \omega) : x \in \tilde B(\epsilon) \times \Omega_{\epsilon}: \exists s \geq 1 \mbox{ such that }  \Gamma_1^s(x, \omega, \epsilon) =n, Q_1^s(x, \omega, \epsilon) = m' \}.
\]
Put $\mathcal I = \{ (m', n) \in \mathbb N^2 : 2 m' \geq \max\{ m, \kappa n \}>0 \}$.

\begin{claim}
\begin{equation}
Bad_m^c(\kappa, \epsilon) \subset \Delta_0 \cup \left( \bigcup_{(m', n) \in \mathcal I}\Delta_n^{m'}  \right).
\end{equation}
\end{claim}

First, we show that for each $(x, \omega) \in Bad_m^c(\kappa, \epsilon) $, there exists an integer $s \geq 0$ such that
\begin{equation}
Q_0^s(x, \omega, \epsilon)  > \max\{m - \kappa, \kappa \Gamma_0^s(x, \omega, \epsilon)  \}.
\end{equation}
Let $s_0$ be minimal such that $Q_0^{s_0}(x, \omega, \epsilon) \geq m$, such $s_0$ always exists. If $Q_0^{s_0}(x, \omega, \epsilon) > \kappa \Gamma_0^{s_0}(x, \omega, \epsilon)$, then set $s = s_0$. Otherwise, take $s < s_0$ be the maximal integer such that $f_{\omega}^s(x) \in \tilde B(\epsilon)$. Since $Q_0^s(x, \omega, \epsilon) < m$, then $Q_0^s(x, \omega, \epsilon) > \kappa \Gamma_0^s(x, \omega, \epsilon)$. Since by assumption, $Q_0^{s_0}(x, \omega, \epsilon) \leq \kappa \Gamma_0^{s_0}(x, \omega, \epsilon)$ and also $ \Gamma_0^{s_0}(x, \omega, \epsilon) =  \Gamma_0^{s}(x, \omega, \epsilon) + 1$, we have
\[
Q_0^s(x, \omega, \epsilon)  > \kappa \Gamma_0^{s_0}(x, \omega, \epsilon) - \kappa \geq Q_0^{s_0}(x, \omega, \epsilon) - \kappa \geq m - \kappa.
\]
This proves (5.8).

Second, for each $(x, \omega) \in Bad_m^c(\kappa, \epsilon) \setminus \Delta_0$, assume that $(x, \omega) \in \Delta_n^{m'}$ where $m' = Q_1^s(x, \omega, \epsilon)$, $n = \Gamma_1^s(x, \omega, s)$ and $s \geq 0$ so that (5.8) holds. It suffices to show that $(m', n) \in \mathcal I$, that is, $2m' \geq m $ and $2m' \geq \kappa n$.

On one hand, 
\[
2m' = 2(Q_0^s(x, \omega, \epsilon) - q_{\epsilon}(x, \omega)) > 2((m-\kappa) - (\frac{m}{2} - \kappa)) =m.
\]
Note that this implies $n > 0$. On the other hand, since $2 q_{\epsilon}(x, \omega) \leq m - 2 \kappa \leq m - \kappa < Q_0^s(x, \omega, \epsilon)$, we have 
\[
\kappa n \leq \kappa \Gamma_0^{s}(x, \omega, \epsilon)  <Q_0^{s}(x, \omega, \epsilon)  =Q_1^{s}(x, \omega, \epsilon) + q_{\epsilon}(x, \omega) \leq 2 Q_1^{s}(x, \omega, \epsilon) = 2 m'.
\]
This proves the claim.

To complete the proof, we separate into two steps.

\noindent
{\bf Step 1.} Estimate of $P_{\epsilon} (\Delta_0)$. We first observe that by the definition of $q_{\epsilon}$, there exists a constant $C = C(\kappa)$ such that for each $\omega \in \Omega_{\epsilon}$, we have
\[
|\{ x \in \tilde B(\epsilon) : q_{\epsilon} (x, \omega) \geq \frac{m}{2} - \kappa \} | \leq C | \tilde B(\epsilon)| e^{-\frac{m}{2 \ell}} \leq C | \tilde B(\epsilon)| e^{-\rho m}.
\]
Indeed, for each $x \in \tilde B(\epsilon)$ with $q_{\epsilon} (x, \omega) \geq m/2 - \kappa$, we have
\[
e^{-(q_{\epsilon}(x, \omega) -1)}\epsilon \geq Df_{\omega}(x) d (x, c) \asymp d(x, c)^{\ell},  
\]
hence
\[
d(x, c)\leq C_1 e^{- \frac{q_{\epsilon}(x, \omega)}{\ell}} \epsilon^{\frac{1}{\ell}} \leq C_2 e^{\frac{\kappa}{\ell}}e^{- \frac{m}{2\ell}} |\tilde B(\epsilon)|.
\]
By Fubini,
\[
P_{\epsilon} (\Delta_0) \leq C | \tilde B(\epsilon)| e^{-\rho m}.
\]

\noindent
{\bf Step 2.} Estimate of $P_{\epsilon} (\Delta_n^{m'})$ for $(m', n) \in \mathcal I$. For each $x \in \tilde B(\epsilon)$, let 
\[
E_n^{m'}(x, \epsilon)  = \{\omega \in \Omega_{\epsilon} : (x ,\omega) \in \Delta_n^{m'} \}.
\]
Then 
\[
E_n^{m'}(x, \epsilon) \subset \bigcup_{\substack{\vec q \in \mathbb N^n\\ |\vec q| = m'} } U_{\epsilon}^n (x, \vec q).
\]
Since the number of $\vec q \in \mathbb N^n$ with $\vec q = m'$ is $C_{m'+n-1}^{n-1}$, by Lemma 5.4, 
\[
\nu_{\epsilon}^{\mathbb N} (E_n^{m'}(x, \epsilon)) \leq C_{m'+n-1}^{n-1} e^{-\rho_0 m'} K_0^n \leq e^{\rho m'} e^{-\rho_0 m'} e^{\kappa \rho n} \leq e^{- \rho(4m' - \kappa n)} \leq e^{-2 \rho m'},
\]
here we use the fact that $m' > \kappa n /2$. By Fubini,
\[
P_{\epsilon} (\Delta_n^{m'}) = \int_{\tilde B(\epsilon)} \nu_{\epsilon}^{\mathbb N} (E_n^{m'}(x, \epsilon))dx \leq e^{-2\rho m'} |\tilde B(\epsilon)|.
\]
Thus,
\begin{align*}
\sum_{(m', n) \in \mathcal I} P_{\epsilon} (\Delta_n^{m'})  & \leq \sum_{m' = [\frac{m}{2}]}^{\infty} \sum_{n: (m', n) \in \mathcal I} P_{\epsilon} (\Delta_n^{m'})  \\
&\leq \sum_{m' = [\frac{m}{2}]}^{\infty}  \frac{2m'}{\kappa} e^{-2\rho m'} |\tilde B(\epsilon)| \leq C' e^{-\rho m} |\tilde B(\epsilon)|.
\end{align*}
Combining with (5.7), we obtain the desired estimates.

\end{proof}

\subsection{Good return estimate}

The main result of this subsection is the following.

\begin{prop}
Given $\theta > 0, \kappa > 1$ and $\alpha > 0$, there exists $\tau > 0$ such that the following holds when $\epsilon > 0$ small enough. For $(x, \omega) \in \tilde B(\epsilon) \times \Omega_{\epsilon}$ and $m \geq 1$, if $(x, \omega) \notin Bad_m(\kappa, \epsilon)$, then $\hat h_{\epsilon, \tau}^{\theta} (x, \omega) \leq m \epsilon^{-\alpha}$.
\end{prop}

We shall need several preparations.

\begin{lem}
Given $K>0$ and $\beta >0$, the following holds for each $(y, \overline \omega) \in \tilde B(\delta) \times \Omega_{\delta}$ provided $\delta> 0$ is small enough.

(1) If $t \geq 1$ is an integer such that $f_{\overline \omega}^t (y) \in \tilde B(\delta)$, then
\[
\log \left( \frac{Df_{\overline \omega}^t (y) d(y, c) }{|\tilde B(\delta)|} \right) \geq K \Gamma_0^{t-1}(y, \overline \omega, \delta) - Q_0^{t-1}(y, \overline \omega, \delta) + \delta^{\beta} t.
\]

(2) If $t \geq 1$ is an integer such that $f_{\overline \omega}^t (y) \notin \tilde B(\delta)$, then
\[
\log \left( \frac{Df_{\overline \omega}^t (y) d(y, c) }{d(f_{\overline \omega}^t (y), c)} \right) \geq K \Gamma_0^{t-1}(y, \overline \omega, \delta) - Q_0^{t-1}(y, \overline \omega, \delta) + \delta^{\beta} t + 2 \log \delta.
\]
\end{lem}

\begin{proof}

(1) It suffices to consider the case $f_{\overline \omega}^j (y) \notin \tilde B(\delta)$ for $1 \leq j < t$, since the general case follows by induction on $\Gamma_0^{t-1}(y, \overline \omega, \delta)$.

Let $x = f_{\overline \omega_0}(y)$. By Theorem 2 statement (1), we have
\[
D f_{\sigma \overline \omega}^{t-1}(x) \geq \frac{e^K}{D(\delta)} e^{\delta^{\beta} t},
\]
provided $\delta>0$ is small enough such that $\alpha(\delta) \leq \beta$. Therefore,
\begin{align*}
\frac{Df_{\overline \omega}^t (y) d(y, c) }{|\tilde B(\delta)|} &= D f_{\sigma \overline \omega}^{t-1}(x) \frac{Df_{\overline \omega_0} (y) d(y, c) }{|\tilde B(\delta)|} \geq D f_{\sigma \overline \omega}^{t-1}(x) \frac{e^{-q_{\delta} (y, \overline \omega)} \delta}{|\tilde B(\delta)|}  \\
& = D f_{\sigma \overline \omega}^{t-1}(x) e^{-q_{\delta} (y, \overline \omega)} D(\delta) \geq e^Ke^{-q_{\delta} (y, \overline \omega)} e^{\delta^{\beta} t}.
\end{align*}
This finishes the proof.

(2) Put $\rho_j = d_* (f_{\overline \omega}^j(y), c)$ for $0 \leq j \leq t$. By part (1) of this lemma, it suffices to consider the case that $\rho_j \geq \delta$ for all $1 \leq j \leq t$. By Theorem 2 statement (2), 
 \begin{align*}
 Df_{\overline \omega}^t (y) d(y, c)  & = D f_{\sigma \overline \omega}^{t-1}(x) Df_{\overline \omega_0} (y) d(y, c) \geq A \delta^{1 - \frac{1}{\ell}} e^{\delta^{\beta} t} e^{-q_{\delta} (y, \overline \omega)} \delta \\
 & > \delta^2 \frac{A}{ \delta^{ \frac{1}{\ell}}} e^{-q_{\delta} (y, \overline \omega)}e^{\delta^{\beta} t} > \delta^2 e^{K-q_{\delta} (y, \overline \omega)}e^{\delta^{\beta} t},
 \end{align*}
 provided $\delta > 0$ small enough. Then we obtain the desired estimate since $d(f_{\overline \omega}^t (y), c) \leq 1$.

\end{proof}

\begin{lem}
Given $\kappa > 1$ and $\theta > 0$ we have the following provided $\epsilon > 0$ is small enough. Let $(x, \omega) \in \tilde B(\epsilon) \times \Omega_{\epsilon}$, let $s \geq 1$ be an integer such that $f_{\omega}^s (x) \in \tilde B(\epsilon)$ and such that for each $0 \leq j < s$,
\[
\kappa \Gamma_j^{s-1}(x,  \omega, \epsilon) \geq Q_j^{s-1}(x, \omega, \epsilon).
\]
Then $s$ is a $\theta$-good return time of $(x, \omega)$ into $\tilde B(\epsilon) \times \Omega_{\epsilon}$.
\end{lem}

\begin{proof}

Let $0 = s_0 < s_1 < \cdots < s_n = s$ be all the integers such that $f_{\omega}^{s_i} (x) \in \tilde B(\epsilon)$. For each $0 \leq i \leq n$, let 
\[
A_i = \frac{Df_{\omega}^{s_i} (x) }{d(f_{\omega}^{s_i} (x), c)} \mbox{ and } \tilde A_i = \frac{Df_{\omega}^{s_i} (x) }{|\tilde B(\epsilon)|}.
\]
It suffices to show that 
\[
\theta \tilde A_n \geq A(x, \omega, s).
\]
By Proposition 9, we have
\begin{align*}
A(x, \omega, s) & = \sum_{i=0}^{s-1} \frac{Df_{\omega}^i (x) }{d(f_{\omega}^i (x), c)} =  \sum_{i=0}^{n-1} A_n + \sum_{i=1}^{n} \sum_{i = s_{i-1}+1}^{s_i-1}  \frac{Df_{\omega}^i (x) }{d(f_{\omega}^i (x), c)} \\
& =\sum_{i=0}^{n-1} A_n + \sum_{i=1}^{n} \sum_{i = s_{i-1}+1}^{s_i-1}  \frac{Df_{\sigma^{s_{i-1}+1}\omega}^{i-(s_{i-1} +1)} (x) }{d(f_{\sigma^{s_{i-1}+1}\omega}^{i-(s_{i-1} +1)} (f_{\omega}^{s_{i-1} +1}(x)), c)}  \cdot Df_{\omega}^{s_{i-1} +1} (x) \\
& \leq  \sum_{i=0}^{n-1} A_n  + \theta(\epsilon) \sum_{i=1}^{n}  \tilde A_n \leq A_0 + (1 + \theta (\epsilon)) \sum_{i=1}^{n-1} A_n + \theta(\epsilon) \tilde A_n,
\end{align*}
where we use the fact that since $f_{\omega}^{s_i} (x) \in \tilde B(\epsilon)$, then $|\tilde B(\epsilon)| \geq d(f_{\omega}^{s_i} (x), c)$ and hence $\tilde A_i \leq A_i$, $1 \leq i \leq n$. Moreover, $\theta(\epsilon) \to 0$ as $\epsilon \to 0$.

Let $K_0$ be a large constant and $\epsilon > 0$ small. By Lemma 5.5 statement (1), for each $i = 0, 1, \cdots, n-1$, 
\begin{align*}
\log \frac{\tilde A_n}{A_i} &= \log \left( \frac{Df_{\omega}^{s} (x) }{|\tilde B(\epsilon)|} \frac{d(f_{\omega}^{s_i} (x), c)}{Df_{\omega}^{s_i} (x) } \right)  = \log \left( \frac{  Df_{\sigma^{s_i} \omega}^{s - s_i} ( f_{\omega}^{s_i} (x)) d(f_{\omega}^{s_i} (x), c) }{|\tilde B(\epsilon)|} \right) \\
& \geq (K_0 + \kappa) \Gamma_{s_i}^{s-1} (x, \omega, \epsilon) - Q_{s_i}^{s-1} (x, \omega, \epsilon) \geq (n - i) K_0.
\end{align*}
Hence $A_i \leq e^{-(n-i) K_0} \tilde A_n$. Therefore
\[
A(x, \omega, s) \leq  \left( e^{- K_0 n} +(1+\theta(\epsilon)) \sum_{i=1}^{n-1} e^{-K_0(n-i)} + \theta(\epsilon)  \right) \tilde A_n \leq \frac{\tilde A_n}{\theta}
\]
provided $\epsilon > 0$ is small enough.

\end{proof}

\begin{lem}
Given $\theta > 0$ and $\gamma > 0$ there exists a constant $\tau > 0$ such that the following holds provided $\epsilon > 0$ is small enough. Let $(x, \omega) \in \tilde B(\epsilon) \times \Omega_{\epsilon}$, let $s \geq 1$ be an integer such that for each $0 \leq j < s$, 
\begin{equation}
s - j > \epsilon^{- \gamma} Q_{j}^{s-1} (x, \omega, \epsilon),
\end{equation}
then $\hat h_{\epsilon, \tau}^{\theta} (x, \omega) \leq s$.
\end{lem}

\begin{proof}

Let $\beta = \gamma/4$. Let $\epsilon_0 > 0$ be small such that Lemma 5.5 holds for all $\delta \in (0, \epsilon_0]$. In the following we may assume that $\epsilon \in (0, \epsilon_0 /e]$ small. Let $N$ be the maximal positive integer such that $e^{N-1} \epsilon \leq \epsilon_0$, then 
\[
2 \leq N \leq \log \frac{\epsilon_0}{\epsilon} + 1 < \epsilon^{- \beta} \mbox{ provided $\epsilon > 0$  is small enough}.
\]

Let $s_0 = s$, for $i = 1, 2, \cdots, N$, define 
\[
s_i = \max\{0 \leq j \leq s : f_{\omega}^j (x) \in \tilde B(e^{N-i} \epsilon)  \}.
\]
Then $s_N \leq s_{N-1} \leq \cdots \leq s_1 \leq s_0$. Define an integer $n \in \{0, 1, \cdots, N-1\}$ in the following way: $n = N$, if $s_0 - s_N < \epsilon^{-3 \beta}$; otherwise, let $n$ be the minimal integer in $ \{0, 1, \cdots, N-1\}$ such that $s_n - s_{n+1} \geq {\epsilon}^{-2 \beta}$. Such minimal $n$ exists since for otherwise $s_0 - s_N = \sum_{i=0}^{N-1} (s_n - s_{n+1}) < N e^{-2 \beta} \leq {\epsilon}^{-3 \beta}$, a contradiction. By minimality of $n$, we have  $s_0 - s_n < N \epsilon^{-2 \beta} \leq \epsilon^{-3 \beta}$. By (5.9), it follows that for $0 \leq j < s_n$,
\begin{equation}
s_n - j = (s - j) - (s - s_n) \geq \epsilon^{- \gamma} Q_{j}^{s-1} (x, \omega, \epsilon) - \epsilon^{-3 \beta} \geq \epsilon^{- 3\beta} Q_{j}^{s_n-1} (x, \omega, \epsilon),
\end{equation}
when $Q_{j}^{s_n-1} (x, \omega, \epsilon)> 0$. Note that when $Q_{j}^{s_n-1} (x, \omega, \epsilon)= 0$, the inequality holds trivally.

If $n = N$, $f_{\omega}^{s_N}(x) \in \tilde B(\epsilon) $. We can argue as in the proof of Lemma 5.6 to show that $s_N$ is a $\theta$-good return time of $(x, \omega)$ into the region $\tilde B(\epsilon) \times \Omega_{\epsilon}$. Indeed, Let $s_N > s_{N+1} > \cdots > s_{N+N_0} = 0$ be all the integers such that $f_{\omega}^{s_{N+i}}(x) \in \tilde B(\epsilon), 0 \leq i \leq N_0$. let 
\[
B_i = \frac{Df_{\omega}^{S_{N+i}}(x)}{d(f_{\omega}^{S_{N+i}}(x), c)} \mbox{ and } \tilde B_i = \frac{Df_{\omega}^{S_{N+i}}(x)}{|\tilde B(\epsilon)|}.
\]
Similarly, we have
\[
A(x, \omega, s_N) \leq B_{N_0} + (1+\theta(\epsilon)) \sum_{i = 1}^{N_0 -1} B_i + \theta(\epsilon) \tilde B_0.
\]
Choose $K_0$ large enough, by Lemma 5.5 statement (1) and (5.10), for each $1 \leq i \leq N_0$,
\begin{align*}
\log \frac{\tilde B_0}{B_i}& \geq K_0 \Gamma_{s_{N+i}}^{s_N-1} (x, \omega, \epsilon) - Q_{s_{N+i}}^{s_N-1} (x, \omega, \epsilon) + \epsilon^{\beta}(s_N - s_{N+i})\\
&  \geq K_0 i  - \epsilon^{3 \beta} (s_N - s_{N+i}) + \epsilon^{\beta} (s_N - s_{N+i}) \geq K_0 i.
\end{align*}
The last inequality is enough to obtain the desired result.

If $n < N$, then $s_n > s_{n+1} \geq s_N$, so $f_{\omega}^{s_n} (x) \notin \tilde B(\epsilon)$. For each $0 \leq j \leq s_n$, let 
\[
A_j := \frac{D f_{\omega}^j(x)}{ d(f_{\omega}^j(x), c)}.
\]
Now we estimate $A_{s_n} / A(x, \omega, s_n)$ from below. Let $( s_n > ) s_N > s_{N+1} > \cdots > s_{N+N_0} = 0$ be all the integers such that $f_{\omega}^{s_{N+i}}(x) \in \tilde B(\epsilon), 0 \leq i \leq N_0$.

For $n \leq k \leq N+N_0$, by Lemma 5.5 statement (2), we have
\[
\log \frac{A_{s_n}}{A_{s_k}} \geq \epsilon^{\beta} (s_n - s_k) + 2 \log \epsilon - Q_{s_k}^{s_n -1}(x, \omega, \epsilon).
\]
To apply Lemma 5.5, we set $K=1$, $(y, \overline \omega) = F^{s_k}(x, \omega)$, and $\delta = \epsilon$ in the case that $k \geq N$ and $\delta = e^{N-k} \epsilon$ for otherwise, the estimate keeps the same. Since $s_n - s_k \geq s_n - s_{n+1} \geq \epsilon^{-2 \beta}$, by (5.10) we have
\begin{align*}
\log \frac{A_{s_n}}{A_{s_k}} & \geq \epsilon^{\beta} (s_n - s_k) + 2 \log \epsilon - \epsilon^{3 \beta} (s_n - s_k) \geq \frac{\epsilon^{\beta}}{2} (s_n - s_k) \\
& \geq \frac{\epsilon^{\beta}}{2} (s_n - s_{n+1}) + \frac{\epsilon^{\beta}}{2} (s_{n+1} - s_k) \geq \frac{\epsilon^{-\beta}}{2} + \frac{\epsilon^{\beta}}{2} \max\{ k- N, 0 \}.
\end{align*}
Thus, 
\[
\frac{\sum_{i=n+1}^{N + N_0} A_{s_i}}{A_{s_n}} \leq \sum_{i=n+1}^{N + N_0} \frac{1}{e^{\frac{1}{2 \epsilon^{\beta}} +  \frac{\epsilon^{\beta}}{2} \max\{ i- N, 0 \}}}
\]
is very large. By Proposition 9, this implies 
\begin{align*}
A(x, \omega, s_{n+1} + 1)& = \sum_{i=S_{N+N_0}}^{s_{n+1}} \frac{D f_{\omega}^j(x)}{ d(f_{\omega}^j(x), c)} = \sum_{i = N+N_0}^{n+1} A_{s_i} + \sum_{i = N+N_0}^{n+2} \sum_{j = s_i+1}^{s_{i-1} -1} A_j \\
& \leq 2 \sum_{i = n+1}^{N + N_0} A_{s_i} \ll A_{s_n}. 
\end{align*}
To finish the proof, we distinguish two cases.

{\it Case 1.} $n > 0$. Then $f_{\omega}^{s_n} (x) \in \tilde B(e^{N-n} \epsilon) \subset \tilde B(\epsilon_0)$. By Proposition 9 again, 
\[
\sum_{j = s_{n+1} +1}^{s_n -1}  \frac{D f_{\omega}^j(x)}{ d(f_{\omega}^j(x), c)} \leq \theta(\epsilon_0) \frac{Df_{\omega}^{s_n}(x)}{ |\tilde B(\epsilon_0)|} \ll  \frac{D f_{\omega}^{s_n}(x)}{ d(f_{\omega}^{s_n}(x), c)} = A_{s_n},
\]
provided $\epsilon_0 > 0$ is small enough. This and the above inequality shows that $A(x, \omega, s_n) \ll A_{s_n}$. Note that since $f_{\omega}^{s_n} (x) \notin \tilde B(\epsilon)$,
\[
A_{s_n} = \frac{Df_{\omega}^{s_n}(x)}{|\tilde B(\epsilon)|} \frac{|\tilde B(\epsilon) |}{d(f_{\omega}^{s_n}(x), c)} \leq 2  \frac{Df_{\omega}^{s_n}(x)}{|\tilde B(\epsilon)|}.
\]
This implies $s_n$ is a $\theta$-good return time of $(x, \omega)$ into $\tilde B(\epsilon) \times \Omega_{\epsilon}$.

{\it Case 2.} $n =0$. Since $s_1 < s_0$ is well-defined, then $f_{\omega}^j(x) \notin \tilde B(\epsilon_0/e)$ for all $s_1 < j \leq s_0$. Note that in this case, $A_{s_0} \leq Df_{\omega}^{s_0} (x) / |\tilde B(\epsilon_0/e)|$. By Theorem 2 statement (2), $Df_{\omega}^{s_0} / Df_{\omega}^j (x)$ is exponentially large in $s_0 -j$, hence 
\[
\sum_{j = s_1 +1}^{s_0 -1}  \frac{D f_{\omega}^j(x)}{ d(f_{\omega}^j(x), c)} \leq C  \frac{D f_{\omega}^{s_0}(x)}{ |\tilde B(\epsilon_0/e)|},
\]
where $C$ depends only on $\ell$ and $\epsilon_0$. Combined with the above estimate, we have 
\[
A(x, \omega, s_0) = A(x, \omega, s_1 + 1) +  \sum_{j = s_1 +1}^{s_0 -1}  \frac{D f_{\omega}^j(x)}{ d(f_{\omega}^j(x), c)} \leq (C + \theta') \frac{Df_{\omega}^{s_0} (x)}{ |\tilde B(\epsilon_0/e)|},
\]
which implies that $s_0$ is a $\tau$-scale expansion time of $(x, \omega)$ for some constant $\tau > 0$ independent of $\epsilon$.

\end{proof}

Now we state the proof of Proposition 15.

\begin{proof}[Proof of Proposition 15]

Fix $\beta \in (0, \frac{\alpha}{8})$. Let  $(x, \omega) \in \tilde B(\epsilon) \times \Omega_{\epsilon}$ with $\epsilon > 0$ small.

\begin{claim}
There exists a constant $\tau > 0$ such that 
\begin{equation}
\hat h_{\epsilon, \tau}^{\theta} (x, \omega) \leq T_1 : = \inf \{s \geq 1 : s > \epsilon^{-4 \beta} Q_0^{s-1} (x, \omega, \epsilon) \}.
\end{equation}
\end{claim}

Indeed, if $T_1 < \infty$, by minimality, for each $0 \leq j < T_1$,
\[
T_1 - j \geq \epsilon^{-4 \beta} Q_j^{T_1-1} (x, \omega, \epsilon).
\]
By Lemma 5.7, there exists $\tau > 0$ such that $\hat h_{\epsilon, \tau}^{\theta} (x, \omega) \leq T_1$.

Assume that $(x, \omega) \notin Bad_m(\kappa, \epsilon)$. If $T_1 \leq m \epsilon^{- \alpha}$, then the proof is finished. So assume that $T_1 > m \epsilon^{- \alpha}$. Set $s_0 = [m \epsilon^{- \frac{\alpha}{2}}] < T_1$. This implies $s_0 \leq \epsilon^{-4 \beta} Q_0^{s_0-1} (x, \omega, \epsilon)$, then $Q_0^{s_0-1} (x, \omega, \epsilon) \geq \epsilon^{4 \beta} [m \epsilon^{- \frac{\alpha}{2}}] > m$.

Since $(x, \omega) \notin Bad_m(\kappa, \epsilon)$, there exists a minimal non-negative integer $s_1$ such that 
\[
Q_0^{s_1} (x, \omega, \epsilon) \leq \min \{ m, \kappa \Gamma_0^{s_1} (x, \omega, \epsilon) \}.
\]
Since $Q_0^{s_1} (x, \omega, \epsilon) \leq m  < Q_0^{s_0 -1} (x, \omega, \epsilon)$, we have $s_1 < s_0 -1 < s_0$. By minimality of $s_1$, it follows that 

(1) $f_{\omega}^{s_1} (x) \in \tilde B(\epsilon)$;

(2) for each $0 \leq j < s_1$, 
\[
Q_j^{s_1} (x, \omega, \epsilon) \leq \kappa \Gamma_j^{s_1} (x, \omega, \epsilon).
\]
This is because $Q_0^j (x, \omega, \epsilon) \leq Q_0^{s_1} (x, \omega, \epsilon) \leq m$, then $Q_0^j (x, \omega, \epsilon) \geq \kappa \Gamma_0^j (x, \omega, \epsilon)$. Let 
\[
s_2 : = \inf\{s > s_1 : f_{\omega}^s(x) \in \tilde B(\epsilon) \}.
\]
By property (2) above, for each $0 \leq j < s_2 -1$, $Q_j^{s_2-1} (x, \omega, \epsilon) \leq \kappa \Gamma_j^{s_2-1} (x, \omega, \epsilon)$. Note that for $s_1 < j \leq s_2 -1,  Q_j^{s_2-1} (x, \omega, \epsilon)= 0 =  \Gamma_j^{s_2-1} (x, \omega, \epsilon)$. By Lemma 5.6, $s_2$ is a $\theta$-good return time, so $h_{\epsilon}^{\theta} (x, \omega)  \leq s_2$, and by (5.11), $\hat h_{\epsilon, \tau}^{\theta} (x, \omega) \leq \min\{s_2, T_1 \}$.

If $s_1 = 0$. By minimality of $T_1$, for each $0 < s < \min\{s_2, T_1 \}$ we have
\[
s \leq \epsilon^{-4 \beta} Q_0^{s-1} (x, \omega, \epsilon) =  \epsilon^{-4 \beta} q_{\epsilon} (x, \omega) \leq  \epsilon^{-4 \beta} \kappa.
\]
Therefore, 
\[
\hat h_{\epsilon, \tau}^{\theta} (x, \omega) \leq  \epsilon^{-4 \beta} \kappa + 1 < m \epsilon^{-\frac{\alpha}{2}} < m \epsilon^{-\alpha}.
\]

If $s_1 \geq 1$. Similarly, for each $s_1 \leq s < \min\{s_2, T_1 \}$, 
\begin{align*}
s \leq & \epsilon^{-4 \beta} Q_0^{s-1} (x, \omega, \epsilon) =\epsilon^{-4 \beta} Q_0^{s_1} (x, \omega, \epsilon) \leq \epsilon^{-4 \beta} \kappa \Gamma_0^{s_1} (x, \omega, \epsilon) \\
& \leq \epsilon^{-4 \beta} \kappa (s_1 + 1) < \epsilon^{-4 \beta} \kappa (m \epsilon^{- \frac{\alpha}{2}} + 1) < m \epsilon^{-\alpha}.
\end{align*}
This completes the proof.

\end{proof}

\subsection{Proof of Proposition 12 \& 13}

\begin{proof}[Proof of Proposition 12 ]

Take $\alpha \in (0, \gamma/p)$, let $\tau > 1$ be given by Proposition 15. Then for $(x, \omega) \in \tilde B(\epsilon) \times \Omega_{\epsilon}$ with $\epsilon > 0$ small enough, we have
\[
\hat h_{\epsilon, \tau}^{\theta} (x, \omega) > m \epsilon^{-\alpha} \Rightarrow (x, \omega) \notin Bad_m^c(\kappa, \epsilon),
\]
where $\kappa$ is the constant given by Proposition 14. Therefore,
\begin{align*}
\iint_{\tilde B(\epsilon) \times \Omega_{\epsilon}} (\hat h^{\theta}_{\epsilon, \tau} (x, \omega) )^p dP_{\epsilon} &= \sum_{m=1}^{\infty} \iint_{\{\hat h^{\theta}_{\epsilon, \tau} (x, \omega) \in ( (m-1) \epsilon^{- \alpha}, m \epsilon^{- \alpha} ]  \}} (\hat h^{\theta}_{\epsilon, \tau} (x, \omega) )^p dP_{\epsilon}\\
& \leq \sum_{m=1}^{\infty} m^p \epsilon^{- \alpha p} P_{\epsilon} (Bad^c_{m-1} (\kappa, \epsilon)).
\end{align*}
By Proposition 14, we complete the proof.

\end{proof}

\begin{proof}[Proof of Proposition 13 ]

By assumption, there exists $\kappa = \kappa(b) $ and $m = m(b)  \geq 1$ such that for $(x, \omega) \in \tilde B(\epsilon) \times \Omega_{\epsilon}$ with $x \notin \tilde B(b \epsilon)$, we have $(x, \omega) \notin Bad_{m}(\kappa, \epsilon)$, provided $\epsilon> 0$ is small enough. Indeed, we can choose $m$ and $\kappa$ large so that $Q_0^0(x, \omega, \epsilon) : = q_{\epsilon}(x, \omega) < \min\{m,\kappa\}$ for all $x \notin \tilde B(b \epsilon)$. Then the desired estimate holds by Proposition 15. 
 
\end{proof}

\section{Inducing to a large scale}

We shall deduce the Reduced Main Theorem from the following proposition.

\begin{prop}
Under the same assumption, we have the following holds. Let $\theta >0$ and $p \geq 1$
 be constants. Then for each $\delta_0 > 0$ small, there exist $\epsilon_0 >0$ and $C>0$ such that for each $\epsilon \in (0, \epsilon_0]$ we have:
 \[
 \iint_{\tilde B(\delta_0) \times \Omega_{\epsilon}} (h_{\delta_0}^{\theta} (x, \omega))^p d P_{\epsilon} \leq C.
 \]

\end{prop}

To prove Proposition 16, we need the following two propositions whose proofs will be given in Subsection 7.2 and 7.3 respectively. For $p \geq 1$ and $0 < \epsilon \leq \delta \leq \delta_0/e$, write
\begin{align}
S_p^{\theta} (\delta, \epsilon; \delta_0) & = \frac{1}{|\tilde B(\delta)|} \iint_{\tilde B(\delta) \times \Omega_{\epsilon}} \left( \inf_{\delta' \in [\delta, \delta_0]} h^{\theta}_{\delta'} (x, \omega) \right)^p d P_{\epsilon},   \\
\hat S_p^{\theta} (\delta, \epsilon; \delta_0)  & = \iint_{(\tilde B(\delta_0) \setminus \tilde B(\delta))\times \Omega_{\epsilon}} \frac{1}{d(x, c)} \left( \inf_{\delta' \in [e\delta, \delta_0]} h^{\theta}_{\delta'} (x, \omega) \right)^p d P_{\epsilon}.
\end{align}

\begin{prop}
Fix $\theta > 0, \gamma > 0$ and $p \geq 1$. For each $\delta_0> 0$ small enough, there exist $\epsilon_0 >0$ and $C>0$ such that the following holds provided that $0 < \delta \leq \delta_0/e $ and $0 < \epsilon \leq \min\{\epsilon_0, \delta \}$:

(1) $S_p^{\theta} (\epsilon, \epsilon; \delta_0) \leq C \epsilon^{-\gamma} $;

(2) $\hat S_p^{\theta} (\delta, \epsilon; \delta_0) \leq C \delta^{-\gamma}$.

\end{prop}

\begin{prop}

Fix $p \geq 1, \gamma >0$ and $\lambda \in (e^{- \frac{1}{\ell}}, 1)$. There exists $\theta_* >0$ such that for each $\theta \in (0, \theta_*)$ the following holds:
\begin{equation}
S_p^{\theta} (e\delta, \epsilon; \delta_0) < \lambda( S_p^{\theta} (\delta, \epsilon; \delta_0) + 2 \hat S_p^{\theta/e} (\delta, \epsilon; \delta_0))
\end{equation}
provided that $0 < \epsilon \leq \delta \leq \delta_0/e$ small enough.

\end{prop}

Let us assume these propositions and prove Proposition 16 and the Reduced Main Theorem.

\begin{proof}[Proof of Proposition 16]

According to (6.1), it suffices to show that $S_p^{\theta} (\delta_0, \epsilon; \delta_0) \leq C$ for some constant $C>0$ provided that $\epsilon > 0$ small enough. Take $\lambda \in (e^{-\frac{1}{\ell}}, 1)$ and $\gamma > 0$ such that $\lambda_0 : = \lambda e^{\gamma} < 1$. Let $p \geq 1$ and $\theta>0$ be given. We may certainly assume that $\theta \in (0, \theta_*)$. It is a fact that $h_{\delta}^{\theta/e}(x, \omega) \leq h_{\delta}^{\theta}(x, \omega)$ for each $\delta > 0$ and $(x, \omega)$, hence $\hat S_p^{\theta/e} (\delta, \epsilon; \delta_0) \leq \hat S_p^{\theta} (\delta, \epsilon; \delta_0)$.

By Proposition 17 and 18, for each $\delta_0 > 0$ small enough, there exist $\epsilon_0 > 0$ and $C_1 > 0$ such that 
\begin{align}
S_p^{\theta} (\epsilon, \epsilon; \delta_0) & \leq C_1 \epsilon^{-\gamma},\\
S_p^{\theta} (e\delta, \epsilon; \delta_0) & \leq \lambda S_p^{\theta} (\delta, \epsilon; \delta_0) + C_1 \lambda \delta^{-\gamma},
\end{align}
for any $0 < \delta \leq \delta_0/e$ and $0 < \epsilon \leq \min \{\delta, \epsilon_0 \}$.

Let $N$ be the maximal integer such that $e^N \epsilon \leq \delta_0$. Let $S_k = (e^{-k} \delta_0)^{\gamma} S_p^{\theta} (e^{-k} \delta_0, \epsilon; \delta_0)$. Then by (6.5), for each $0 \leq k <N$, we have
\[
S_k \leq (e^{-k} \delta_0)^{\gamma} \left( \lambda S_p^{\theta} (e^{-(k+1)} \delta_0, \epsilon; \delta_0) + C_1 \lambda (e^{-(k+1)}\delta_0)^{\gamma}  \right) \leq \lambda_0 (S_{k+1} + C_1).
\]
It follows that 
\[
S_0 : = \delta_0^{\gamma} S_p^{\theta} (\delta_0, \epsilon; \delta_0)  \leq \lambda_0^N S_N + C_2 \leq S_N + C_2,
\]
where $C_2 > 0$ is a constant. Since $e^{-N} \delta_0 < e \epsilon \leq e^{-(N-1)} \delta_0 $, by (6.1), 
\[
\inf_{\delta' \in [e^{-N} \delta_0, \delta_0]} h_{\delta'}^{\theta}(x, \omega) \leq  \inf_{\delta' \in [e \epsilon, \delta_0]} h_{\delta'}^{\theta}(x, \omega),
\]
and hence
\[
\frac{S_p^{\theta} (e^{-N}\delta_0, \epsilon; \delta_0)}{S_p^{\theta} (e\epsilon, \epsilon; \delta_0)} \leq \frac{|\tilde B(e \epsilon)|}{|\tilde B(e^{-N} \delta_0)|} \leq \frac{|\tilde B(e^{-(N-1)} \delta_0)|}{|\tilde B(e^{-N} \delta_0)|} = C_3.
\]
So by (6.4) and (6.5),
\begin{align*}
S_N & : = (e^{-N} \delta_0)^{\gamma} S_p^{\theta} (e^{-N} \delta_0, \epsilon; \delta_0) \leq C_3 (e^{-N} \delta_0)^{\gamma}S_p^{\theta} (e\epsilon, \epsilon; \delta_0) \\
& \leq C_3 (e^{-N} \delta_0)^{\gamma} \left( \lambda S_p^{\theta} (\epsilon, \epsilon; \delta_0) + C_1 \lambda \epsilon^{- \gamma} \right)  \\
& \leq 2 \lambda C_1  C_3 (e^{-N} \delta_0)^{\gamma} \epsilon^{- \gamma}  \leq C_4.
\end{align*}
Thus  $S_p^{\theta} (\delta_0, \epsilon; \delta_0)$ is bounded from above by a constant independent of $\epsilon$.

\end{proof}

\begin{proof}[Proof of the Reduced Main Theorem]

Fix $p>1$ and $\theta > 0$ small so that
\[
\theta < \min\left\{ \frac{\theta_0}{4\tilde \kappa}, \frac{1}{\tilde \kappa^2 e^3} \right\} \mbox{ where \ } \tilde \kappa : = \sup_{\delta \leq \delta_*} \frac{|\tilde B(2\delta)|}{|\tilde B(\delta)|}.
\] 
Let $\delta_0 > 0$ be small such that the conclusion of Proposition 16 holds. Reducing $\delta_0$ if necessary, by Lemma 5.2, $l_{\delta_0} (x, \omega) = h_{\delta_0}^{\theta} (x, \omega)$ holds for all $x \in I_c \setminus \tilde B(\delta_0)$ and $\omega \in \Omega_{\delta_0}$. By Proposition 6, there exist constants $K_1 = K_1(\delta_0)> 0$ and $\eta = \eta(\delta_0) > 1$ such that 
\[
P_{\epsilon} (\{ (x, \omega) \in (I_c \setminus \tilde B(\delta_0)) \times \Omega_{\epsilon}: l_{\delta_0}(x, \omega) \geq n \}) \leq K_1 e^{-\eta n}.
\]
Then,
\[
\iint_{ (I_c \setminus \tilde B(\delta_0)) \times \Omega_{\epsilon}} (h_{\delta_0}^{\theta} (x, \omega))^p d P_{\epsilon} = \iint_{ (I_c \setminus \tilde B(\delta_0)) \times \Omega_{\epsilon}} (l_{\delta_0} (x, \omega))^p d P_{\epsilon} \leq \sum_{n=1}^{\infty} K_1 n^p  e^{-\eta n} \leq C',
\]
where $C'$ is a constant, provided $\epsilon > 0$ is small enough. By Proposition 16, there exists a constant $C>0$ such that 
\[
\iint_{ I_c  \times \Omega_{\epsilon}} (h_{\delta_0}^{\theta} (x, \omega))^p d P_{\epsilon}  \leq C
\]
holds when $\epsilon > 0$ is small enough.

By Proposition 11, when $0 < \epsilon \leq \delta_0$, there exists a nice set $V$ for $\epsilon$-perturbations such that $\tilde B(\delta_0) \subset V^{\omega} \subset \tilde B( 2\delta_0) $ for each $\omega \in \Omega_{\epsilon}$. By Lemma 3.5 and Lemma 6.1, for each $(x, \omega) \in V$, there exists a interval $J \ni x$ such that  $f_{\omega}^m$ maps $J$ diffeomorhpically onto $\tilde B(2\delta_0) \supset V^{\sigma^m \omega}$ with $\mathcal N(f_{\omega}^m |J)< 1$ where $m =h_{\delta_0}^{\theta} (x, \omega)$. By bounded distortion and the choice of $\theta$, 
\begin{align*}
\inf_{y \in J} Df_{\omega}^m(y) & \geq \frac{Df_{\omega}^m(x) }{e} \geq \frac{1}{e \theta} A(x, \omega, m) |\tilde B(\delta_0)|   \geq \frac{1}{e \theta} \frac{|\tilde B(\delta_0)|}{d(x, c)} \\
& \geq \frac{1}{e\theta} \frac{|\tilde B(\delta_0)|}{|\tilde B(2\delta_0)|}  \geq \frac{1}{\tilde \kappa^2 e\theta} \frac{|\tilde B(2\delta_0)|}{|\tilde B(\delta_0)|} \geq e^2 \frac{|V^{\sigma^m \omega}|}{|V^{\omega}|}.
\end{align*}
This shows that $h_{\delta_0}^{\theta} (x, \omega)$ is a Markov inducing time. So $m_V(x, \omega) \leq h_{\delta_0}^{\theta} (x, \omega)$. Therefore,
\[
\int_V (m_V(x, \omega))^p dP_{\epsilon} \leq \iint_{ I_c  \times \Omega_{\epsilon}} (h_{\delta_0}^{\theta} (x, \omega))^p d P_{\epsilon}  \leq C.
\]
This completes the proof.

\end{proof}

\subsection{Preparatory lemmas}

\begin{definition}
We say that a positive integer $s$ is a {\it $\theta$-close return time} of $(x, \omega) \in I_c \times \Omega$ if:
\[
\theta Df_{\omega}^s(x) \geq A(x, \omega, s) d(f_{\omega}^s(x), c).
\]
\end{definition}

If $(x, \omega) \in I_c \times \Omega$ and $s$ is a $\theta$-good return time of $(x, \omega)$ into $\tilde B(\delta) \times \Omega$, then 
\[
\theta Df_{\omega}^s(x) \geq A(x, \omega, s) |\tilde B(\delta)| \geq A(x, \omega, s) d(f_{\omega}^s(x), c).
\]
Hence $s$ is also a $\theta$-close return time. If $s$ is a $\tau$-scale expansion time of $(x, \omega) \in I_c \times \Omega$, then it is a $\theta_0/\tau$-close return time since:
\[
\frac{\theta_0}{\tau} Df_{\omega}^s(x) \geq e A(x, \omega, s) \geq A(x, \omega, s) d(f_{\omega}^s(x), c).
\]

\begin{lem}

Consider $(x, \omega) \in I_c \times \Omega$.

(1) Let $0 = T_0 < T_1 < \cdots < T_n$ be integers such that for each $0 \leq i < n$, $T_{i+1} - T_i$ is a $\frac{1}{2}$-close return of $F^{T_i} (x, \omega)$, then $T_n$ is a 1-close return time of $(x, \omega)$.

(2) If $t$ is a $\theta_1$-close return of $(x, \omega)$ and $s$ is a $\theta_2$-good return of $(y, \sigma^t \omega) = F^t(x, \omega)$ into $\tilde B(\delta) \times \Omega$ for some $\delta >0$. Then $t+s$ is a $(1 + \theta_1)\theta_2$-good return of $(x, \omega)$ into $\tilde B(\delta) \times \Omega$.

\end{lem}

\begin{proof}

(1) For $0 \leq i < n$, let
\[
A_i = A(F^{T_i} (x, \omega), T_{i+1} -T_i) \mbox{ and \ } \tilde A_i =  Df_{\omega}^{T_i} (x) A_i.
\]
By assumption, for each $i$, we have
\[
\frac{1}{2} \frac{Df_{\omega}^{T_{i+1}}(x)}{Df_{\omega}^{T_i}(x)} = \frac{1}{2}  Df_{\sigma^{T_i}\omega}^{T_{i+1}-T_i}(f_{\omega}^{T_i}(x)) \geq A_i d(f_{\omega}^{T_{i+1}}(x), c).
\]
For $i < n-1$, we have
\[
d(f_{\omega}^{T_{i+1}}(x), c) A_{i+1}  \geq d(f_{\omega}^{T_{i+1}}(x), c) \cdot \frac{1}{d(f_{\omega}^{T_{i+1}}(x), c)} = 1.
\]
Therefore,
\[
\tilde A_{i+1} = Df_{\omega}^{T_{i+1}}(x) A_{i+1} \geq 2 Df_{\omega}^{T_i}(x) \cdot A_i d(f_{\omega}^{T_{i+1}}(x), c) \cdot A_{i+1} \geq 2 \tilde A_i.
\]
Thus 
\begin{align*}
A(x, \omega, T_n) & = \sum_{i=0}^{n-1} \tilde A_i \leq \left( \frac{1}{2^{n-1}} + \frac{1}{2^{n-2}} + \cdots + 1 \right) \tilde A_{n-1} \leq 2  \tilde A_{n-1}\\
& = 2 Df_{\omega}^{T_{n-1}}(x) A_{n-1} \leq 2 Df_{\omega}^{T_{n-1}}(x) \cdot \frac{1}{2} \frac{Df_{\omega}^{T_{n}}(x)}{Df_{\omega}^{T_{n-1}}(x)} \cdot \frac{1}{d(f_{\omega}^{T_n}(x), c)}\\
& = \frac{Df_{\omega}^{T_{n}}(x)}{d(f_{\omega}^{T_n}(x), c)}.
\end{align*}
This finishes the proof.

(2) Since $A(y, \sigma^t \omega, s) \geq  1 /d(y, c) = 1/d(f_{\omega}^t(x), c)$, we have
\[
\theta_1 Df_{\omega}^t (x) \geq A(x, \omega, t) d(f_{\omega}^t(x), c) \geq \frac{ A(x, \omega, t)}{A(y, \sigma^t \omega, s)}.
\]
Thus,
\begin{align*}
A(x, \omega, t+s)  |\tilde B(\delta)|& = (A(x, \omega, t) + Df_{\omega}^t(x) A(y, \sigma^t \omega, s)  )\tilde B(\delta)|\\
& \leq (1+\theta_1) Df_{\omega}^t(x) A(y, \sigma^t \omega, s) |\tilde B(\delta)| \\
& = (1 + \theta_1) Df_{\omega}^{t+s}(x) \frac{ A(y, \sigma^t \omega, s) |\tilde B(\delta)| }{D f_{\sigma^t \omega}^s(y)} \\
& \leq (1 + \theta_1) \theta_2 Df_{\omega}^{t+s}(x).
\end{align*}
The statement follows.

\end{proof}

Let $\mathscr G: \mathcal E \to I \times \Omega$ be a Borel measurable map defined on Borel subset $\mathcal E \subset I_c \times \Omega$.

(1) We say that $\mathscr G$ is induced by $F$ if there exists a Borel measurable function $T : \mathcal E \to \mathbb Z^+$ such that 
\[
\mathscr G(x, \omega) = F^{T(x, \omega)} (x, \omega) \mbox{ for each \ } (x, \omega) \in \mathcal E.
\]

(2) We say that $\mathscr G$ is {\it future-free} provided for each $(x, \omega) \in \mathcal E$ and $\tilde \omega \in \Omega$ with $\omega_i = \tilde \omega_i$ for $0 \leq i < T(x, \omega)$, we have $(x, \tilde \omega) \in \mathcal E$ and $T(x, \tilde \omega) = T(x, \omega)$.

Given a Borel probability measure $\nu$ on $[-1, 1]$, we define the randomized transfer operator corresponding to the map $\mathscr G$ as 
\begin{equation}
\mathcal L_{\mathscr G}^{\nu} (y) = \int_{\Omega} \mathcal L_{\mathscr G}^{\omega} (y) d \nu^{\mathbb N} (\omega)
\end{equation}
for each $y \in I$, where
\begin{equation}
 \mathcal L_{\mathscr G}^{\omega} (y) = \sum_{\substack{x \in \mathcal E^{\omega} \\ f_{\omega}^{T(x, \omega)}(x) =y}} \frac{1}{Df_{\omega}^{T(x, \omega)}(x)}.
\end{equation}

The following lemma can be viewed as change of variable.

\begin{lem}

Let $\mathscr G: \mathcal E \to I \times \Omega$ be a future-free and Borel measurable induced map with an inducing time function $T$ and let $\phi: I \times \Omega \to [0, \infty)$ be a Borel measurable function. Then for any Borel probability measure $\nu$ on $[-1, 1]$, we have
\[
\int_{\mathcal E} \phi(\mathscr G(x, \omega)) dx d \nu^{\mathbb N} (\omega) = \int_{\Omega} \int_0^1 \mathcal L_{\mathscr G}^{\nu} (y) \phi(y, \tilde \omega) dy d  \nu^{\mathbb N} (\tilde \omega),
\]
where $(y, \tilde \omega) = \mathscr G(x, \omega)$.

\end{lem}

\begin{proof}

Let $X_T = \{ (x, \omega) \in \mathcal E: T(x, \omega) = T \}$, and let 
\[
\mathcal L_T^{\omega}(y) = \sum_{\substack{x \in X_T^{\omega} \\ f_{\omega}^T(x) = y}} \frac{1}{D f_{\omega}^T(x)}.
\]
Then $\mathcal E = \bigcup_{T=1}^{\infty} X_T$ and $ \mathcal L_{\mathscr G}^{\omega} (y)  = \sum_{T=1}^{\infty}  \mathcal L_{T}^{\omega} (y) $. By Fubini's Theorem and change of variable, we have
\begin{align*}
\int_{X_T} \phi(\mathscr G(x, \omega)) dx d \nu^{\mathbb N} (\omega) & = \int_{\Omega} \int_{X_T^{\omega}} \phi(\mathscr G(x, \omega)) dx d \nu^{\mathbb N} (\omega)\\
& =\int_{\Omega} \int_0^1 \mathcal L_T^{\omega}(y) \phi(y, \sigma^T \omega) dy d \nu^{\mathbb N} (\omega)\\
&=\int_0^1 \int_{\Omega}  \mathcal L_T^{\omega}(y) \phi(y, \sigma^T \omega)  d \nu^{\mathbb N} (\omega) dy.
\end{align*}
Since $\mathscr G$ is future-free, $\mathcal L_T^{\omega}(y)$ depends only on the first $T$ coordinates of $\omega$. Hence
\begin{align*}
\int_{\Omega}  \mathcal L_T^{\omega}(y) \phi(y, \sigma^T \omega)  d \nu^{\mathbb N} (\omega) & =\int_{\Omega}  \mathcal L_T^{\omega}(y)  d \nu^{\mathbb N} (\omega) \int_{\Omega} \phi(y, \sigma^T \omega)  d \nu^{\mathbb N} (\omega) \\
& =\int_{\Omega}  \mathcal L_T^{\omega}(y)  d \nu^{\mathbb N} (\omega) \int_{\Omega} \phi(y, \omega)  d \nu^{\mathbb N} (\omega).
\end{align*}
By Fubini again,
\begin{align*}
\int_{X_T} \phi(\mathscr G(x, \omega)) dx d \nu^{\mathbb N} (\omega) & = \int_0^1 \left( \int_{\Omega}  \mathcal L_T^{\omega}(y)  d \nu^{\mathbb N} (\omega) \int_{\Omega} \phi(y, \omega)  d \nu^{\mathbb N} (\omega)\right) dy\\
& = \int_{\Omega} \int_0^1 \mathcal L_{\mathscr G}^{\nu} (y) \phi(y, \tilde \omega) dy d  \nu^{\mathbb N} (\tilde \omega).
\end{align*}

\end{proof}

\subsection{Proof of Proposition 17}

\begin{lem}

Given $\tau>0$ and $\theta>0$ the following holds provided $0 < \epsilon \leq \delta \leq \delta_0$ are small enough. For any $x \in \tilde B(\delta)$ and $\omega \in \Omega_{\epsilon}$, if $h = \hat h_{\delta, \tau}^{\theta} (x, \omega) < \infty$ and $l = l_{\delta_0} (F^h(x, \omega)) < \infty$, then $l+h$ is a $\theta$-good return map of $(x, \omega)$ into $\tilde B(\delta') \times \Omega$ for some $\delta' \in [\delta, \delta_0]$.

\end{lem}

\begin{proof}

Let $\theta_0$ be a small constant given by Lemma 3.5. Let $\theta_1 = \max\{ \theta_0/\tau, 1 \}, \theta_2 = \theta/(1 + \theta_1)$. If $\delta_0$ is small enough, then 
\begin{equation}
| \tilde B(\delta_0) | \leq  \frac{\tau \theta}{\theta_0}.
\end{equation}
Moreover, by Lemma 5.2, we have either
\begin{equation}
l=0 \mbox{ or \ } l = h_{\delta_0}^{\theta_2} (F^h(x, \omega)).
\end{equation}

{\it Case 1.}  If $l=0$, which means $f_{\omega}^h(x) \in \tilde B(\delta_0)$. In case that $h$ is a $\tau$-scale expansion time, by (6.8) we have
\[
\theta Df_{\omega}^h(x) = \frac{\theta}{\theta_0} \theta_0 Df_{\omega}^h(x) \geq \frac{\theta}{\theta_0} e\tau A(x, \omega, h) \geq e A(x, \omega, h) |\tilde B(\delta_0)|,
\]
which shows $h$ is a $\theta$-good return time of $(x, \omega)$ into $\tilde B(\delta_0) \times \Omega$. Otherwise, by definition $h$ is a $\theta$-good return time of $(x, \omega)$ into $\tilde B(\delta'') \times \Omega$ for some $\delta'' \geq \delta$. Let $\delta' = \min\{ \delta'', \delta_0 \}$, be definition it follows that $h$ is a $\theta$-good return time of $(x, \omega)$ into $\tilde B(\delta') \times \Omega$.

{\it Case 2.} If $l \geq 1$, then $f_{\omega}^h(x) \notin \tilde B(\delta_0)$. Hence $h = T_{\tau}(x, \omega)$ and 
$l = h_{\delta_0}^{\theta_2} (F^h(x, \omega))$. Then $h$ is a $\theta_1$-close return time, and $l$ is a $\theta_2$-good return of $(x, \omega)$ into $\tilde B(\delta_0) \times \Omega$. By Lemma 6.1, $l + h$ is a $\theta$-good return time of $(x, \omega)$ into $\tilde B(\delta_0) \times \Omega$.

\end{proof}

\begin{proof}[Proof of Proposition 17]

(1) Fix $\theta>0, \gamma>0, p \geq 1$ and let $\tau >0$ be given by Proposition 12. Assume that $\delta_0$ is small enough, then for each $\epsilon \in (0, \delta_0]$ we have
\begin{equation}
\frac{1}{|\tilde B(\epsilon)|} \iint_{\tilde B(\epsilon) \times \Omega_{\epsilon}} (\hat h^{\theta}_{\epsilon, \tau} (x, \omega) )^p dP_{\epsilon} \leq \epsilon^{- \gamma}.
\end{equation}
By Proposition 6, there exist constants $\epsilon_0 \in (0, \delta_0), C_1>0$ and $\rho_0 >0$ such that for each $\omega \in \Omega_{\epsilon}$ with $\epsilon \in (0, \epsilon_0]$, we have
\[
|\{ y : l_{\delta_0}(y, \omega) \geq l \} | \leq C_1 e^{- \rho_0 l}.
\]

Fix $\epsilon \in (0, \epsilon_0]$. Write
\begin{align*}
h(x, \omega) & = \hat h_{\epsilon, \tau}^{\theta} (x, \omega), \\
l(x, \omega) & = l_{\delta_0} (x, \omega),\\
H(x, \omega) & = \inf_{\delta' \in [\epsilon, \delta_0]} h_{\delta'}^{\theta} (x, \omega).
\end{align*}
Then $h(x, \omega)$ and $l(x, \omega)$ are finite $P_{\epsilon}$-almost everywhere. By Lemma 6.3, for each $(x, \omega) \in \tilde B(\epsilon) \times \Omega_{\epsilon}$, we have 
\begin{equation}
h(x, \omega) + l(F^{h(x, \omega)} (x, \omega)) \geq H(x, \omega),
\end{equation}
provided that $\delta_0$ is small enough.

For each $k \geq 1$, let $X_k = \{ (x, \omega) \in \tilde B(\epsilon) \times \Omega_{\epsilon} : h(x, \omega) = k\}$, let $\mathscr G_k: X_k \to I \times \Omega$ be the measurable induced map defined by 
\[
(x, \omega) \to F^k(x, \omega),
\]
and let $\phi_k : I \times \Omega  \to [0, \infty)$ be defined as
\[
\phi_k(y, \tilde \omega) = \begin{cases}
l(y, \tilde \omega), & \mbox{if $k < l(y, \tilde \omega) < \infty$ and $\tilde \omega \in \Omega_{\epsilon}$}, \\
0, & \mbox{otherwise}.
\end{cases}
\]
Let 
\[
L_k : = \iint_{I \times \Omega_{\epsilon}} (\phi_k(y, \tilde \omega))^p d P_{\epsilon}.
\]
By the choice of $\epsilon_0$, there exists a constant $C_2 > 0$ such that
\begin{equation}
\sum_{k=1}^{\infty} L_k \leq \sum_{k=1}^{\infty} \sum_{m \geq k} \iint_{l(y, \tilde \omega) = m} m^p d P_{\epsilon} \leq \sum_{k=1}^{\infty} \sum_{m \geq k} m^p C_1 e^{-\rho_0 m} \leq C_2.
\end{equation}

\begin{claim}
There exists a constant $C_3 > 0$ such that for each $y \in I_c \setminus \tilde B(\delta_0)$, each $\omega \in \Omega_{\epsilon}$ and each $k \geq 1$, we have
\[
\mathcal L_{\mathscr G_k}^{\nu_{\epsilon}} (y) \leq C_3 |\tilde B(\epsilon)|.
\]
\end{claim}

For $(x, \omega) \in X_k$, let $J_{x, k}^{\omega}$ be defined as (5.1), then $\mathcal N(f_{\omega}^k | J_{x, k}^{\omega} ) \leq 1$ and 
\[
|J_{x, k}^{\omega}| = \frac{2 \theta_0}{A(x, \omega, k)} = 2 \theta_0 \left( \sum_{i=0}^{k-1} \frac{Df_{\omega}^i(x)}{d(f_{\omega}^i(x), c)} \right)^{-1} \leq 2 \theta_0 d(x, c) \leq 2 \theta_0 |\tilde B(\epsilon)|.
\]
Given $\omega \in \Omega_{\epsilon}$ and $k \geq 1$, these intervals $J_{x, k}^{\omega}$ with $(x, \omega) \in X_k$ and $f_{\omega}^k(x) = y$ are pairwise disjoint. If $h(x, \omega) = T^{\tau}(x, \omega)$, then $\theta_0 Df_{\omega}^k(x) \geq e \tau A(x, \omega, k)$. Hence 
\[
|f_{\omega}^k (J_{x, k}^{\omega})| \geq \frac{Df_{\omega}^k(x)}{e} |J_{x, k}^{\omega}| \geq \frac{\tau}{\theta_0} A(x, \omega, k) \cdot  \frac{2 \theta_0}{A(x, \omega, k)} = 2 \tau.
\]
If $h(x, \omega) = h_{\delta'}^{\theta}(x, \omega)$ for some $\delta' \geq \epsilon$. By Lemma 5.1 and since $y \notin \tilde B(\delta_0)$,
\[
|f_{\omega}^k (J_{x, k}^{\omega})| \geq d(y, c) \geq C' |\tilde B(\delta_0)|.
\]
This implies that $|f_{\omega}^k (J_{x, k}^{\omega})|$ is bounded from below by a constant $\tau_1 = \tau_1(\tau, \delta_0) > 0$. Therefore,
\[
\mathcal L_{\mathscr G_k}^{\omega} (y) = \sum_{\substack{x \in X_k^{\omega} \\ f_{\omega}^k(x) =y}} \frac{1}{Df_{\omega}^k(x)} \leq \frac{e}{\tau_1}  \sum_{\substack{x \in X_k^{\omega} \\ f_{\omega}^k(x) =y}} |J_{x, k}^{\omega}| \leq \frac{e}{\tau_1}(1 + 4 \theta_0) |\tilde B(\epsilon)|.
\]
Then the claim follows.

Since $\mathscr G_k$ is future-free, by Lemma 6.2, we have
\[
M_k : = \int_{X_k} (\phi_k(\mathscr G_k(x, \omega)))^p dP_{\epsilon} = \int_{\Omega_{\epsilon}} \int_0^1 \mathcal L_{\mathscr G_k}^{\nu_{\epsilon}} (y) (\phi_k(y, \tilde \omega))^p dP_{\epsilon}.
\]
Since $\phi_k(y, \tilde \omega) = 0$ for each $y \in \tilde B(\delta_0)$ and by the claim above, we have
\[
M_k = \int_{\Omega_{\epsilon}} \int_{I_c \setminus \tilde B(\delta_0)} \mathcal L_{\mathscr G_k}^{\nu_{\epsilon}} (y) (\phi(y, \tilde \omega))^p dP_{\epsilon} \leq C_3 |\tilde B(\epsilon)| L_k.
\]
By (6.12),
\begin{equation}
\sum_{k=1}^{\infty} M_k \leq C_2 C_3  |\tilde B(\epsilon)|.
\end{equation}
On each $X_k$, we have 
\[
H(x, \omega) \leq h(x, \omega) + l(F^{h(x, \omega)} (x, \omega)) \leq 2h(x, \omega) + \phi_k(\mathscr G_k(x, \omega)).
\]

So,
\[
\int_{X_k} (H(x, \omega))^p d P_{\epsilon} \leq \int_{X_k} (2h(x, \omega) + \phi_k(\mathscr G_k(x, \omega)))^p d P_{\epsilon} \leq C_4 \int_{X_k} (h(x, \omega) )^p d P_{\epsilon} + C_4 M_k,
\]
where $C_4 > 0$ is a constant. Here we use the inequality 
\[
|a+b|^p \leq 2^p (|a|^p + |b|^p), p \geq 1.
\]
Then, by (6.10) and (6.13), 
\begin{align*}
| \tilde B(\epsilon) | S_p^{\theta} (\epsilon, \epsilon; \delta_0) & =  \sum_{k=1}^{\infty} \iint_{X_k} (H(x, \omega))^p d P_{\epsilon}  \leq C_4 \sum_{k=1}^{\infty} \iint_{X_k} (h(x, \omega))^p d P_{\epsilon} + C_4 \sum_{k=1}^{\infty} M_k\\
& \leq C_4 \iint_{\tilde B(\epsilon) \times \Omega_{\epsilon}} (h(x, \omega))^p d P_{\epsilon} + C_4 \sum_{k=1}^{\infty} M_k \leq C_4 \epsilon^{-\gamma} |\tilde B(\epsilon)| + C_5  |\tilde B(\epsilon)|.
\end{align*}
Then the desired estimate holds.

(2) The proof is similar as in (1). We shall use Proposition 13 instead of Proposition 12 to show that for  each $\delta_0> 0$ small enough, there exist $\epsilon_0 >0$ such that if $0 < \delta \leq \delta_0/e $ and $0 < \epsilon \leq \min\{\epsilon_0, \delta \}$, then 
\[
\frac{1}{|\tilde B(\delta)|} \iint_{(\tilde B(e \delta) \setminus \tilde B(\delta))\times \Omega_{\epsilon}} \frac{1}{d(x, c)} \left( \inf_{\delta' \in [e\delta, \delta_0]} h^{\theta}_{\delta'} (x, \omega) \right)^p d P_{\epsilon} \leq C' \delta^{-\gamma},
\]
where $C' > 0$ is a constant. Which implies that
\[
\iint_{(\tilde B(e \delta) \setminus \tilde B(\delta))\times \Omega_{\epsilon}} \frac{1}{d(x, c)} \left( \inf_{\delta' \in [e\delta, \delta_0]} h^{\theta}_{\delta'} (x, \omega) \right)^p d P_{\epsilon} \leq C'' \delta^{-\gamma}.
\]
This finishes the proof.

\end{proof}

\subsection{Proof of Proposition 18}

Fix $p \geq 1, \gamma >0$ and $\lambda \in (e^{- \frac{1}{\ell}}, 1)$. Let $\theta>0$ be small such that 
\[
(1 - (36\theta/\theta_0)^{1/p})^p \lambda e^{\frac{1}{\ell}} > 1.
\]
Let $\delta_0 >0$ be small enough and consider $0 < \epsilon \leq \delta \leq \delta_0/e$. Let 
\begin{align*}
s(x, \omega) &= \inf_{\delta' \in [\delta, \delta_0]} h_{\delta'}^{\theta} (x, \omega), \\
\hat s(x, \omega) & = \inf_{\delta' \in [e\delta, \delta_0]} h_{\delta'}^{\theta/e} (x, \omega),\\
\varphi(x, \omega) & = \begin{cases}
s(x, \omega), & \mbox{if \ } x \in \tilde B(\delta),\\
\hat s(x, \omega), & \mbox{otherwise}.
\end{cases}
\end{align*}
Let $\tilde E_0 = \tilde B(\delta_0) \times \Omega_{\epsilon} \supset E_0 = \tilde B(e\delta) \times \Omega_{\epsilon} $, let 
\[
E_1 = \{(x, \omega) \in  \tilde B(\delta) \times \Omega_{\epsilon}: s(x, \omega) < \hat s(x, \omega) \}.
\]
Let $\mathscr G: E_1 \to \tilde E_0$ denote the map $(x, \omega) \to F^{s(x, \omega)} (x, \omega)$. For each $n \geq 1$, let $E_n = {\rm dom}(\mathscr G^n)$ and $\varphi_n = \chi_{E_n} \cdot \varphi \circ \mathscr G^n$. For each $n \geq 0$, let
\[
K_n = \left( \iint_{E_n} \varphi (\mathscr G^n(x, \omega))^p d P_{\epsilon}  \right)^{\frac{1}{p}}.
\]

\begin{lem}

Let $\delta_0 > 0$ be small enough, then
\[
\left( |\tilde B(e \delta)| S_p^{\theta} (e\delta, \epsilon; \delta_0) \right)^{\frac{1}{p}} \leq \sum_{n=0}^{\infty} K_n.
\]

\end{lem}

\begin{proof}
By Minkowski's inequality, it suffices to prove that for each $(x, \omega) \in E_0$, we have
\begin{equation}
\inf_{\delta' \in [e\delta, \delta_0]} h_{\delta'}^{\theta} (x, \omega) \leq \sum_{n=0}^{\infty} \varphi_n(x, \omega)
\end{equation}
provided $\delta_0 > 0$ is small enough.

If $(x, \omega) \in \bigcap_{n=0}^{\infty} E_n$, then the right hand side is infinity, so (6.14) holds. If $(x, \omega) \in E_0 \setminus E_1$, then  
\[
\varphi_0(x, \omega) = \hat s(x, \omega)  = \inf_{\delta' \in [e\delta, \delta_0]} h_{\delta'}^{\theta/e} (x, \omega) \geq \inf_{\delta' \in [e\delta, \delta_0]} h_{\delta'}^{\theta} (x, \omega),
\]
so (6.14) holds.

Now assume that there exists an integer $n \geq 1$ such that $(x, \omega) \in E_n \setminus E_{n+1}$. We will show that (6.14) follows from Lemma 6.1. By definition, for $0 \leq i \leq n-1, \mathscr G^i(x, \omega) \in E_1$ and $\mathscr G^n(x, \omega) \notin E_1$. Let $T_0 = 0$ and $T_i = \sum_{j=0}^{i-1} \varphi_j(x, \omega) $ for $1 \leq i \leq n+1$. Then for each $0 \leq i \leq n-1$,
\[
T_{i+1} - T_i = \varphi_i(x, \omega) = \varphi(\mathscr G^i(x, \omega)) = s (\mathscr G^i(x, \omega)) = s (F^{T_i}(x, \omega));
\]
and 
\[
T_{n+1} - T_n = \hat s(F^{T_n}(x, \omega)).
\]
By definition, for $0 \leq i \leq n-1$,
\begin{align*}
\frac{1}{2} Df_{\sigma^{T_i}\omega}^{T_{i+1} - T_i } ( F^{T_i}(x, \omega))& \geq \theta Df_{\sigma^{T_i}\omega}^{T_{i+1} - T_i } ( F^{T_i}(x, \omega))  \geq A(F^{T_i}(x, \omega)), T_{i+1} - T_i ) | \tilde B(\delta'') |\\
& \geq A(F^{T_i}(x, \omega)), T_{i+1} - T_i ) d(f_{\omega}^{T_{i+1}}(x), c);
\end{align*}
and 
\[
\frac{\theta}{2} Df_{\sigma^{T_n}\omega}^{T_{n+1} - T_n} ( F^{T_n}(x, \omega)) \geq \frac{\theta}{e} Df_{\sigma^{T_n}\omega}^{T_{n+1} - T_n} ( F^{T_n}(x, \omega))  \geq  A(F^{T_n}(x, \omega)), T_{n+1} - T_n ) | \tilde B(\delta') |.
 \]
This shows that for $0 \leq i < n$, $T_{i+1} - T_i$ is a $\frac{1}{2}$-close return of $F^{T_i}(x, \omega)$ and by Lemma 6.1 part (1), $T_n$ is a $1$-close return of $(x, \omega)$. Also $T_{n+1} - T_n$ is a $\frac{\theta}{2}$-good return time of $F^{T_n}(x, \omega)$ into $\tilde B(\delta') \times \Omega_{\epsilon}$ for some $\delta' \in [e\delta, \delta_0]$. By Lemma 6.1 part (2), $T_{n+1}$ is a $\theta$-good return time of $(x, \omega)$ into $\tilde B(\delta') \times \Omega_{\epsilon}$. This completes the proof.

\end{proof}

Now we estimate $K_n$.

\begin{lem}
Let $\delta_0 > 0$ be small enough, then for any $y \in \tilde B(\delta_0)$, 
\[
\mathcal L_{\mathscr G}^{\nu_{\epsilon}} (y) \leq \frac{36 \theta}{\theta_0} \frac{|\tilde B(\delta)|}{|\tilde B(\delta')|},
\]
where $\delta' = \max \{ \delta, d_*(y, c)\}$.
\end{lem}

\begin{proof}

It suffices to prove that for any fixed $y \in \tilde B(\delta_0)$, $\omega \in \Omega_{\epsilon}$ and $\delta' = \max \{ \delta, d_*(y, c)\}$, 
\[
\mathcal L_{\mathscr G}^{\omega} (y) \leq \frac{36 \theta}{\theta_0} \frac{|\tilde B(\delta)|}{|\tilde B(\delta')|}.
\]
Denote
\[
\mathcal X : = \{x \in \tilde B(\delta) : (x, \omega)\in E_1, f_{\omega}^{s(x, \omega)} (x) = y \}.
\]
For each $x \in \mathcal X$, let $\hat J_x = \hat J_{x, s(x, \omega)}^{\omega}$ be defined in (5.1). Then $\hat J_x \subset I^{\pm}$,  $f_{\omega}^{s(x, \omega)} | \hat J_x$ is a diffeomorphism with $\mathcal N (f_{\omega}^{s(x, \omega)} | \hat J_x) \leq 1$. Let $J_0 \subset J$ be two nested closed intervals centered at $c$ with 
 \[
 |J_0| = 4 |\tilde B(\delta')| \mbox{ and \ } |J| = \frac{\theta_0  |\tilde B(\delta')|}{e \theta}.
 \]
 Since $f_{\omega}^{s(x, \omega)} (x) = y$, by definition $s(x, \omega)$ is a $\theta$-good return time of $(x, \omega)$ into region $\tilde B(\delta_x) \times \Omega_{\epsilon}$ for some $\delta_x \in [\delta', \delta_0]$. Let $\mathcal J$ be any component of $\hat J_x \setminus \{ x\}$, then
\[
 f_{\omega}^s(\mathcal J) \geq \frac{Df_{\omega}^s(x)}{e}  \frac{\theta_0}{A(x, \omega, s)} \geq \frac{\theta_0}{e \theta} |\tilde B(\delta_x)|.
\]
Thus $f_{\omega}^{s(x, \omega)}(\hat J_x) \supset J$. Let $\tilde J_x \subset J_x \subset \hat J_x$ be such that $f_{\omega}^{s(x, \omega)}(\tilde J_x) = J_0$ and $f_{\omega}^{s(x, \omega)}(J_x) = J$. Then
 \begin{equation}
 |\tilde J_x| \leq e \frac{|J_0|}{|J|} |J_x| \leq \frac{4 e^2 \theta}{\theta_0} |J_x|,
 \end{equation}
 and both component of $J_x \setminus \tilde J_x$ have length bigger that $|\tilde J_x|$. Therefore 
\[
\mathcal L_{\mathscr G}^{\omega} (y) \leq e \sum_{x \in \mathcal X} \frac{|J_x|}{|J|} \leq \frac{e^2 \theta}{\theta_0} \sum_{x \in \mathcal X} \frac{|J_x|}{|\tilde B(\delta')|}.
\]
Now it suffice to show that 
\begin{equation}
\sum_{x \in \mathcal X} |J_x| \leq 4 |\tilde B(\delta)|.
\end{equation}

\begin{claim}
For each $x' \in J_x \cap \mathcal X$ with $s(x', \omega) > s(x, \omega)$, then we have $J_x \supset \tilde J_x \supset J_{x'} $.
\end{claim}

Let $s = s(x, \omega), s'=s(x', \omega), (z, \tilde \omega) = F^s(x', \omega)$. We first show that $d_*(z, c) \leq \delta'$. Argue by contradiction, assume that $d_*(z, c)>  \delta'$. Since $f_{\tilde \omega}^{s'-s}(z) = f_{\omega}^s(x') = y$, there exists a minimal integer $0 < t \leq s'-s$ such that $d_*(f_{\tilde \omega}^{t}(z), c) \leq \delta'$. Let $\delta'' \in (\delta', \delta_0]$ be such that $d_*(f_{\tilde \omega}^{j}(z), c) \geq \delta''$ for all $0 \leq j < t$. Then by Lemma 5.2, $t$ is a $\frac{\theta}{2e^2}$-good return of $(z, \tilde \omega)$ into $\tilde B(\delta'') \times \Omega_{\epsilon}$, provided $\delta_0$ is small enough. By Lemma 3.5 and the fact that $\theta$-good return time of $(x, \omega)$ into region $\tilde B(\delta_x) \times \Omega_{\epsilon}$ with $\delta_x \in [\delta', \delta_0]$, we have
\[
\frac{ D f_{\omega}^s(x')}{ A(x', \omega, s) } \geq \frac{1}{e} \frac{Df_{\omega}^s(x)}{A(x, \omega, s)} \geq \frac{|\tilde B(\delta_x)|}{e \theta} \geq d(f_{\omega}^{s}(x'), c).
\]
Hence $s$ is a $1$-close return of $(x', \omega)$. By Lemma 6.1, $t+s$ is a $\frac{\theta}{e^2}$-good return time of $(x', \omega)$ into $\tilde B(\delta'') \times \Omega_{\epsilon}$. Since $\delta'' \geq \delta$, this implies that $\hat s(x', \omega) \leq s+ t \leq s'$. Since $\hat s(x', \omega) \geq s(x', \omega) =s'$, then $\hat s(x', \omega) =s'$. A contradiction since we assume that $(x', \omega) \in E_1$. This proves $d_*(z, c) \leq \delta'$. Since $\mathcal N(f_{\omega}^s | J_{x'}) \leq 1$, we have
\[
\frac{|f_{\omega}^s (J_{x'})| }{ d(f_{\omega}^s(x'), c) } \leq \frac{e Df_{\omega}^s(x') |J_{x'}| }{d(f_{\omega}^s(x'), c)} \leq \frac{2e \theta_0}{ A(x', \omega, s')}\frac{Df_{\omega}^s(x') }{d(f_{\omega}^s(x'), c)} \leq 2 e \theta_0 < \frac{1}{3},
\] 
it follows that $f_{\omega}^s (J_{x'}) \subset J_0$. This proves the claim.

To complete the proof of (6.16), we decompose $\mathcal X$ as a disjoint union of sub-collections $\mathcal X(k), k \geq 0$ as follows: 

(i) $\mathcal X (0)$ is the subset of $\mathcal X$ consisting of those points $x$ for which $s(x, \omega) \leq s(x', \omega)$ for each $x' \in J_x \cap \mathcal X$;

(ii) for each $k \geq 1$, $\mathcal X(k)$ is the subset of $\mathcal X \setminus (\bigcup_{i=0}^{k-1} \mathcal X(i))$ consisting of those pints $x$ for which $s(x, \omega) \leq s(x', \omega)$ for each $x' \in J_x \cap (\mathcal X \setminus (\bigcup_{i=0}^{k-1} \mathcal X(i)))$.

Then by the claim above and (6.15), for each $k \geq 1$ we have
\[
\sum_{x' \in \mathcal X(k)} |J_{x'}| \leq \sum_{x \in \mathcal X(k-1)} |\tilde J_x| \leq \frac{1}{2} \sum_{x \in \mathcal X(k-1)} |J_x|.
\]
Since each $J_x, x \in \mathcal X$, has length less that $|\tilde B(\delta)|$, then 
\begin{align*}
\sum_{x \in \mathcal X} |J_x| & = \sum_{k\geq 0} \sum_{x \in \mathcal X(k)} |J_x| = \sum_{x \in \mathcal X(0)} |J_x| + \sum_{k\geq 1} \sum_{x \in \mathcal X(k)} |J_x| \\
& \leq  \sum_{x \in \mathcal X(0)} |J_x| + \frac{1}{2} \sum_{k\geq 0} \sum_{x \in \mathcal X(k)} |J_x| \leq 2  \sum_{x \in \mathcal X(0)} |J_x|  \leq 4 |\tilde B(\delta)|.
\end{align*}

\end{proof}

We conclude this section with the proof of Proposition 18.

\begin{proof}[Proof of Proposition 18]

Let $S = S_p^{\theta}(\delta, \epsilon; \delta_0), \hat S = \hat S_p^{\theta/e} (\delta, \epsilon; \delta_0)$ and for each $n \geq 0$,
\[
\hat K_n = \frac{K_n^p}{|\tilde B(\delta)|}.
\]
We shall prove by induction that
\begin{equation}
\hat K_n \leq (36 \theta/\theta_0)^n (S+ 2 \hat S). 
\end{equation}

For $n=0$, by definition, 
\begin{align*}
K_0^p &: = \iint_{E_0} \varphi(x, \omega)^p d P_{\epsilon} = \iint_{\tilde B(\delta) \times \Omega_{\epsilon}} s(x, \omega)^p dP_{\epsilon} + \iint_{( \tilde B(e \delta) \setminus \tilde B(\delta)) \times \Omega_{\epsilon}} \hat s(x, \omega)^p dP_{\epsilon} \\
& \leq \frac{|\tilde B(\delta)|}{|\tilde B(\delta)|} \iint_{\tilde B(\delta) \times \Omega_{\epsilon}} s(x, \omega)^p dP_{\epsilon} + \frac{2|\tilde B(2\delta)|}{|\tilde B(e\delta)|}  \iint_{( \tilde B(e \delta) \setminus \tilde B(\delta)) \times \Omega_{\epsilon}} \hat s(x, \omega)^p dP_{\epsilon} \\
& \leq |\tilde B(\delta)| S_p^{\theta} (\delta, \epsilon; \delta_0) + |\tilde B(2\delta)|  \iint_{( \tilde B(e \delta) \setminus \tilde B(\delta)) \times \Omega_{\epsilon}} \frac{1}{d(x, c)} \hat s(x, \omega)^p dP_{\epsilon}\\
& \leq |\tilde B(\delta)| S + |\tilde B(2 \delta)| \hat S,
\end{align*}
provided $\delta_0$ is small enough.

For $n=1$. Since $\mathscr G$ is future-free, applying Lemma 6.2 and 6.5 to $\mathscr G$ and $\phi = \varphi^p$, we have
\begin{align*}
K_1^p & = \iint_{\tilde B(\delta_0) \times \Omega_{\epsilon}} \mathcal L_{\mathscr G}^{\nu_{\epsilon}} (y) (\varphi(y, \omega))^p dy d \nu_{\epsilon}^{\mathbb N} \\
& =  \iint_{\tilde B(\delta) \times \Omega_{\epsilon}} \mathcal L_{\mathscr G}^{\nu_{\epsilon}} (y) (\varphi(y, \omega))^p dy d \nu_{\epsilon}^{\mathbb N} +  \iint_{ (\tilde B(\delta_0) \setminus \tilde B(\delta)) \times \Omega_{\epsilon}} \mathcal L_{\mathscr G}^{\nu_{\epsilon}} (y) (\varphi(y, \omega))^p dy d \nu_{\epsilon}^{\mathbb N} \\
& \leq \frac{36\theta}{\theta_0} |\tilde B(\delta)| \left( \frac{1}{|\tilde B(\delta)|} \iint_{\tilde B(\delta) \times \Omega_{\epsilon}}  s(y, \omega)^p dP_{\epsilon} + \iint_{ (\tilde B(\delta_0) \setminus \tilde B(\delta)) \times \Omega_{\epsilon}} \frac{1}{d(y, c)} \hat s(y, \omega)^p  d P_{\epsilon}\right) \\
& \leq \frac{36 \theta}{\theta_0} |\tilde B(\delta)| (S +  \hat S) \leq \frac{36 \theta}{\theta_0} (S +  \hat S).
\end{align*}
So (6.17) holds for $n =1$. Similarly for each $n \geq 1$, applying Lemma 6.2 and 6.5 to $\mathscr G$ and $\phi = \varphi_n^p$, we have
\[
K_{n+1}^p = \iint_{E_n} \mathcal L_{\mathscr G}^{\nu_{\epsilon}} (y) (\varphi_n(y, \omega))^p d P_{\epsilon} \leq \frac{36 \theta}{\theta_0} K_n^p.
\]
So (6.17) holds by induction.

By Lemma 6.4, (6.17) and the choice of $\theta$, we have 
\begin{align*}
S_p^{\theta} (e \delta, \epsilon; \delta_0) & \leq \frac{1}{|\tilde B(e \delta)|} \left( \sum_{n=0}^{\infty} K_n \right)^p \leq  \frac{1}{|\tilde B(e \delta)|} \left( |\tilde B(\delta)|^{\frac{1}{p}} \sum_{n=0}^{\infty} {\hat K_n}^{\frac{1}{p}} \right)^p\\
& \leq \frac{|\tilde B(\delta)|}{|\tilde B(e \delta)|} \left(  \sum_{n=0}^{\infty} \left( \frac{36 \theta}{\theta_0} \right)^{\frac{n}{p}}  \right)^p (S + 2 \hat S)\\
& \leq e^{-\frac{1}{\ell}} \left(1 - \left( \frac{36 \theta}{\theta_0} \right)^{\frac{1}{p}} \right)^{-p} (S + 2 \hat S)\\
& < \lambda (S + 2 \hat S).
\end{align*}
The proposition follows.
\end{proof}

\subsection*{Conflict of interest}

The authors declared no potential conﬂicts of interest with respect to the research.

\subsection*{Data availability statement}

No datasets were generated or analysed during the current study.


%
%



{\textsc{ School of Mathematics and Statistics,
Zhengzhou University, Zhengzhou,
450001,  CHINA}} (e-mail:jihymath@zzu.edu.cn)

\end{document}